\documentclass[a4paper,12pt]{article}

\usepackage{amsmath,amssymb,amsthm, amscd}

\newcommand{\Fg}{\mathfrak{g}}
\newcommand{\Fh}{\mathfrak{h}}

\newcommand{\CB}{\mathcal{B}}
\newcommand{\CZ}{\mathcal{Z}}
\newcommand{\CG}{\mathcal{G}}
\newcommand{\CK}{\mathcal{K}}
\newcommand{\CO}{\mathcal{O}}
\newcommand{\CX}{\mathcal{X}}
\newcommand{\CY}{\mathcal{Y}}
\newcommand{\CP}{\mathcal{P}}

\newcommand{\BC}{\mathbb{C}}
\newcommand{\BR}{\mathbb{R}}

\newcommand{\BZ}{\mathbb{Z}}
\newcommand{\BB}{\mathbb{B}}
\newcommand{\BS}{\mathbb{S}}

\newcommand{\mv}{\mathcal{MV}}

\newcommand{\mul}{\mathsf{m}}
\newcommand{\tr}{\mathsf{tr}}
\newcommand{\cl}{\mathsf{cl}}

\newcommand{\IC}{IC}
\newcommand{\bR}{R}

\newcommand{\str}[3]{\star^{#1}_{#2,\,#3}}
\newcommand{\stra}[1]{\star^{#1}_{\mu^{(1)}_{#1},\,\mu^{(2)}_{#1}}}

\newcommand{\Hom}{\mathop{\rm Hom}\nolimits}
\newcommand{\wt}{\mathop{\rm wt}\nolimits}
\newcommand{\Conv}{\mathop{\rm Conv}\nolimits}
\newcommand{\Ad}{\mathop{\rm Ad}\nolimits}
\newcommand{\Rep}{\mathop{\rm Rep}\nolimits}
\newcommand{\Irr}{\mathop{\rm Irr}\nolimits}

\newcommand{\Gr}{\CG r}

\newcommand{\pair}[2]{\langle #1,\,#2 \rangle}

\newcommand{\ti}[1]{\widetilde{#1}}
\newcommand{\ol}[1]{\overline{#1}}
\newcommand{\ud}[1]{\underline{#1}}

\newcommand{\disp}[1]{{\displaystyle #1}}

\newcommand{\ve}{\varepsilon}
\newcommand{\vp}{\varphi}

\newcommand{\bzero}{\mathbf{0}}

\newcommand{\bb}{\mathbf{b}}
\newcommand{\bi}{\mathbf{i}}
\newcommand{\bj}{\mathbf{j}}

\newcommand{\wi}[1]{w^{\bi}_{#1}}
\newcommand{\wj}[1]{w^{\bj}_{#1}}
\newcommand{\vi}[1]{v^{\bi}_{#1}}
\newcommand{\yi}[1]{y^{\bi}_{#1}}
\newcommand{\si}[1]{s_{i_{#1}}}
\newcommand{\Ni}[1]{n^{\bi}_{#1}}
\newcommand{\Nj}[1]{n^{\bj}_{#1}}
\newcommand{\bti}[1]{\beta^{\bi}_{#1}}

\newcommand{\twp}{\,\ti{\times}\,}

\newcommand{\bqed}{\quad \hbox{\rule[-0.5pt]{3pt}{8pt}}}

\makeatletter
\@addtoreset{equation}{subsection}

\renewcommand\section{\@startsection{section}{1}{0pt}
{-3.5ex plus -1ex minus -.2ex}{1.0ex plus .2ex}{\large\bf}}
\renewcommand\subsection{\@startsection{subsection}{1}{0pt}
{2.5ex plus 1ex minus .2ex}{-1em}{\bf}}
\makeatother

\newcommand{\vsp}{\vspace{3mm}}

\theoremstyle{plain}
\newtheorem{thm}{Theorem}[subsection]
\newtheorem{lem}[thm]{Lemma}
\newtheorem{prop}[thm]{Proposition}
\newtheorem{cor}[thm]{Corollary}

\newtheorem*{claim*}{Claim}

\newtheorem{ithm}{Theorem}

\theoremstyle{definition}
\newtheorem{dfn}[thm]{Definition}

\newtheorem*{question}{Question}

\theoremstyle{remark}
\newtheorem{rem}[thm]{Remark}
\newtheorem{ex}[thm]{Example}

\newtheorem*{irem}{Remark}

\begin{document}

\setlength{\baselineskip}{18pt}

\title{\Large\bf 
Polytopal Estimate of Mirkovi\'c-Vilonen polytopes \\
lying in a Demazure crystal}
\author{
 Syu Kato%
 \footnote{Supported in part by JSPS Research Fellowships 
           for Young Scientists (No.\,20740011).} \\
 \small Research Institute for Mathematical Sciences, Kyoto University, \\
 \small Kitashirakawa-Oiwake-cho, Sakyo, Kyoto 606-8502, Japan \\
 \small (e-mail: {\tt kato@kurims.kyoto-u.ac.jp}) \\[5mm]
 Satoshi Naito%
 \footnote{Supported in part by Grant-in-Aid for Scientific Research 
           (No.\,20540006), JSPS.} \\ 
 \small Institute of Mathematics, University of Tsukuba, \\
 \small Tsukuba, Ibaraki 305-8571, Japan \ 
 (e-mail: {\tt naito@math.tsukuba.ac.jp})
 \\[2mm] and \\[2mm]
 Daisuke Sagaki%
 \footnote{Supported in part by JSPS Research Fellowships 
           for Young Scientists (No.\,19740004).} \\ 
 \small Institute of Mathematics, University of Tsukuba, \\
 \small Tsukuba, Ibaraki 305-8571, Japan \ 
 (e-mail: {\tt sagaki@math.tsukuba.ac.jp})
}
\date{}
\maketitle

%
\begin{abstract} \setlength{\baselineskip}{16pt}
In this paper, we give a polytopal estimate of 
Mirkovi\'c-Vilonen polytopes lying in a Demazure crystal 
in terms of Minkowski sums of extremal Mirkovi\'c-Vilonen polytopes.
As an immediate consequence of this result, 
we provide a necessary (but not sufficient) 
polytopal condition for a Mirkovi\'c-Vilonen polytope 
to lie in a Demazure crystal.
\end{abstract}
%
%
\section{Introduction.}
\label{sec:intro}
This paper is a continuation of our previous one \cite{NS-dp}, and 
our purpose is to give a polytopal estimate of 
Mirkovi\'c-Vilonen polytopes lying in a Demazure crystal 
in terms of Minkowski sums of extremal Mirkovi\'c-Vilonen polytopes.
It should be mentioned that 
as an immediate consequence of this result, 
we can provide an affirmative answer to 
a question posed in \cite[\S4.6]{NS-dp}.

Following the notation and terminology of \cite{NS-dp}, 
we now explain our results more precisely.
Let $G$ be a complex, connected, semisimple algebraic group 
with Lie algebra $\Fg$, $T$ a maximal torus 
with Lie algebra (Cartan subalgebra) $\Fh$, 
$B$ a Borel subgroup containing $T$, 
and $U$ the unipotent radical of $B$;  
by our convention, the roots in $B$ are the negative ones. 
Let $X_{*}(T)$ denote the coweight lattice $\Hom(\BC^{*},\,T)$ 
for $G$, which we regard as an additive subgroup of 
a real form $\Fh_{\BR}:=\BR \otimes_{\BZ} X_{*}(T)$ of $\Fh$. 
Denote by $W$ the Weyl group of $\Fg$, 
with $e$ the identity element and $w_{0}$ 
the longest element of length $m$.
Also, let $\Fg^{\vee}$ denote
the (Langlands) dual Lie algebra of $\Fg$ 
with Weyl group $W$, and let $U_{q}(\Fg^{\vee})$ 
be the quantized universal enveloping algebra 
of $\Fg^{\vee}$ over $\BC(q)$. 

For each dominant coweight 
$\lambda \in X_{\ast}(T) \subset \Fh_{\BR}$, 
let us denote by $\mv(\lambda)$ 
the set of Mirkovi\'c-Vilonen (MV for short) polytopes 
with highest vertex $\lambda$ that is contained in 
the convex hull $\Conv (W \cdot \lambda)$ in $\Fh_{\BR}$
of the Weyl group orbit $W \cdot \lambda$ through $\lambda$, 
and by $\CB(\lambda)$ the crystal basis of 
the irreducible highest weight 
$U_{q}(\Fg^{\vee})$-module $V(\lambda)$ 
of highest weight $\lambda$.
Recall that Kamnitzer \cite{Kam1}, \cite{Kam2} 
proved the existence of an isomorphism of crystals 
$\Psi_{\lambda}$ from the crystal basis $\CB(\lambda)$ 
to the set $\mv(\lambda)$ of MV polytopes, 
which is endowed with the Lusztig-Berenstein-Zelevinsky 
(LBZ for short) crystal structure; he also 
proved the coincidence of this LBZ crystal structure on 
$\mv(\lambda)$ with the Braverman-Finkelberg-Gaitsgory 
(BFG for short) crystal structure on $\mv(\lambda)$.

In \cite{NS-dp}, for each $x \in W$, 
we gave a combinatorial description, 
in terms of the lengths of edges of an MV polytope, 
of the image $\mv_{x}(\lambda) \subset \mv(\lambda)$ 
(resp., $\mv^{x}(\lambda) \subset \mv(\lambda))$ of 
the Demazure crystal $\CB_{x}(\lambda) \subset \CB(\lambda)$ 
(resp., opposite Demazure crystal 
$\CB^{x}(\lambda) \subset \CB(\lambda)$) 
under the isomorphism 
$\Psi_{\lambda} : \CB(\lambda) \rightarrow \mv(\lambda)$ of crystals.
Furthermore, in \cite{NS-dp}, 
we proved that for each $x \in W$, 
an MV polytope $P \in \mv(\lambda)$ lies in 
the opposite Demazure crystal 
$\mv^{x}(\lambda)$ if and only if 
the MV polytope $P$ contains (as a set) 
the extremal MV polytope 
$P_{x \cdot \lambda}$ of weight $x \cdot \lambda$, 
which is identical to the convex hull 
$\Conv(W_{\le x} \cdot \lambda)$ in $\Fh_{\BR}$
of a certain subset $W_{\le x} \cdot \lambda$ of $W \cdot \lambda$ 
(see \S\ref{subsec:extpoly} for details). 
However, we were unable to prove 
an analogous statement for Demazure crystals 
$\CB_{x}(\lambda)$, $x \in W$.
Thus, we posed the following question in 
\cite[\S4.6]{NS-dp}:

\begin{question}
Let us take an arbitrary $x \in W$. 
Are all the MV polytopes lying 
in the Demazure crystal $\mv_{x}(\lambda)$ 
contained (as sets) in the extremal MV polytope 
$P_{x \cdot \lambda} = \Conv(W_{\le x} \cdot \lambda)$ ?
\end{question}

\noindent
Note that the converse statement fails to hold, 
as mentioned in \cite[Remark~4.6.1]{NS-dp}.

In this paper, we provide an affirmative answer to this question. 
In fact, we considerably sharpen the polytopal estimate above of 
MV polytopes lying in a Demazure crystal as follows. 
In what follows, for each dominant coweight 
$\lambda \in X_{\ast}(T) \subset \Fh_{\BR}$, 
we denote by $W_{\lambda} \subset W$ 
the stabilizer of $\lambda$ in $W$, and 
by $W^{\lambda}_{\min} \subset W$ 
the set of minimal (length) coset representatives modulo 
the subgroup $W_{\lambda} \subset W$.
%
%
\begin{ithm}[$=$ Theorem~\ref{thm:main} combined 
with Proposition~\ref{prop:N3}] \label{ithm1}
Let $\lambda \in X_{\ast}(T) \subset \Fh_{\BR}$ 
be a dominant coweight, and 
let $x \in W^{\lambda}_{\min} \subset W$. 
If an MV polytope $P \in \mv(\lambda)$ lies 
in the Demazure crystal $\mv_{x}(\lambda)$, 
then there exist a positive integer $N \in \BZ_{\ge 1}$ 
and minimal coset representatives 
$x_{1},\,x_{2},\,\dots,\,x_{N} \in 
 W^{\lambda}_{\min} \subset W$ 
such that 
\begin{equation*}
\begin{cases}
x \ge x_{k} \quad 
 \text{\rm for all $1 \le k \le N$;} \\[3mm]
N \cdot P \subseteq
P_{x_1 \cdot \lambda} + 
P_{x_2 \cdot \lambda} + \cdots + 
P_{x_N \cdot \lambda},
\end{cases}
\end{equation*}
where $N \cdot P := 
\bigl\{N v \mid v \in P \bigr\} \subset \Fh_{\BR}$ 
is an MV polytope in $\mv(N \lambda)$, and 
$P_{x_1 \cdot \lambda} + P_{x_2 \cdot \lambda} + \cdots + 
 P_{x_N \cdot \lambda}$ is the Minkowski sum of 
the extremal MV polytopes $P_{x_1 \cdot \lambda}, \,
P_{x_2 \cdot \lambda},\,\dots,\,P_{x_N \cdot \lambda}$.
\end{ithm}

\begin{irem}
We see from Remark~\ref{rem:MV-LS} and 
Theorem~\ref{thm:main} below that 
the elements $x_{1}$, $x_{2}$, $\dots$, $x_{N} \in 
W^{\lambda}_{\min} \subset W$ can be chosen 
in such a way that the vectors 
$x_1 \cdot \lambda$, $x_2 \cdot \lambda$, $\dots$, 
$x_N \cdot \lambda \in W \cdot \lambda \subset \Fh_{\BR}$
give the directions of the Lakshmibai-Seshadri path of shape $\lambda$ 
that corresponds to the MV polytope $P \in \mv(\lambda)$ 
under the (inexplicit) bijection via the crystal basis $\CB(\lambda)$.
Hence it also follows that 
$x_1 \ge x_2 \ge \dots \ge x_N$ 
in the Bruhat ordering on $W$.
\end{irem}

From the theorem above, we can deduce immediately that 
for an arbitrary $P \in \mv_{x}(\lambda)$, 
there holds $N \cdot P \subset N \cdot P_{x \cdot \lambda}$ and 
hence $P \subset P_{x \cdot \lambda}$. Indeed, this follows from 
the inclusion 
$P_{x_k \cdot \lambda}=\Conv(W_{\le x_k} \cdot \lambda) 
 \subset \Conv(W_{\le x} \cdot \lambda)=P_{x \cdot \lambda}$
for each $1 \le k \le N$, and the fact that the Minkowski sum 
$P_{x \cdot \lambda}+P_{x \cdot \lambda}+\cdots+P_{x \cdot \lambda}$ 
($N$ times) is identical to $N \cdot P_{x \cdot \lambda}$ 
(see Remark~\ref{rem:N1}). 

The main ingredient in our proof of the theorem is 
the following polytopal estimate of 
tensor products of MV polytopes.
Let $\lambda_{1},\,\lambda_{2} \in X_{\ast}(T) \subset \Fh_{\BR}$ 
be dominant coweights. 
Since $\mv(\lambda) \cong \CB(\lambda)$ as crystals 
for every dominant coweight $\lambda \in X_{\ast}(T)$, 
the tensor product 
$\mv(\lambda_{2}) \otimes \mv(\lambda_{1})$ of 
the crystals $\mv(\lambda_{1})$ and $\mv(\lambda_{2})$ 
decomposes into a disjoint union of 
connected components as follows:
\begin{equation*}
\mv(\lambda_{2}) \otimes \mv(\lambda_{1}) \cong 
 \bigoplus_{
   \begin{subarray}{c}
   \lambda \in X_{\ast}(T) \\[0.5mm]
   \text{$\lambda$ : dominant}
   \end{subarray}
   } 
\mv(\lambda)^{\oplus m_{\lambda_{1},\lambda_{2}}^{\lambda}},
\end{equation*}
where $m_{\lambda_{1},\lambda_{2}}^{\lambda} \in \BZ_{\ge 0}$ denotes 
the multiplicity of $\mv(\lambda)$ in 
$\mv(\lambda_{2}) \otimes \mv(\lambda_{1})$. 
For each dominant coweight 
$\lambda \in X_{\ast}(T)$ 
such that $m_{\lambda_{1},\lambda_{2}}^{\lambda} \ge 1$, 
we take (and fix) an arbitrary embedding 
$\iota_{\lambda}:\mv(\lambda) \hookrightarrow 
 \mv(\lambda_{2}) \otimes \mv(\lambda_{1})$ 
of crystals that maps $\mv(\lambda)$ 
onto a connected component of 
$\mv(\lambda_{2}) \otimes \mv(\lambda_{1})$, 
which is isomorphic to $\mv(\lambda)$ as a crystal. 
%
%
\begin{ithm}[$=$ Theorem~\ref{thm:tensor}] \label{ithm2}
Keep the notation above.
Let $P \in \mv(\lambda)$, and 
write $\iota_{\lambda}(P) \in 
\mv(\lambda_{2}) \otimes \mv(\lambda_{1})$ as\,{\rm:} 
$\iota_{\lambda}(P)=P_{2} \otimes P_{1}$ for some 
$P_{1} \in \mv(\lambda_{1})$ and 
$P_{2} \in \mv(\lambda_{2})$. 
We assume that the MV polytope 
$P_{2} \in \mv(\lambda_{2})$ is an extremal MV polytope 
$P_{x \cdot \lambda_{2}}$ for some $x \in W$. 
Then, we have
\begin{equation*}
P \subset P_{1} + P_{2},
\end{equation*}
where $P_{1} + P_{2}$ is the Minkowski sum of 
the MV polytopes $P_{1} \in \mv(\lambda_{1})$ and 
$P_{2} \in \mv(\lambda_{2})$.
\end{ithm}


We have not yet found a purely combinatorial proof 
of the theorem above. In fact, our argument is a geometric one, 
which is based on results of Braverman-Gaitsgory in \cite{BrGa}, 
where tensor products of highest weight crystals are 
described in terms of MV cycles in the affine Grassmannian; 
here we should remark that the convention on the tensor product
rule for crystals in \cite{BrGa} is opposite to ours, i.e., to 
that of Kashiwara \cite{Kasoc}, \cite{Kasb}. 
Also, it seems likely that the theorem above still holds 
without the assumption of extremality on 
the MV polytope $P_{2} \in \mv(\lambda_{2})$.

This paper is organized as follows.
In \S\ref{sec:MVda}, we first recall 
the basic notation and terminology concerning MV polytopes and 
Demazure crystals, and also review the relation between 
MV polytopes and MV cycles in the affine Grassmannian. 
Furthermore, we obtain 
a few new results on extremal MV polytopes and MV cycles, 
which will be used in the proof of Theorem~\ref{ithm2} 
($=$ Theorem~\ref{thm:tensor}).
In \S\ref{sec:multi}, we introduce 
the notion of $N$-multiple maps 
from $\mv(\lambda)$ to $\mv(N \lambda)$ 
for a dominant coweight $\lambda \in X_{\ast}(T)$ and 
$N \in \BZ_{\geq 1}$, which is given explicitly 
by: $P \mapsto N \cdot P$ in terms of MV polytopes, 
and also show that for each MV polytope 
$P \in \mv(\lambda)$, there exists some $N \in \BZ_{\geq 1}$ 
such that $N \cdot P \in \mv(N \lambda)$ 
can be written as the tensor product of 
certain $N$ extremal MV polytopes in $\mv(\lambda)$.
Furthermore, assuming Theorem~\ref{ithm2}, 
we prove Theorem~\ref{ithm1} above, 
which provides an answer to 
the question mentioned above.
In \S\ref{sec:tensor}, after revisiting results of 
Braverman-Gaitsgory on tensor products of 
highest weight crystals
in order to adapt them to our situation, 
we prove Theorem~\ref{ithm2} by using 
the geometry of the affine Grassmannian.
In the Appendix, we give a brief account of 
why Theorem~\ref{thm:BG} below is 
a reformulation of results of Braverman-Gaitsgory. 

%
\section{Mirkovi\'c-Vilonen polytopes and Demazure crystals.}
\label{sec:MVda}

%
\subsection{Basic notation.}
\label{subsec:notation}

Let $G$ be a complex, connected, reductive algebraic group, 
$T$ a maximal torus, $B$ a Borel subgroup containing $T$, and 
$U$ the unipotent radical of $B$; 
we choose the convention that 
the roots in $B$ are the negative ones. 
Let $X_{*}(T)$ denote the (integral) coweight lattice 
$\Hom(\BC^{*},\,T)$ for $G$, and $X_{*}(T)_{+}$ 
the set of dominant (integral) coweights for $G$; 
we regard the coweight lattice 
$X_{*}(T)$ as an additive subgroup of 
a real form $\Fh_{\BR}:=\BR \otimes_{\BZ} X_{*}(T)$ 
of the Lie algebra $\Fh$ of the maximal torus $T$. 
We denote by $G^{\vee}$ 
the (complex) Langlands dual group of $G$. 

For the rest of this paper, 
except in \S\ref{subsec:geom}, \S\ref{subsec:BG}, and 
the Appendix, we assume that $G$ is semisimple. 
Denote by $\Fg$ the Lie algebra of $G$, 
which is a complex semisimple Lie algebra. 
Let 
\begin{equation*}
\Bigl(A=(a_{ij})_{i,j \in I}, \, 
 \Pi:=\bigl\{\alpha_{j}\bigr\}_{j \in I}, \, 
 \Pi^{\vee}:=\bigl\{h_{j}\bigr\}_{j \in I}, \, 
 \Fh^{\ast},\,\Fh
 \Bigr)
\end{equation*}
be the root datum of $\Fg$, where 
$A=(a_{ij})_{i,j \in I}$ is the Cartan matrix, 
$\Fh$ is the Cartan subalgebra, 
$\Pi:=\bigl\{\alpha_{j}\bigr\}_{j \in I} \subset 
 \Fh^{\ast}:=\Hom_{\BC}(\Fh,\,\BC)$ 
is the set of simple roots, and 
$\Pi^{\vee}:=\bigl\{h_{j}\bigr\}_{j \in I} \subset \Fh$ 
is the set of simple coroots; note that 
$\pair{h_{i}}{\alpha_{j}}=a_{ij}$ for $i,\,j \in I$, 
where $\pair{\cdot}{\cdot}$ denotes the canonical pairing 
between $\Fh$ and $\Fh^{\ast}$.
We set $\Fh_{\BR}:= \sum_{j \in I} \BR h_{j} \subset \Fh$, 
which is a real form of $\Fh$; 
we regard the coweight lattice 
$X_{*}(T)=\Hom(\BC^{\ast},\,T)$ 
as an additive subgroup of $\Fh_{\BR}$. 
Also, for $h,\,h' \in \Fh_{\BR}$, we write 
$h' \ge h$ if $h'-h \in Q^{\vee}_{+}:=\sum_{j \in I}\BZ_{\ge 0}h_{j}$. 
Let $W:=\langle s_{j} \mid j \in I \rangle$ 
be the Weyl group of $\Fg$, where $s_{j}$, $j \in I$, are 
the simple reflections, with length function 
$\ell:W \rightarrow \BZ_{\ge 0}$, 
the identity element $e \in W$, and 
the longest element $w_{0} \in W$; 
we denote by $\le$ the (strong) Bruhat ordering on $W$. 
Let $\Fg^{\vee}$ be the Lie algebra of 
the Langlands dual group $G^{\vee}$ of $G$, 
which is the complex semisimple Lie algebra
associated to the root datum 
\begin{equation*}
\Bigl({}^{t}A=(a_{ji})_{i,j \in I}, \, 
 \Pi^{\vee}=\bigl\{h_{j}\bigr\}_{j \in I}, \, 
 \Pi=\bigl\{\alpha_{j}\bigr\}_{j \in I}, \, 
 \Fh,\,\Fh^{\ast}
 \Bigr);
\end{equation*}
note that the Cartan subalgebra of 
$\Fg^{\vee}$ is $\Fh^{\ast}$, not $\Fh$. 
Let $U_{q}(\Fg^{\vee})$ be 
the quantized universal enveloping algebra of 
$\Fg^{\vee}$ over $\BC(q)$, and $U_{q}^{+}(\Fg^{\vee})$ 
its positive part. 
For a dominant coweight 
$\lambda \in X_{\ast}(T) \subset \Fh_{\BR}$, 
denote by $V(\lambda)$ the integrable highest 
weight $U_{q}(\Fg^{\vee})$-module 
of highest weight $\lambda$, and by 
$\CB(\lambda)$ the crystal basis of $V(\lambda)$. 

Now, let $\lambda \in X_{\ast}(T) \subset \Fh_{\BR}$ 
be a dominant coweight, and $x \in W$. 
The Demazure module $V_{x}(\lambda)$ is defined to be 
the $U_{q}^{+}(\Fg^{\vee})$-submodule of $V(\lambda)$ 
generated by the one-dimensional weight space 
$V(\lambda)_{x \cdot \lambda}  \subset V(\lambda)$ of 
weight $x \cdot \lambda \in X_{\ast}(T) \subset \Fh_{\BR}$. 
Recall from \cite{Kas} that 
the Demazure crystal $\CB_{x}(\lambda)$ is 
a subset of $\CB(\lambda)$ such that 
%
%
\begin{equation} \label{eq:dem-mod}
V(\lambda) \supset 
V_{x}(\lambda) = 
 \bigoplus_{b \in \CB_{x}(\lambda)}
 \BC(q) G_{\lambda}(b),
\end{equation}
where $G_{\lambda}(b)$, $b \in \CB(\lambda)$, 
form the lower global basis of $V(\lambda)$.
%
%
\begin{rem} \label{rem:demcos}
If $x,\,y \in W$ satisfies 
$x \cdot \lambda=y \cdot \lambda$, then 
we have $V_{x}(\lambda)=V_{y}(\lambda)$ since 
$V(\lambda)_{x \cdot \lambda}=
 V(\lambda)_{y \cdot \lambda}$. 
Therefore, it follows from \eqref{eq:dem-mod} that 
$\CB_{x}(\lambda)=\CB_{y}(\lambda)$. 
\end{rem}

We know from \cite[Proposition~3.2.3]{Kas} that 
the Demazure crystals $\CB_{x}(\lambda)$, 
$x \in W$, are characterized by the inductive relations: 
%
%
\begin{align}
& \CB_{e}(\lambda)=
   \bigl\{u_{\lambda}\bigr\}, \label{eq:demint1} \\[1.5mm]
& \CB_{x}(\lambda)=
    \bigcup_{k \in \BZ_{\ge 0}} f_{j}^{k} 
     \CB_{s_{j}x}(\lambda) 
     \setminus \{0\}
  \quad \text{for $x \in W$ and $j \in I$ with $s_{j}x < x$},
  \label{eq:demint2}
\end{align}
where $u_{\lambda} \in \CB(\lambda)$ denotes 
the highest weight element of $\CB(\lambda)$, and 
$f_{j}$, $j \in I$, denote the lowering Kashiwara operators
for $\CB(\lambda)$. 

%
\subsection{Mirkovi\'c-Vilonen polytopes.}
\label{subsec:MV}

In this subsection, following \cite{Kam1}, we recall 
a (combinatorial) characterization of 
Mirkovi\'c-Vilonen (MV for short) polytopes; 
the relation between this characterization and 
the original (geometric) definition of MV polytopes 
given by Anderson \cite{A} will be explained 
in \S\ref{subsec:geom}. 

As in (the second paragraph of) \S\ref{subsec:notation}, 
we assume that $\Fg$ is a complex semisimple Lie algebra. 
Let $\mu_{\bullet}=(\mu_{w})_{w \in W}$ be 
a collection of elements of 
$\Fh_{\BR}=\sum_{j \in I} \BR h_{j}$.
We call $\mu_{\bullet}=(\mu_{w})_{w \in W}$ a 
Gelfand-Goresky-MacPherson-Serganova 
(GGMS) datum if it satisfies 
the condition that 
$w^{-1} \cdot \mu_{w'} - w^{-1} \cdot \mu_{w} 
 \in Q^{\vee}_{+}$ for all $w,\,w' \in W$. 
It follows by induction 
with respect to the (weak) Bruhat ordering on $W$ 
that $\mu_{\bullet}=(\mu_{w})_{w \in W}$ is a GGMS datum 
if and only if 
%
%
\begin{equation} \label{eq:length}
\mu_{ws_{i}}-\mu_{w} \in \BZ_{\ge 0}\,(w \cdot h_{i}) 
\quad 
\text{for every $w \in W$ and $i \in I$}. 
\end{equation}
%
%
\begin{rem} \label{rem:GGMS-sum}
Let $\mu_{\bullet}^{(1)}=(\mu_{w}^{(1)})_{w \in W}$ and 
$\mu_{\bullet}^{(2)}=(\mu_{w}^{(2)})_{w \in W}$ be GGMS data. 
Then, it is obvious from the definition of GGMS data 
(or equivalently, from \eqref{eq:length}) that 
the (componentwise) sum 
\begin{equation*}
\mu_{\bullet}^{(1)}+\mu_{\bullet}^{(2)}:=
 (\mu_{w}^{(1)}+\mu_{w}^{(2)})_{w \in W}
\end{equation*}
of $\mu_{\bullet}^{(1)}$ and $\mu_{\bullet}^{(2)}$ is 
also a GGMS datum. 
\end{rem}

Following \cite{Kam1} and \cite{Kam2}, 
to each GGMS datum $\mu_{\bullet}=(\mu_{w})_{w \in W}$, 
we associate a convex polytope 
$P(\mu_{\bullet}) \subset \Fh_{\BR}$ by:
%
%
\begin{equation} \label{eq:poly}
P(\mu_{\bullet})=
\bigcap_{w \in W}
 \bigl\{ 
 v \in \Fh_{\BR} \mid 
 w^{-1} \cdot v- w^{-1} \cdot \mu_{w} \in 
 \textstyle{\sum_{j \in I}\BR_{\ge 0} h_{j}}
 \bigr\}; 
\end{equation}
the polytope $P(\mu_{\bullet})$ is called 
a pseudo-Weyl polytope with GGMS datum $\mu_{\bullet}$. 
Note that the GGMS datum $\mu_{\bullet}=(\mu_{w})_{w \in W}$ is 
determined uniquely by the convex polytope $P(\mu_{\bullet})$. 
Also, we know from \cite[Proposition~2.2]{Kam1} that 
the set of vertices of the polytope $P(\mu_{\bullet})$ 
is given by the collection $\mu_{\bullet}=(\mu_{w})_{w \in W}$ 
(possibly, with repetitions). In particular, we have
%
%
\begin{equation} \label{eq:conv}
P(\mu_{\bullet}) = 
\Conv\,\bigl\{\mu_{w} \mid w \in W\bigr\},
\end{equation}
where for a subset $X$ of $\Fh_{\BR}$, 
$\Conv X$ denotes the convex hull in $\Fh_{\BR}$ of $X$. 

We know the following proposition from \cite[Lemma~6.1]{Kam1}, 
which will be used later. 
%
%
\begin{prop} \label{prop:Minkowski}
Let $P_{1}=P(\mu_{\bullet}^{(1)})$ and 
$P_{2}=P(\mu_{\bullet}^{(2)})$ be pseudo-Weyl polytopes 
with GGMS data 
$\mu_{\bullet}^{(1)}=(\mu_{w}^{(1)})_{w \in W}$ and 
$\mu_{\bullet}^{(2)}=(\mu_{w}^{(2)})_{w \in W}$, respectively. 
Then, the Minkowski sum
\begin{equation*}
P_{1}+P_{2}:=
 \bigl\{v_{1}+v_{2} \mid v_{1} \in P_{1},\,v_{2} \in P_{2}\bigr\}
\end{equation*}
of the pseudo-Weyl polytopes $P_{1}$ and $P_{2}$ is identical to 
the pseudo-Weyl polytope $P(\mu_{\bullet}^{(1)}+\mu_{\bullet}^{(2)})$ 
with GGMS datum $\mu_{\bullet}^{(1)}+\mu_{\bullet}^{(2)}=
 (\mu_{w}^{(1)}+\mu_{w}^{(2)})_{w \in W}$ 
(see Remark~\ref{rem:GGMS-sum}). 
\end{prop}

Let $R(w_{0})$ denote the set of all reduced words for $w_{0}$, that is, 
all sequences $(i_{1},\,i_{2},\,\dots,\,i_{m})$ of elements of $I$ 
such that $s_{i_{1}}s_{i_{2}} \cdots s_{i_{m}}=w_{0}$, where 
$m$ is the length $\ell(w_{0})$ of the longest element $w_{0}$.
Let $\bi=(i_{1},\,i_{2},\,\dots,\,i_{m}) \in R(w_{0})$ 
be a reduced word for $w_{0}$. 
We set $\wi{l}:=\si{1}\si{2} \cdots \si{l} \in W$ 
for $0 \le l \le m$. For a GGMS datum 
$\mu_{\bullet}=(\mu_{w})_{w \in W}$, define integers 
(called lengths of edges) 
$\Ni{l}=\Ni{l}(\mu_{\bullet}) \in \BZ_{\ge 0}$, 
$1 \le l \le m$, via the following 
``length formula'' (see \cite[Eq.(8)]{Kam1} and 
\eqref{eq:length} above): 
%
%
\begin{equation} \label{eq:n}
\mu_{\wi{l}}-\mu_{\wi{l-1}}=\Ni{l} \wi{l-1} \cdot h_{i_{l}}.
\end{equation}

\vsp

{\small 
\hspace{55mm}
\unitlength 0.1in
\begin{picture}( 15.0000,  3.3500)( 15.5000, -7.0000)
%
\special{pn 8}%
\special{sh 0.600}%
\special{ar 1600 600 50 50  0.0000000 6.2831853}%
\put(30.0000,-7.0000){\makebox(0,0)[lt]{$\mu_{\wi{l}}=\mu_{\wi{l-1}\si{l}}$}}%
\put(23.0000,-4.5000){\makebox(0,0){$\Ni{l}$}}%
%
\special{pn 8}%
\special{sh 0.600}%
\special{ar 3000 600 50 50  0.0000000 6.2831853}%
%
\special{pn 8}%
\special{pa 1600 600}%
\special{pa 3000 600}%
\special{fp}%
\put(16.0000,-7.0000){\makebox(0,0)[lt]{$\mu_{\wi{l-1}}$}}%
\end{picture}%

}

\vsp

Now we recall a (combinatorial) characterization of 
Mirkovi\'c-Vilonen (MV) polytopes, due to Kamnitzer \cite{Kam1}. 
This result holds for an arbitrary complex semisimple Lie algebra $\Fg$, 
but we give its precise statement only in the case that 
$\Fg$ is simply-laced since we do not make use of it 
in this paper; we merely mention that 
when $\Fg$ is not simply-laced, 
there are also conditions on the lengths 
$\Ni{l}$, $1 \le l \le m$, $\bi \in R(w_{0})$, 
for the other possible values of $a_{ij}$ and $a_{ji}$
(we refer the reader to \cite[\S3]{BeZe} for explicit formulas).
%
%
\begin{dfn} \label{dfn:MV}
A GGMS datum $\mu_{\bullet}=(\mu_{w})_{w \in W}$ is 
said to be a Mirkovi\'c-Vilonen (MV) datum if 
it satisfies the following conditions: 

(1) If $\bi=(i_{1},\,i_{2},\,\dots,\,i_{m}) \in R(w_{0})$ and 
$\bj=(j_{1},\,j_{2},\,\dots,\,j_{m}) \in R(w_{0})$ are related by 
a $2$-move, that is, if there exist indices $i,\,j \in I$ 
with $a_{ij}=a_{ji}=0$ and an integer $0 \le k \le m-2$ 
such that $i_{l}=j_{l}$ for all $1 \le l \le m$ with 
$l \ne k+1,\,k+2$, and such that 
$i_{k+1}=j_{k+2}=i$, $i_{k+2}=j_{k+1}=j$, then 
there hold
\begin{equation*}
\begin{cases}
\Ni{l}=\Nj{l} \quad 
  \text{for all $1 \le l \le m$ with $l \ne k+1,\,k+2$, and} \\[1.5mm]
\Ni{k+1}=\Nj{k+2}, \quad \Ni{k+2}=\Nj{k+1}.
\end{cases}
\end{equation*}

\vspace{5mm}

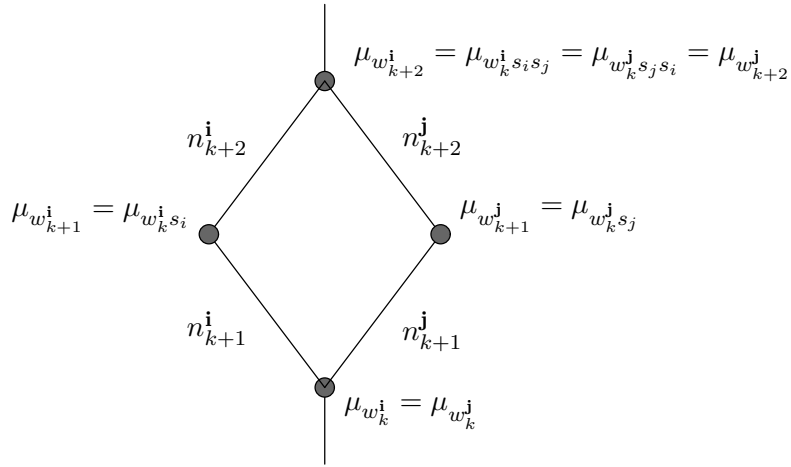
\newcommand{\vertexa}{
  $\mu_{\wi{k+2}}=\mu_{\wi{k}s_is_j}=
   \mu_{\wj{k}s_js_i}=\mu_{\wj{k+2}}$
}

{\small 
\hspace*{-20mm}
\unitlength 0.1in
\begin{picture}( 44.6000, 24.0000)(  4.4000,-28.0000)
%
\special{pn 8}%
\special{pa 4200 800}%
\special{pa 4200 400}%
\special{fp}%
%
\special{pn 8}%
\special{pa 4200 2400}%
\special{pa 4200 2800}%
\special{fp}%
%
\special{pn 8}%
\special{sh 0.600}%
\special{ar 4200 800 50 50  0.0000000 6.2831853}%
%
\special{pn 8}%
\special{sh 0.600}%
\special{ar 4200 2400 50 50  0.0000000 6.2831853}%
\put(35.0000,-16.0000){\makebox(0,0)[rb]{$\mu_{\wi{k+1}}=\mu_{\wi{k}s_{i}}$}}%
\put(43.0000,-24.5000){\makebox(0,0)[lt]{$\mu_{\wi{k}}=\mu_{\wj{k}}$}}%
%
\special{pn 8}%
\special{pa 4200 2400}%
\special{pa 4800 1600}%
\special{fp}%
\special{pa 4800 1600}%
\special{pa 4200 800}%
\special{fp}%
\special{pa 4200 800}%
\special{pa 3600 1600}%
\special{fp}%
\special{pa 3600 1600}%
\special{pa 4200 2400}%
\special{fp}%
%
\special{pn 8}%
\special{sh 0.600}%
\special{ar 4800 1600 50 50  0.0000000 6.2831853}%
%
\special{pn 8}%
\special{sh 0.600}%
\special{ar 3600 1600 50 50  0.0000000 6.2831853}%
\put(49.0000,-16.0000){\makebox(0,0)[lb]{$\mu_{\wj{k+1}}=\mu_{\wj{k}s_{j}}$}}%
\put(38.0000,-20.0000){\makebox(0,0)[rt]{$\Ni{k+1}$}}%
\put(38.0000,-12.0000){\makebox(0,0)[rb]{$\Ni{k+2}$}}%
\put(46.0000,-20.0000){\makebox(0,0)[lt]{$\Nj{k+1}$}}%
\put(46.0000,-12.0000){\makebox(0,0)[lb]{$\Nj{k+2}$}}%
\put(43.0000,-8.0000){\makebox(0,0)[lb]{\vertexa}}%
\end{picture}%

}

\vsp

(2) If $\bi=(i_{1},\,i_{2},\,\dots,\,i_{m}) \in R(w_{0})$ and 
$\bj=(j_{1},\,j_{2},\,\dots,\,j_{m}) \in R(w_{0})$ are 
related by a $3$-move, that is, 
if there exist indices $i,\,j \in I$ 
with $a_{ij}=a_{ji}=-1$ and an integer $0 \le k \le m-3$ 
such that $i_{l}=j_{l}$ for all $1 \le l \le m$ with 
$l \ne k+1,\,k+2,\,k+3$, and such that
$i_{k+1}=i_{k+3}=j_{k+2}=i$, 
$i_{k+2}=j_{k+1}=j_{k+3}=j$, then 
there hold 
\begin{equation*}
\begin{cases}
\Ni{l}=\Nj{l} \quad
  \text{for all $1 \le l \le m$ with $l \ne k+1,\,k+2,\,k+3$, and} \\[1.5mm]
\Nj{k+1}=\Ni{k+2}+\Ni{k+3}-\min \bigl(\Ni{k+1},\, \Ni{k+3}\bigr), \\[1.5mm]
\Nj{k+2}=\min \bigl(\Ni{k+1},\, \Ni{k+3}\bigr), \\[1.5mm]
\Nj{k+3}=\Ni{k+1}+\Ni{k+2}-\min \bigl(\Ni{k+1},\, \Ni{k+3}\bigr).
\end{cases}
\end{equation*}

\vspace{5mm}

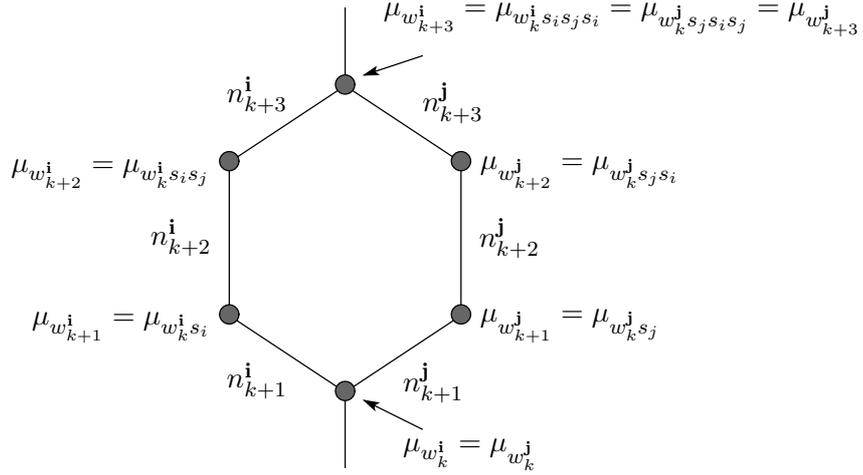
\newcommand{\vertexb}{
  $\mu_{\wi{k+3}}=\mu_{\wi{k}s_is_js_i}=
   \mu_{\wj{k}s_js_is_j}=\mu_{\wj{k+3}}$
}

{\small 
\hspace*{-22.5mm}
\unitlength 0.1in
\begin{picture}( 45.5000, 24.2000)(  1.5000,-28.0000)
%
\special{pn 8}%
\special{pa 4000 806}%
\special{pa 3400 1206}%
\special{fp}%
%
\special{pn 8}%
\special{pa 3400 1206}%
\special{pa 3400 2006}%
\special{fp}%
%
\special{pn 8}%
\special{pa 3400 2006}%
\special{pa 4000 2406}%
\special{fp}%
%
\special{pn 8}%
\special{pa 4000 2406}%
\special{pa 4600 2006}%
\special{fp}%
%
\special{pn 8}%
\special{pa 4600 2006}%
\special{pa 4600 1206}%
\special{fp}%
%
\special{pn 8}%
\special{pa 4600 1206}%
\special{pa 4000 806}%
\special{fp}%
%
\special{pn 8}%
\special{pa 4000 800}%
\special{pa 4000 400}%
\special{fp}%
%
\special{pn 8}%
\special{pa 4000 2400}%
\special{pa 4000 2800}%
\special{fp}%
%
\special{pn 8}%
\special{sh 0.600}%
\special{ar 4000 800 50 50  0.0000000 6.2831853}%
%
\special{pn 8}%
\special{sh 0.600}%
\special{ar 4600 1200 50 50  0.0000000 6.2831853}%
%
\special{pn 8}%
\special{sh 0.600}%
\special{ar 4600 2000 50 50  0.0000000 6.2831853}%
%
\special{pn 8}%
\special{sh 0.600}%
\special{ar 4000 2400 50 50  0.0000000 6.2831853}%
%
\special{pn 8}%
\special{sh 0.600}%
\special{ar 3400 2000 50 50  0.0000000 6.2831853}%
%
\special{pn 8}%
\special{sh 0.600}%
\special{ar 3400 1200 50 50  0.0000000 6.2831853}%
\put(41.5000,-5.5000){\makebox(0,0)[lb]{\vertexb}}%
\put(37.0000,-9.5000){\makebox(0,0)[rb]{$\Ni{k+3}$}}%
\put(44.0000,-10.0000){\makebox(0,0)[lb]{$\Nj{k+3}$}}%
\put(48.5000,-16.0000){\makebox(0,0){$\Nj{k+2}$}}%
\put(31.5000,-16.0000){\makebox(0,0){$\Ni{k+2}$}}%
\put(37.0000,-22.5000){\makebox(0,0)[rt]{$\Ni{k+1}$}}%
\put(43.0000,-22.5000){\makebox(0,0)[lt]{$\Nj{k+1}$}}%
\put(47.0000,-13.5000){\makebox(0,0)[lb]{$\mu_{\wj{k+2}}=\mu_{\wj{k}s_js_i}$}}%
\put(47.0000,-21.5000){\makebox(0,0)[lb]{$\mu_{\wj{k+1}}=\mu_{\wj{k}s_j}$}}%
\put(43.0000,-26.5000){\makebox(0,0)[lt]{$\mu_{\wi{k}}=\mu_{\wj{k}}$}}%
\put(33.0000,-21.5000){\makebox(0,0)[rb]{$\mu_{\wi{k+1}}=\mu_{\wi{k}s_i}$}}%
\put(33.0000,-13.5000){\makebox(0,0)[rb]{$\mu_{\wi{k+2}}=\mu_{\wi{k}s_is_j}$}}%
%
\special{pn 8}%
\special{pa 4400 650}%
\special{pa 4100 750}%
\special{fp}%
\special{sh 1}%
\special{pa 4100 750}%
\special{pa 4170 748}%
\special{pa 4152 734}%
\special{pa 4158 710}%
\special{pa 4100 750}%
\special{fp}%
%
\special{pn 8}%
\special{pa 4400 2600}%
\special{pa 4100 2450}%
\special{fp}%
\special{sh 1}%
\special{pa 4100 2450}%
\special{pa 4152 2498}%
\special{pa 4148 2474}%
\special{pa 4170 2462}%
\special{pa 4100 2450}%
\special{fp}%
\end{picture}%

}
\end{dfn}
%
%

The pseudo-Weyl polytope $P(\mu_{\bullet})$ 
with GGMS datum $\mu_{\bullet}=(\mu_{w})_{w \in W}$ 
(see \eqref{eq:poly}) is 
a Mirkovi\'c-Vilonen (MV) polytope if and only if 
the GGMS datum $\mu_{\bullet}=(\mu_{w})_{w \in W}$ is 
an MV datum (see the proof of \cite[Proposition~5.4]{Kam1} and 
the comment following \cite[Theorem~7.1]{Kam1}).
Also, for a dominant coweight 
$\lambda \in X_{*}(T) \subset \Fh_{\BR}$ and a coweight 
$\nu \in X_{*}(T) \subset \Fh_{\BR}$, 
an MV polytope $P=P(\mu_{\bullet})$ with GGMS datum 
$\mu_{\bullet}=(\mu_{w})_{w \in W}$ is 
an MV polytope of highest vertex $\lambda$ and lowest vertex $\nu$ 
if and only if $\mu_{w_{0}}=\lambda$, $\mu_{e}=\nu$, and 
$P$ is contained in the convex hull $\Conv(W \cdot \lambda)$ of 
the $W$-orbit $W \cdot \lambda \subset \Fh_{\BR}$ 
(see \cite[Proposition~7]{A}); 
we denote by $\mv(\lambda)_{\nu}$ the set of MV polytopes 
of highest vertex $\lambda$ and lowest vertex $\nu$. 
For each dominant coweight 
$\lambda \in X_{*}(T) \subset \Fh_{\BR}$, we set
\begin{equation*}
\mv(\lambda):=\bigsqcup_{
 \nu \in X_{*}(T)} \mv(\lambda)_{\nu}.
\end{equation*}

%
\subsection{Relation between MV polytopes and MV cycles.}
\label{subsec:geom}

In this subsection, we review the relation between MV polytopes 
and MV cycles in the affine Grassmannian. 

Let us recall the definition of MV cycles 
in the affine Grassmannian, following \cite{MV2} (and \cite{A}). 
Let $G$ be a complex, connected, reductive algebraic group 
as in (the beginning of) \S\ref{subsec:notation}. 
Let $\CO = \BC[[t]]$ denote the ring of formal power series, 
and $\CK = \BC((t))$ the field of formal Laurent series 
(the fraction field of $\CO$). 
The affine Grassmannian $\Gr$ for $G$ over $\BC$ is 
defined to be the quotient space $G(\CK)/G(\CO)$, 
equipped with the structure of a complex, algebraic ind-variety, 
where $G(\CK)$ denotes the set of $\CK$-valued points of $G$, and 
$G(\CO) \subset G(\CK)$ denotes the set of $\CO$-valued points of $G$; 
we denote by $\pi:G(\CK) \twoheadrightarrow \Gr=G(\CK)/G(\CO)$ 
the natural quotient map, which is locally trivial 
in the Zariski topology. In the following, 
for a subgroup $H \subset G(\CK)$ 
that is stable under the adjoint action of $T$ and 
for an element $w$ of the Weyl group $W \cong N_{G}(T)/T$ of $G$, 
we denote by ${}^{w}H$ the $w$-conjugate $\dot{w}H\dot{w}^{-1}$ of $H$, 
where $\dot{w} \in N_{G}(T)$ is a lift of $w \in W$. 

Since each coweight $\nu \in X_{*}(T)=\Hom(\BC^{\ast},\,T)$ 
is a regular map from $\BC^{\ast}$ to $T \subset G$, 
it gives a point $t^{\nu} \in G (\CK)$, 
which in turn, descends to a point 
$[t^{\nu}] \in \Gr=G(\CK)/G(\CO)$. 
The following simple lemma will be used in 
the proof of Lemma~\ref{lem:tran}. 
%
%
\begin{lem} \label{lem:emb}
Let $L \subset G$ be a complex, connected, 
reductive algebraic group containing the maximal torus $T$ of $G$. 
Then, for each $\nu \in X_{*}(T)$, the inclusion 
$L(\CK) [t^{\nu}] \hookrightarrow \Gr$ gives an embedding of 
the affine Grassmannian for $L$ into $\Gr$.
\end{lem}

\begin{proof}
Observe that 
$\bigl(t^{\nu} G (\CO) t^{-\nu}\bigr) \cap L (\CK) 
 = t^{\nu} L (\CO) t^{-\nu}$.
Hence the map $i_{L}:L(\CK) \rightarrow \Gr$, 
$g \mapsto g[t^{\nu}]$, is factored through 
$L(\CK) / (t^{\nu} L (\CO) t^{-\nu})$. 
Since $t^{\nu} \in L (\CK)$, we conclude that 
the map $i_{L}:L(\CK) \rightarrow \Gr$ descends to a map between 
the affine Grassmannian for $L$ and $\Gr$, as desired. 
(This construction is only at the level of sets, 
but we can indeed show that the map above commutes 
with the ind-variety structures.)
\end{proof}

For each $\nu \in X_{*}(T)$, we set
\begin{equation*}
\Gr^{\nu}:=G(\CO)[t^{\nu}] \subset \Gr,
\end{equation*}
the $G(\CO)$-orbit of $[t^{\nu}]$, which is a smooth 
quasi-projective algebraic variety over $\BC$. 
Also, for each $\nu \in X_{*}(T)$ and $w \in W$, 
we set 
\begin{equation*}
S_{\nu}^{w} := {}^{w}U(\CK)[t^{\nu}] \subset \Gr,
\end{equation*}
the ${}^{w}U(\CK)$-orbit of $[t^{\nu}]$, which is a (locally closed) 
ind-subvariety of $\Gr$; we write simply $S_{\nu}$ for $S_{\nu}^{e}$.
Then, we know the following two kinds 
of decompositions of $\Gr$ into orbits. 
First, we have
\begin{equation*}
\Gr=\bigsqcup_{\lambda \in X_{*}(T)_{+}} \Gr^{\lambda} \qquad 
\text{(Cartan decomposition)},
\end{equation*}
with $\Gr^{w \cdot \lambda}=\Gr^{\lambda}$ 
for $\lambda \in X_{*}(T)_{+}$ and $w \in W$; 
note that (see, for example, \cite[\S2]{MV2}) 
for each $\lambda \in X_{*}(T)_{+}$, the quasi-projective 
variety $\Gr^{\lambda}$ is simply-connected, and of dimension 
$2\pair{\lambda}{\rho}$, where $\rho$ denotes the half-sum 
of the positive roots $\alpha \in \Delta_{+}$ for $G$, i.e., 
$2\rho=\sum_{\alpha \in \Delta_{+}}\alpha$. 
Second, we have for each $w \in W$, 
\begin{equation*}
\Gr=\bigsqcup_{\nu \in X_{*}(T)} S_{\nu}^{w} \qquad
\text{(Iwasawa decomposition)}.
\end{equation*}
Moreover, the (Zariski-) closure relations among these orbits are 
described as follows (see \cite[\S2 and \S3]{MV2}): 
%
%
\begin{equation} \label{eq:Grlam}
\ol{\Gr^{\lambda}}=
 \bigsqcup_{
   \begin{subarray}{c} 
   \lambda' \in X_{\ast}(T)_{+} \\[1mm]
   \lambda' \le \lambda
   \end{subarray}
 } \Gr^{\lambda'}
\qquad \text{for $\lambda \in X_{*}(T)_{+}$};
\end{equation}
%
%
\begin{equation} \label{eq:Snuw}
\ol{S_{\nu}^{w}}=
 \bigsqcup_{
   \begin{subarray}{c} 
   \gamma \in X_{\ast}(T) \\[1mm]
   w^{-1} \cdot \gamma \ge w^{-1} \cdot \nu
   \end{subarray}
 } S_{\gamma}^{w}
\qquad \text{for $\nu \in X_{\ast}(T)$ and $w \in W$}.
\end{equation}
%
%
\begin{rem} \label{rem:Snuw}
Let $\CX \subset \Gr$ be an irreducible algebraic subvariety, 
and $\nu \in X_{*}(T)$, $w \in W$. Then, 
it follows from \eqref{eq:Snuw} that the intersection 
$\CX \cap S_{\nu}^{w}$ is an open dense subset of $\CX$ 
if and only if $\CX \cap S_{\nu}^{w} \ne \emptyset$ and 
$\CX \cap S_{\gamma}^{w} = \emptyset$ 
for every $\gamma \in X_{*}(T)$ with 
$w^{-1} \cdot \gamma \not\ge w^{-1} \cdot \nu$. 
\end{rem}

For $\lambda \in X_{*}(T)_{+}$, 
let $L(\lambda)$ denote the irreducible finite-dimensional 
representation of the (complex) Langlands dual 
group $G^{\vee}$ of $G$ with highest weight $\lambda$, and 
$\Omega(\lambda) \subset X_{*}(T)$ 
the set of weights of $L(\lambda)$. 
We know from \cite[Theorem~3.2 and Remark~3.3]{MV2} that 
$\nu \in X_{*}(T)$ is an element of $\Omega(\lambda)$ 
if and only if $\Gr^{\lambda} \cap S_{\nu} \ne \emptyset$, 
and then the intersection $\Gr^{\lambda} \cap S_{\nu}$ 
is of pure dimension $\pair{\lambda-\nu}{\rho}$. 

Now we come to the definition of MV cycles 
in the affine Grassmannian.
%
%
\begin{dfn}[{\cite[\S3]{MV2}; see also \cite[\S5.3]{A}}] \label{dfn:MVcycle}
Let $\lambda \in X_{*}(T)_{+}$ and $\nu \in X_{*}(T)$ be such that 
$\Gr^{\lambda} \cap S_{\nu} \ne \emptyset$, 
i.e., $\nu \in \Omega(\lambda)$. 
An MV cycle of highest weight $\lambda$ and weight $\nu$ is 
defined to be an irreducible component of the (Zariski-) closure of 
the intersection $\Gr^{\lambda} \cap S_{\nu}$. 
\end{dfn}

We denote by $\CZ(\lambda)_{\nu}$ the set of MV cycles of 
highest weight $\lambda \in X_{*}(T)_{+}$ and weight $\nu \in X_{*}(T)$. 
Also, for each $\lambda \in X_{*}(T)_{+}$, we set 
\begin{equation*}
\CZ(\lambda) := 
 \bigsqcup_{\nu \in X_{*}(T)} \CZ(\lambda)_{\nu},
\end{equation*}
where $\CZ(\lambda)_{\nu} := \emptyset$ if 
$\Gr^{\lambda} \cap S_{\nu}=\emptyset$. 
%
%
\begin{ex}[{cf.~\cite[Eq.(3.6)]{MV2}}] \label{ex:MVc}
For each $\lambda \in X_{\ast}(T)_{+}$, we have
\begin{equation*}
\CZ(\lambda)_{\lambda} = \bigl\{\,[ t^{\lambda}]\,\bigr\}, 
\quad \text{and} \quad
\CZ(\lambda)_{w_{0} \lambda} = \bigl\{\,\ol{\Gr^{\lambda}}\,\bigr\}.
\end{equation*}
\end{ex}
%
%
\begin{rem}[{\cite[Lemma~5.2]{NP}, \cite[Eq.(3.6)]{MV2}}] 
\label{rem:extcyc}
Let $\lambda \in X_{*}(T)_{+}$. If $\nu \in X_{*}(T)$ is 
of the form $\nu=x \cdot \lambda$ for some $x \in W$, then 
\begin{equation*}
\bb_{x \cdot \lambda}:=
\ol{ U(\CO)[t^{x \cdot \lambda}] } \subset 
\ol{ G(\CO)[t^{\lambda}] \cap U(\CK)[t^{x \cdot \lambda}] }=
\ol{ \Gr^{\lambda} \cap S_{x \cdot \lambda} }
\end{equation*}
is the unique MV cycle of highest weight $\lambda$ 
and weight $x \cdot \lambda$ 
(extremal MV cycle of weight $x \cdot \lambda$). 
For an explicit (combinatorial) description 
of the corresponding extremal MV polytope, 
see \S\ref{subsec:extpoly} below. 
\end{rem}

Motivated by the discovery of MV cycles in the affine Grassmannian, 
Anderson \cite{A} proposed considering the ``moment map images'' 
of MV cycles as follows: Let $\lambda \in X_{*}(T)_{+}$. 
For an MV cycle $\bb \in \CZ(\lambda)$, we set 
\begin{equation*}
P(\bb):=\Conv \bigl\{
  \nu \in X_{\ast}(T) \subset \Fh_{\BR} \mid 
  [t^{\nu}] \in \bb \bigr\},
\end{equation*}
and call $P(\bb) \subset \Fh_{\BR}$ the moment map 
image of $\bb$\,; note that $P(\bb)$ is indeed 
a convex polytope in $\Fh_{\BR}$. 

For the rest of this paper, except in \S\ref{subsec:BG} and the Appendix, 
we assume that $G$ (and hence its Lie algebra $\Fg$) is semisimple. 
The following theorem, due to Kamnitzer \cite{Kam1}, 
establishes an explicit relationship between 
MV polytopes and MV cycles. 
%
%
\begin{thm} \label{thm:Kam1}
{\rm (1)} 
Let $\lambda \in X_{*}(T)_{+}$ and $\nu \in X_{*}(T)$ be 
such that $\Gr^{\lambda} \cap S_{\nu} \ne \emptyset$. 
If $\mu_{\bullet}=(\mu_{w})_{w \in W}$ denotes the GGMS datum 
of an MV polytope $P \in \mv(\lambda)_{\nu}$, that is, 
$P=P(\mu_{\bullet}) \in \mv(\lambda)_{\nu}$, then 
\begin{equation*}
\bb(\mu_{\bullet}):=
\ol{ \bigcap_{w \in W} S^{w}_{\mu_{w}} } \subset \ol{\Gr^{\lambda}}
\end{equation*}
is an MV cycle that belongs to $\CZ(\lambda)_{\nu}$. 

{\rm (2)} 
Let $\lambda \in X_{*}(T)_{+}$. 
For an MV polytope $P=P(\mu_{\bullet}) \in \mv(\lambda)$ 
with GGMS datum $\mu_{\bullet}$, 
we set $\Phi_{\lambda}(P):=\bb(\mu_{\bullet})$. Then, 
the map $\Phi_{\lambda}:\mv(\lambda) \rightarrow \CZ(\lambda)$, 
$P \mapsto \Phi_{\lambda}(P)$, is a bijection 
from $\mv(\lambda)$ onto $\CZ(\lambda)$ 
such that $\Phi_{\lambda}(\mv(\lambda)_{\nu})=
\CZ(\lambda)_{\nu}$ for all $\nu \in X_{\ast}(T)$
with $\Gr^{\lambda} \cap S_{\nu} \ne \emptyset$. 
In particular, for each MV cycle $\bb \in \CZ(\lambda)$, 
there exists a unique MV datum $\mu_{\bullet}$ 
such that $\bb=\bb(\mu_{\bullet})$, and in this case, 
the moment map image $P(\bb)$ of the MV cycle $\bb=\bb(\mu_{\bullet})$
is identical to the MV polytope $P(\mu_{\bullet}) \in \mv(\lambda)$. 
\end{thm}
%
%
\begin{rem}[{\cite[\S2.2]{Kam1}}] \label{rem:moment}
For $\nu \in X_{*}(T)$ and $w \in W$, 
the ``moment map image'' $P(\ol{S_{\nu}^{w}})$ of 
$\ol{S_{\nu}^{w}}$ is, by definition, 
the convex hull in $\Fh_{\BR}$ of the set
$\bigl\{
  \gamma \in X_{*}(T) \subset \Fh_{\BR} \mid 
  [t^{\gamma}] \in \ol{S_{\nu}^{w}}
\bigr\} \subset \Fh_{\BR}$,
which is identical to the (shifted) convex cone
$\bigl\{
 v \in \Fh_{\BR} \mid 
 w^{-1} \cdot v - w^{-1} \cdot \nu \in 
 \textstyle{\sum_{j \in I}\BR_{\ge 0}h_{j}}
\bigr\}$. 
\end{rem}

%
\subsection{Lusztig-Berenstein-Zelevinsky (LBZ) crystal structure.}
\label{subsec:cry-MV}

We keep the notation and assumptions of \S\ref{subsec:MV}.
For an MV datum $\mu_{\bullet}=(\mu_{w})_{w \in W}$ and $j \in I$, 
we denote by $f_{j}\mu_{\bullet}$ 
(resp., $e_{j}\mu_{\bullet}$ if $\mu_{e} \ne \mu_{s_{j}}$; 
note that $\mu_{s_{j}}-\mu_{e} \in 
\BZ_{\ge 0}h_{j}$ by \eqref{eq:length})
a unique MV datum $\mu_{\bullet}'=(\mu_{w}')_{w \in W}$ 
such that $\mu_{e}'=\mu_{e}-h_{j}$ 
(resp., $\mu_{e}'=\mu_{e}+h_{j}$) and 
$\mu_{w}'=\mu_{w}$ for all $w \in W$ with $s_{j}w < w$ 
(see \cite[Theorem~3.5]{Kam2} and its proof); 
note that $\mu_{w_{0}}'=\mu_{w_{0}}$ and 
$\mu_{s_j}'=\mu_{s_j}$. 

Let $\lambda \in X_{\ast}(T) \subset \Fh_{\BR}$ 
be a dominant coweight. 
Following \cite[\S6.2]{Kam2}, we endow $\mv(\lambda)$ 
with the Lusztig-Berenstein-Zelevinsky (LBZ) 
crystal structure for $U_{q}(\Fg^{\vee})$
as follows. 
Let $P=P(\mu_{\bullet}) \in \mv(\lambda)$ be 
an MV polytope with GGMS datum 
$\mu_{\bullet}=(\mu_{w})_{w \in W}$.
The weight $\wt(P)$ of $P$ is, by definition, 
equal to the vertex $\mu_{e} \in \lambda-Q^{\vee}_{+}$. 
For each $j \in I$, 
we define the lowering Kashiwara operator
$f_{j}:\mv(\lambda) \cup \{\bzero\} \rightarrow 
 \mv(\lambda) \cup \{\bzero\}$ and 
the raising Kashiwara operator
$e_{j}:\mv(\lambda) \cup \{\bzero\} \rightarrow 
 \mv(\lambda) \cup \{\bzero\}$ by: 
\begin{align*}
e_{j}\bzero=f_{j}\bzero & :=\bzero, \\[3mm]
f_{j}P=f_{j}P(\mu_{\bullet}) & :=
\begin{cases}
 P(f_{j}\mu_{\bullet}) 
 & \text{if $P(f_{j}\mu_{\bullet}) \subset 
         \Conv (W \cdot \lambda)$}, \\[1.5mm]
 \bzero & \text{otherwise}, 
 \end{cases} \\[3mm]
e_{j}P=e_{j}P(\mu_{\bullet}) & :=
\begin{cases}
 P(e_{j}\mu_{\bullet}) & 
 \text{if $\mu_{e} \ne \mu_{s_{j}}$ 
 (i.e., $\mu_{s_{j}}-\mu_{e} \in \BZ_{> 0}h_{j}$)}, \\[1.5mm]
 \bzero & \text{otherwise}, 
 \end{cases}
\end{align*}
where $\bzero$ is an additional element, 
not contained in $\mv(\lambda)$. 
For $j \in I$, we set 
$\ve_{j}(P):=
 \max \bigl\{k \in \BZ_{\ge 0} \mid e_{j}^{k}P \ne \bzero \bigr\}$ and 
$\vp_{j}(P):=
 \max \bigl\{k \in \BZ_{\ge 0} \mid f_{j}^{k}P \ne \bzero \bigr\}$;
note that for each $j \in I$, we have
%
%
\begin{equation} \label{eq:ax}
\vp_{j}(P)=
\pair{\wt(P)}{\alpha_{j}}+\ve_{j}(P) 
 \qquad \text{for all $P \in \mv(\lambda)$.}
\end{equation}
\begin{rem} \label{rem:ve}
Let $P=P(\mu_{\bullet}) \in \mv(\lambda)$ be 
an MV polytope with GGMS datum 
$\mu_{\bullet}=(\mu_{w})_{w \in W}$. 
Then, we deduce from the definition of 
the raising Kashiwara operators $e_{j}$ 
(or, the MV datum $e_{j}\mu_{\bullet}$) that 
$\mu_{s_{j}}-\mu_{e}=\ve_{j}(P) h_{j}$ 
for $j \in I$.
\end{rem}

%
\begin{thm}[{\cite[Theorem~6.4]{Kam2}}] \label{thm:kam-int}
The set $\mv(\lambda)$, equipped with the maps 
$\wt$, $e_{j},\,f_{j} \ (j \in I)$, and 
$\ve_{j},\,\vp_{j} \ (j \in I)$ above, is 
a crystal for $U_{q}(\Fg^{\vee})$. 
Moreover, there exists a unique isomorphism 
$\Psi_{\lambda}:\CB(\lambda) 
 \stackrel{\sim}{\rightarrow} 
 \mv(\lambda)$ of crystals for $U_{q}(\Fg^{\vee})$. 
\end{thm}

\begin{rem}
Kamnitzer \cite{Kam2} proved that 
for each $\lambda \in X_{*}(T)_{+}$, the bijection 
$\Phi_{\lambda}:\mv(\lambda) \rightarrow \CZ(\lambda)$ in 
Theorem~\ref{thm:Kam1}\,(2) also intertwines the LBZ crystal 
structure on $\mv(\lambda)$ and the crystal structure on $\CZ(\lambda)$ 
defined in \cite{BrGa} (and \cite{BFG}). 
\end{rem}

For each $x \in W$, we denote by 
$\mv_{x}(\lambda) \subset \mv(\lambda)$ 
the image $\Psi_{\lambda}(\CB_{x}(\lambda))$ 
of the Demazure crystal 
$\CB_{x}(\lambda) \subset \CB(\lambda)$ 
associated to $x \in W$ under the isomorphism 
$\Psi_{\lambda}:\CB(\lambda) 
 \stackrel{\sim}{\rightarrow} 
 \mv(\lambda)$ in Theorem~\ref{thm:kam-int};
for a combinatorial description of $\mv_{x}(\lambda)$ 
in terms of the lengths $\Ni{l} \in \BZ_{\ge 0}$, 
$\bi \in R(w_{0})$, $0 \le l \le m$, 
of edges of an MV polytope, see \cite[\S3.2]{NS-dp}. 

%
\subsection{Extremal MV polytopes.}
\label{subsec:extpoly}

Let $\Fg$ be a complex semisimple Lie algebra as in 
(the second paragraph of) \S\ref{subsec:notation}. 
Let $\lambda \in X_{\ast}(T) \subset \Fh_{\BR}$ be 
a dominant coweight. 
For each $x \in W$, we denote by 
$P_{x \cdot \lambda}$ the image of the extremal element 
$u_{x \cdot \lambda} \in \CB(\lambda)$ of weight 
$x \cdot \lambda \in X_{\ast}(T) \subset \Fh_{\BR}$ 
under the isomorphism 
$\Psi_{\lambda}:\CB(\lambda) 
 \stackrel{\sim}{\rightarrow} 
 \mv(\lambda)$
in Theorem~\ref{thm:kam-int}; 
we call $P_{x \cdot \lambda} \in \mv(\lambda)$ 
the extremal MV polytope of weight $x \cdot \lambda$.
We know the following polytopal description of 
the extremal MV polytopes 
from \cite[Theorem~4.1.5\,(2)]{NS-dp}.
%
%
\begin{prop} \label{prop:ext}
Let $\lambda \in X_{\ast}(T) \subset \Fh_{\BR}$ 
be a dominant coweight, and $x \in W$. 
The extremal MV polytope 
$P_{x \cdot \lambda}$ of weight $x \cdot \lambda$ 
is identical to the convex hull 
$\Conv(W_{\le x} \cdot \lambda)$ 
in $\Fh_{\BR}$ of the set $W_{\le x} \cdot \lambda$, 
where $W_{\le x}$ denotes the subset 
$\bigl\{z \in W \mid z \le x\bigr\}$ of $W$. 
\end{prop}

\begin{rem}
In \cite{NS-dp}, we proved Proposition~\ref{prop:ext} above and 
Theorem~\ref{thm:GGMS-Ext} below in the case that 
$\Fg$ is simply-laced. However, 
these results hold also in the case that 
$\Fg$ is not simply-laced; for example, we can use 
a standard technique of ``folding'' by diagram automorphisms 
(see \cite{NS-fm}, \cite{Hong}, and also \cite{Lu08}). 
\end{rem}

%
\begin{rem} \label{rem:ext}
It follows from Theorem~\ref{thm:Kam1} 
that for each $\lambda \in X_{*}(T)_{+} \subset \Fh_{\BR}$ 
and $x \in W$, the extremal MV polytope $P_{x \cdot \lambda}$ 
is identical to the moment map image $P(\bb_{x \cdot \lambda})$ 
of the extremal MV cycle $\bb_{x \cdot \lambda}$ 
(see Remark~\ref{rem:extcyc}). 
In particular, the highest weight element 
$P_{e \cdot \lambda}=P_{\lambda}$ of $\mv(\lambda)$ 
is identical to the set $P([t^{\lambda}])=\bigl\{\lambda\bigr\}$, 
and the lowest weight element $P_{w_{0} \cdot \lambda}$ of 
$\mv(\lambda)$ is identical to the set 
$P(\ol{\Gr^{\lambda}})=\Conv(W \cdot \lambda)$. 
\end{rem}

The GGMS datum of an extremal MV polytope is given 
as follows (see \cite[\S4.1]{NS-dp}). 
Let us fix a dominant coweight 
$\lambda \in X_{*}(T) \subset \Fh_{\BR}$ 
and $x \in W$ arbitrarily. 
Let $p$ denote the length $\ell(xw_{0})$ of $xw_{0} \in W$. 
For each $\bi=(i_{1},\,i_{2},\,\dots,\,i_{m}) \in R(w_{0})$, 
with $m=\ell(w_{0})$, we set 
\begin{equation*}
S(xw_{0},\,\bi)=\left\{
\begin{array}{l|l}
 (a_{1},\,a_{2},\,\dots,\,a_{p}) \in [1,m]_{\BZ}^{p} \ & \ 
\begin{array}{l}
1 \le a_{1} < a_{2} < \cdots < a_{p} \le m, \\[1.5mm]
\si{a_1}\si{a_2} \cdots \si{a_p}=xw_{0}
\end{array}
\end{array}
\right\},
\end{equation*}
where $[1,m]_{\BZ}:=\bigl\{a \in \BZ \mid 1 \le a \le m\bigr\}$.
We denote by $\min S(xw_{0},\,\bi)$ 
the minimum element of the set $S(xw_{0},\,\bi)$ 
in the lexicographic ordering;
recall that the lexicographic ordering $\succeq$ on 
$S(xw_{0},\,\bi)$ is defined as follows: 
$(a_{1},\,a_{2},\,\dots,\,a_{p}) \succ
 (b_{1},\,b_{2},\,\dots,\,b_{p})$ 
if there exists some integer $1 \le q_0 \le p$ 
such that $a_{q}=b_{q}$ for all $1 \le q \le q_0-1$ 
and $a_{q_0} > b_{q_0}$. 
Now we define a sequence 
$\yi{0},\,\yi{1},\,\dots,\,\yi{m}$ of elements of $W$ 
inductively by the following formula (see \cite[\S4.2]{NS-dp}): 
%
%
\begin{equation} \label{eq:yi}
\yi{m}=e, \qquad 
\yi{l-1}=
 \begin{cases}
 \yi{l} & \text{if $l$ appears in $\min S(xw_{0},\,\bi)$}, \\[1.5mm]
 s_{\bti{l}}\yi{l} & \text{otherwise}
 \end{cases}
\end{equation}
for $1 \le l \le m$, 
where we set $\bti{l}:=\wi{l-1} \cdot \alpha_{i_{l}}$ 
for $1 \le l \le m$, and denote 
by $s_{\beta} \in W$ the reflection 
with respect to a root $\beta$. 
%
%
\begin{rem} \label{rem:zi}
The element $\yi{l} \in W$ above does not depend 
on the dominant coweight $\lambda \in X_{*}(T) \subset \Fh_{\BR}$. 
\end{rem}
%
%
\begin{rem} \label{rem:vi}
Let $\bi=(i_{1},\,i_{2},\,\dots,\,i_{m}) \in R(w_{0})$. 
We define a sequence 
$\vi{0},\,\vi{1},\,\dots,\,\vi{m}$ of elements of $W$ 
inductively by the following formula: 
\begin{equation*}
\vi{m}=e, \qquad 
\vi{l-1}=
 \begin{cases}
 s_{i_{l}}\vi{l} & \text{if $l$ appears in $\min S(xw_{0},\,\bi)$}, \\[1.5mm]
 \vi{l} & \text{otherwise}
 \end{cases}
\end{equation*}
for $1 \le l \le m$; we see from the definition of 
the set $S(xw_{0},\,\bi)$ that 
$\ell(\vi{l-1})=\ell(\vi{l})+1$ 
if $l$ appears in $\min S(xw_{0},\,\bi)$. 
Then we know from \cite[Lemma~4.2.1]{NS-dp}
that $\yi{l}=\wi{l}\vi{l}w_{0}^{-1}$ for every $0 \le l \le m$. 
\end{rem}
%
%
\begin{thm} \label{thm:GGMS-Ext}
Keep the notation and assumptions above. 
Let $\mu_{\bullet}=(\mu_{w})_{w \in W}$ be the GGMS datum of 
the extremal MV polytope $P_{x \cdot \lambda}$, i.e., 
$P_{x \cdot \lambda}=P(\mu_{\bullet})$. 
Let $w \in W$ be such that $w=\wi{l}$ for some 
$\bi \in R(w_{0})$ and $0 \le l \le m$. Then, we have
$\mu_{w}=\mu_{\wi{l}}=\yi{l} \cdot \lambda$.
\end{thm}

The following results on extremal MV polytopes and extremal MV cycles 
play an important role in the proof of Theorem~\ref{thm:tensor}
given in \S\ref{subsec:prf-tensor}. 
%
%
\begin{lem} \label{lem:ext1}
Keep the notation and assumptions of 
Theorem~\ref{thm:GGMS-Ext}.
For each $w \in W$ and $j \in I$ with $w < ws_{j}$, 
we have either 
{\rm (a)} $\mu_{ws_{j}}=\mu_{w}$, or 
{\rm (b)} $\mu_{ws_{j}}=ws_{j}w^{-1} \cdot \mu_{w}$.
Moreover, in both of the cases {\rm (a)} and {\rm (b)}, 
we have $\pair{\mu_{ws_{j}}}{w \cdot \alpha_{j}} \ge 0$. 
\end{lem}

\begin{proof}
Take $\bi=(i_{1},\,i_{2},\,\dots,\,i_{m}) \in R(w_{0})$ such that 
$\wi{l-1}=w$ and $\wi{l}=ws_{j}$ for some $1 \le l \le m$; 
note that $i_{l}=j$, $\bti{l}=\wi{l-1} \cdot \alpha_{i_{l}}=
w \cdot \alpha_{j}$, and hence $s_{\bti{l}}=ws_{j}w^{-1}$. 
Since $\mu_{w}=\mu_{\wi{l-1}}=\yi{l-1} \cdot \lambda$ and 
$\mu_{ws_{j}}=\mu_{\wi{l}}=\yi{l} \cdot \lambda$ 
by Theorem~\ref{thm:GGMS-Ext}, and since 
$\yi{l-1}$ is equal to $\yi{l}$ or 
$s_{\bti{l}}\yi{l}=ws_{j}w^{-1}\yi{l}$ 
by definition, it follows immediately that 
either (a) $\mu_{ws_{j}}=\mu_{w}$ or 
(b) $\mu_{ws_{j}}=ws_{j}w^{-1} \cdot \mu_{w}$ holds. 

We will show that 
$\pair{\mu_{ws_{j}}}{w \cdot \alpha_{j}} \ge 0$. 
First, let us assume that 
$l$ does not appear in $\min S(xw_{0},\,\bi)$.
Then, we have $\yi{l-1}=s_{\bti{l}}\yi{l}$ by definition, 
and hence $\mu_{\wi{l-1}}=s_{\bti{l}} \cdot \mu_{\wi{l}}$ 
by Theorem~\ref{thm:GGMS-Ext}. Also, 
it follows from the length formula \eqref{eq:length} 
(or \eqref{eq:n}) that 
$\mu_{\wi{l}}-\mu_{\wi{l-1}} \in 
 \BZ_{\ge 0}(\wi{l-1} \cdot h_{i_{l}})=
 \BZ_{\ge 0}(\bti{l})^{\vee}$, where 
$(\bti{l})^{\vee}$ denotes the coroot corresponding to 
the root $\bti{l}$. 
Combining these, we obtain 
\begin{equation*}
\BZ_{\ge 0}(\bti{l})^{\vee} \ni 
\mu_{\wi{l}}-\mu_{\wi{l-1}}=
\mu_{\wi{l}}-s_{\bti{l}} \cdot \mu_{\wi{l}}=
\pair{\mu_{\wi{l}}}{\bti{l}}(\bti{l})^{\vee}, 
\end{equation*}
and hence $\pair{\mu_{\wi{l}}}{\bti{l}} \ge 0$. 
This implies $\pair{\mu_{ws_{j}}}{w \cdot \alpha_{j}} \ge 0$ 
since $\wi{l}=ws_{j}$ and $\bti{l}=w \cdot \alpha_{j}$. 

Next, let us assume that 
$l$ appears in $\min S(xw_{0},\,\bi)$. 
Because $\mu_{\wi{l}}=\yi{l} \cdot \lambda=
\wi{l}\vi{l}w_{0}^{-1} \cdot \lambda$ 
by Theorem~\ref{thm:GGMS-Ext} and Remark~\ref{rem:vi}, 
we see, by noting $\wi{l}=ws_{j}$ and $i_{l}=j$, that 
\begin{align*}
\pair{\mu_{ws_{j}}}{w \cdot \alpha_{j}} & =
\pair{\mu_{\wi{l}}}{w \cdot \alpha_{j}} =
\pair{\wi{l}\vi{l}w_{0}^{-1} \cdot \lambda}{w \cdot \alpha_{j}} \\
& =
 \pair{ws_{j}\vi{l}w_{0}^{-1} \cdot \lambda}{w \cdot \alpha_{j}}=
 -\pair{\vi{l}w_{0}^{-1} \cdot \lambda}{\alpha_{j}} \\
&= -\pair{\lambda}{ w_{0}(\vi{l})^{-1} \cdot \alpha_{j} }
 = -\pair{\lambda}{ w_{0}(\vi{l})^{-1} \cdot \alpha_{i_{l}} }.
\end{align*}
Also, since $l$ appears in $\min S(xw_{0},\,\bi)$ by assumption, 
we have 
$\vi{l-1}=s_{i_{l}}\vi{l}$ with $\ell(\vi{l-1})=\ell(\vi{l})+1$ 
(see Remark~\ref{rem:vi}). 
It follows from the exchange condition that 
$(\vi{l})^{-1} \cdot \alpha_{i_{l}}$ is a positive root, 
and hence $w_{0}(\vi{l})^{-1} \cdot \alpha_{i_{l}}$ 
is a negative root. Therefore, we conclude that 
\begin{equation*}
\pair{\mu_{ws_{j}}}{w \cdot \alpha_{j}}=
- \underbrace{\pair{\lambda}
 { w_{0}(\vi{l})^{-1} \cdot \alpha_{i_{l}} }}_{\le 0} \ge 0
\end{equation*}
since $\lambda \in X_{*}(T) \subset \Fh_{\BR}$ is 
a dominant coweight, This proves the lemma. 
\end{proof}

Let $G$ be a complex, connected, semisimple algebraic group 
with Lie algebra $\Fg$. 
Take $\lambda \in X_{*}(T)_{+}$ and $x \in W$ arbitrarily, and 
let $\mu_{\bullet}=(\mu_{w})_{w \in W}$ denote the GGMS datum of 
the extremal MV polytope $P_{x \cdot \lambda} \in \mv(\lambda)$ of 
weight $x \cdot \lambda$, i.e., 
$P_{x \cdot \lambda}=P(\mu_{\bullet})$; 
recall from Theorem~\ref{thm:GGMS-Ext} that 
$\mu_{w} \in W \cdot \lambda$ for all $w \in W$. 
Now, for each $w \in W$, we consider the irreducible 
algebraic variety 
%
%
\begin{equation} \label{eq:bbw}
\bb^{w} := \ol{\Gr^{\lambda} \cap S_{\mu_{w}}^{w}},
\end{equation}
which is the $\dot{w}$-translate of 
the extremal MV cycle $\bb_{w^{-1} \cdot \mu_{w}}$ 
of weight $w^{-1} \cdot \mu_{w}$ 
since $\Gr^{w^{-1} \cdot \lambda}=\Gr^{\lambda}$ 
(see Remark~\ref{rem:extcyc}); 
note that $\bb^{e} = \bb_{x \cdot \lambda}$ 
since $\mu_{e}=x \cdot \lambda$. 

For each $j \in I$, 
we set $P_{j} := B \sqcup (B \dot{s_{j}} B)$, 
which is the minimal parabolic subgroup 
(containing $B$) of $G$ corresponding to $s_{j} \in W$. 
Also, let $P_{j} = L_{j} U_{j}$ be its Levi decomposition 
such that $T \subset L_{j}$. 
%
%
\begin{lem} \label{lem:tran}
Keep the notation above. 
For each $w \in W$ and $j \in I$ with $ws_{j} < w$, 
we have ${}^{w} L_{j}(\CO) \bb^{ws_{j}} \subset \bb^{w}$.
\end{lem}

\begin{proof}
For simplicity of notation, we write 
$N_{+}$ for ${}^{w}(L_{j} \cap U)$; 
the root in $N_{+}$ is $-w \cdot \alpha_{j}$ 
by our convention. 
Because $\mu_{\bullet}=(\mu_{w})_{w \in W}$ is 
the GGMS datum of the extremal MV polytope $P_{x \cdot \lambda}$, 
it follows from Lemma~\ref{lem:ext1} that 
we have either (a) $\mu_{ws_{j}}=\mu_{w}$ or 
(b) $\mu_{ws_{j}}=ws_{j}w^{-1} \cdot \mu_{w}$, and that 
in both of the cases (a) and (b), we have
$\pair{\mu_{w}}{w \cdot \alpha_{j}} \le 0$. 
Consequently, by taking into account 
Lemma~\ref{lem:emb}, applied to ${}^{w}L_{j} \subset G$ 
and $\mu_{w} \in X_{*}(T)$, we deduce from Example~\ref{ex:MVc} 
that
%
%
\begin{equation} \label{eq:tran1}
\ol{N_{+}(\CO) [t^{\mu_{w}}]} = 
\ol{{}^{w} L_{j}(\CO) [t^{\mu_{w}}]}.
\end{equation}
Also, by Remark~\ref{rem:extcyc}, applied to 
the extremal MV cycle $\dot{w}^{-1} \cdot \bb^{w}$ of 
weight $w^{-1} \cdot \mu_{w}$, we obtain 
$\dot{w}^{-1} \cdot \bb^{w}=
\ol{U(\CO)[t^{w^{-1} \cdot \mu_{w}}]}$, 
and hence
%
%
\begin{equation} \label{eq:tran2}
\bb^{w}=\ol{{}^{w}U(\CO)[t^{\mu_{w}}]}.
\end{equation}
Similarly, we obtain
%
%
\begin{equation} \label{eq:tran3}
\bb^{ws_{j}}=\ol{{}^{ws_{j}}U(\CO)[t^{\mu_{ws_{j}}}]}.
\end{equation}
Here we note that 
${}^{w}L_{j} {}^{ws_{j}}U = {}^{w}U {}^{w}L_{j}$ 
since $L_{j}U=UL_{j}$ and $\dot{s}_{j} \in L_{j}$.
It follows that 
\begin{equation*}
{}^{w}L_{j}(\CO) {}^{ws_{j}}U(\CO) [t^{\mu_{ws_{j}}}]
= {}^{w}U(\CO) {}^{w}L_{j}(\CO)  [t^{\mu_{ws_{j}}}].
\end{equation*}
Because 
$[t^{\mu_{ws_{j}}}] = [t^{\mu_{w}}]$ in case (a), and 
$[t^{\mu_{ws_{j}}}] = \dot{w} \dot{s}_{j} \dot{w}^{-1} [t^{\mu_{w}}]$ 
in case (b) as above, we deduce that in both of the cases (a) and (b), 
\begin{equation*}
{}^{w}U(\CO) {}^{w}L_{j}(\CO)  [t^{\mu_{ws_{j}}}]=
{}^{w}U(\CO) {}^{w}L_{j}(\CO)  [t^{\mu_{w}}].
\end{equation*}
Thus, we get
\begin{equation*}
{}^{w}L_{j}(\CO) {}^{ws_{j}}U(\CO) [t^{\mu_{ws_{j}}}]=
{}^{w}U(\CO) {}^{w}L_{j}(\CO)  [t^{\mu_{w}}].
\end{equation*}
In addition, we have 
\begin{align*}
{}^{w}U(\CO) {}^{w}L_{j}(\CO)  [t^{\mu_{w}}]
& \subset {}^{w}U(\CO) \ol{{}^{w}L_{j}(\CO)  [t^{\mu_{w}}]}
  = {}^{w}U(\CO) \ol{N_{+}(\CO) [t^{\mu_{w}}]}
  \quad 
  \text{by \eqref{eq:tran1}} \\
& \subset {}^{w}U(\CO) \ol{{}^{w}U(\CO) [t^{\mu_{w}}]}
  \quad
  \text{since $N_{+} \subset {}^{w}U$ by definition} \\
& \subset \ol{ {}^{w}U(\CO){}^{w}U(\CO) [t^{\mu_{w}}]} \\
& = \ol{{}^{w}U(\CO) [t^{\mu_{w}}]}=\bb^{w}
  \quad \text{by \eqref{eq:tran2}}.
\end{align*}
Hence we obtain 
${}^{w}L_{j}(\CO) 
 {}^{ws_{j}}U(\CO) [t^{\mu_{ws_{j}}}] \subset \bb^{w}$. 
From this, we conclude, by using \eqref{eq:tran3}, that
\begin{align*}
{}^{w} L_{j}(\CO) \bb^{ws_{j}} & =
{}^{w}L_{j}(\CO) \ol{{}^{ws_{j}}U(\CO) [t^{\mu_{ws_{j}}}]} \\
& \subset 
\ol{{}^{w}L_{j}(\CO) {}^{ws_{j}}U(\CO) [t^{\mu_{ws_{j}}}]} \\
& \subset \bb^{w}.
\end{align*}
This proves the lemma.
\end{proof}

%
\section{$N$-multiple maps for MV polytopes and their applications.}
\label{sec:multi}

As in (the second paragraph of) \S\ref{subsec:notation}, 
we assume that $\Fg$ is a complex semisimple Lie algebra. 
Let $\lambda \in X_{\ast}(T) \subset \Fh_{\BR}$ be 
an arbitrary (but fixed) dominant coweight. 

%
\subsection{$N$-multiple maps for MV polytopes.}
\label{subsec:multiple}

Let $N \in \BZ_{\ge 1}$. 
For a collection $\mu_{\bullet}=(\mu_{w})_{w \in W}$ of 
elements of $\Fh_{\BR}$, we set
\begin{equation*}
N \cdot \mu_{\bullet}:=(N\mu_{w})_{w \in W}.
\end{equation*}
Also, for a subset $P \subset \Fh_{\BR}$, we set 
\begin{equation*}
N \cdot P:=\bigl\{Nv \mid v \in P\bigr\} \subset \Fh_{\BR}.
\end{equation*}
The next lemma follows immediately 
from the definitions. 
%
%
\begin{lem} \label{lem:N1}
Let $N \in \BZ_{\ge 1}$ be a positive integer. 

{\rm (1)} If $\mu_{\bullet}=(\mu_{w})_{w \in W}$ 
is a GGMS datum, then $N \cdot \mu_{\bullet}$ is 
also a GGMS datum.

{\rm (2)} If $\mu_{\bullet}=(\mu_{w})_{w \in W}$ 
is an MV datum, then $N \cdot \mu_{\bullet}$ is 
also an MV datum.

{\rm (3)} Let $P=P(\mu_{\bullet}) \in \mv$ be an MV polytope 
with GGMS datum $\mu_{\bullet}$. Then, 
$N \cdot P$ is the MV polytope
with GGMS datum $N \cdot \mu_{\bullet}$, that is, 
$N \cdot P=P(N \cdot \mu_{\bullet})$. Moreover, 
if $P \in \mv(\lambda)$, 
then $N \cdot P \in \mv(N\lambda)$. 
\end{lem}
%
%
\begin{rem} \label{rem:N1}
Let $N \in \BZ_{\ge 1}$ be a positive integer. 
If $P=P(\mu_{\bullet})$ is a pseudo-Weyl polytope 
with GGMS datum $\mu_{\bullet}$, then 
the set $N \cdot P$ is identical to the Minkowski sum 
$P+P+ \cdots +P$ ($N$ times). Indeed, we see that 
\begin{align*}
N \cdot P & = 
 P(N \cdot \mu_{\bullet}) \quad \text{by Lemma~\ref{lem:N1}\,(3)} \\
& = 
 P(\underbrace{\mu_{\bullet}+\mu_{\bullet}+ \cdots +\mu_{\bullet}}_{\text{$N$ times}})
 \quad \text{(see Remark~\ref{rem:GGMS-sum})} \\[3mm]
& 
 =\underbrace{P(\mu_{\bullet})+P(\mu_{\bullet})+ \cdots +P(\mu_{\bullet})}_{\text{$N$ times}}
 \quad \text{by Proposition~\ref{prop:Minkowski}} \\[3mm]
& = \underbrace{P+P+ \cdots +P}_{\text{$N$ times}}. 
\end{align*}
\end{rem}

By Lemma~\ref{lem:N1}\,(3), 
we obtain an injective map 
$S_{N}:\mv(\lambda) \hookrightarrow \mv(N\lambda)$ 
that sends $P \in \mv(\lambda)$ to $N \cdot P \in \mv(N\lambda)$; 
we call the map $S_{N}$ an $N$-multiple map. 
Note that $S_{N}(P_{\lambda})=P_{N\lambda}$ 
(see Remark~\ref{rem:ext}).

%
\begin{prop} \label{prop:N2}
Let $N \in \BZ_{\ge 1}$. 
For $P \in \mv(\lambda)$, we have
\begin{align*}
& \wt (S_{N}(P))=N \wt (P), \\
& 
S_{N}(e_{j}P)=e_{j}^{N}(S_{N}(P)), \qquad 
S_{N}(f_{j}P)=f_{j}^{N}(S_{N}(P)) \qquad 
\text{\rm for $j \in I$}, \\
& 
\ve_{j}(S_{N}(P))=N \ve_{j}(P), \qquad 
\vp_{j}(S_{N}(P))=N \vp_{j}(P) \qquad 
\text{\rm for $j \in I$},
\end{align*}
where it is understood 
that $S_{N}(\bzero)=\bzero$. 
\end{prop}
\begin{proof}
Let $\mu_{\bullet}=(\mu_{w})_{w \in W}$ be 
the GGMS datum of $P \in \mv(\lambda)$. 
It follows from the definition of $\wt$ and Lemma~\ref{lem:N1}\,(3) 
that $\wt(P)=\mu_{e}$ and $\wt(S_{N}(P))=\wt(N \cdot P)=N\mu_{e}$. 
Hence we have $\wt (S_{N}(P))=N \wt (P)$. 

Next, let us show that 
$\ve_{j}(S_{N}(P))=N \ve_{j}(P)$ and 
$\vp_{j}(S_{N}(P))=N \vp_{j}(P)$ for $j \in I$. 
Let $j \in I$. By Remark~\ref{rem:ve}, we have
$\mu_{s_{j}}-\mu_{e}=\ve_{j}(P)h_{j}$. Also, 
since $S_{N}(P)=N \cdot P=P(N \cdot \mu_{\bullet})$ 
by Lemma~\ref{lem:N1}\,(3), 
it follows from Remark~\ref{rem:ve} that
$N\mu_{s_{j}}-N\mu_{e}=
\ve_{j}(N \cdot P)h_{j}=
\ve_{j}(S_{N}(P))h_{j}$. 
Combining these equations, we have
\begin{equation*}
\ve_{j}(S_{N}(P))h_{j}=
N\mu_{s_{j}}-N\mu_{e}=
N(\mu_{s_{j}}-\mu_{e})=
N\ve_{j}(P)h_{j}, 
\end{equation*}
which implies that 
$\ve_{j}(S_{N}(P))=N \ve_{j}(P)$. 
In addition, we have 
\begin{align*}
\vp_{j}(S_{N}(P)) 
& =
 \pair{\wt (S_{N}(P))}{\alpha_{j}}+\ve_{j}(S_{N}(P))
 \qquad \text{by \eqref{eq:ax}} \\
& =\pair{N \wt (P)}{\alpha_{j}}+N\ve_{j}(P)
 \qquad \text{by the equations shown above} \\
& =N\bigl(\pair{\wt (P)}{\alpha_{j}}+\ve_{j}(P)\bigr) \\
& =N\vp_{j}(P)
 \qquad \text{by \eqref{eq:ax}}.
\end{align*}

Finally, let us show that 
$S_{N}(e_{j}P)=e_{j}^{N}(S_{N}(P))$ and 
$S_{N}(f_{j}P)=f_{j}^{N}(S_{N}(P))$ for $j \in I$; 
we give a proof only for the equality 
$S_{N}(e_{j}P)=e_{j}^{N}(S_{N}(P))$ 
since the equality $S_{N}(f_{j}P)=f_{j}^{N}(S_{N}(P))$ 
can be shown similarly. Let $j \in I$. 
First observe that 
\begin{align*}
S_{N}(e_{j} P)=\bzero
 & \quad \Longleftrightarrow \quad 
   e_{j} P=\bzero
 \quad \text{by Remark~\ref{rem:ve}} 
 \\
 & \quad \Longleftrightarrow \quad 
   \ve_{j}(P)=0
 \quad \text{by the definition of $\ve_{j}(P)$} \\
 & \quad \Longleftrightarrow \quad 
   e_{j}^{N}(S_{N}(P))=\bzero
 \quad \text{since $\ve_{j}(S_{N}(P))=N \ve_{j}(P)$}.
\end{align*}
Now, assume that $S_{N}(e_{j} P) \ne \bzero$, 
or equivalently, $e_{j} P \ne \bzero$. Recall 
from the definition of the raising Kashiwara operator 
$e_{j}$ that the GGMS datum of 
the MV polytope $e_{j} P \in \mv(\lambda)$ is equal to 
$e_{j}\mu_{\bullet}=(\mu_{w}')_{w \in W}$, 
which is the unique MV datum 
such that $\mu_{e}'=\mu_{e}+h_{j}$ and 
$\mu_{w}'=\mu_{w}$ for all $w \in W$ with $s_{j}w < w$. 
Hence we see from Lemma~\ref{lem:N1}\,(3) that 
the GGMS datum of the MV polytope 
$S_{N}(e_{j}P)=N \cdot (e_{j} P) \in \mv(N\lambda)$ is equal to 
the unique MV datum $\mu_{\bullet}''=(\mu_{w}'')_{w \in W}$ 
such that $\mu_{e}''=N\mu_{e}+Nh_{j}$ and 
$\mu_{w}''=N\mu_{w}$ for all $w \in W$ with $s_{j}w < w$. 
Because the GGMS datum of the MV polytope 
$S_{N}(P)=N \cdot P \in \mv(N\lambda)$ 
is equal to $N \cdot \mu_{\bullet}=(N\mu_{w})_{w \in W}$, 
we deduce from the definition of 
the raising Kashiwara operator $e_{j}$ that 
the GGMS datum $\mu_{\bullet}'''=(\mu_{w}''')_{w \in W}$ of 
$e_{j}^{N}(S_{N}(P))=e_{j}^{N}(N \cdot P) \in \mv(N\lambda)$ 
also satisfies the condition that 
$\mu_{e}'''=N\mu_{e}+Nh_{j}$ and 
$\mu_{w}'''=N\mu_{w}$ for all $w \in W$ with $s_{j}w < w$. 
Hence, by the uniqueness of such an MV datum, we obtain 
$\mu_{\bullet}''=\mu_{\bullet}'''$, which implies that 
$S_{N}(e_{j}P)=P(\mu_{\bullet}'')=
 P(\mu_{\bullet}''')=e_{j}^{N}(S_{N}(P))$.
This completes the proof of Proposition~\ref{prop:N2}.
\end{proof}

%
\subsection{Application of $N$-multiple maps.}
\label{subsec:app-multiple}

Let $\lambda \in X_{\ast}(T) \subset \Fh_{\BR}$ be 
an arbitrary (but fixed) dominant coweight 
as in \S\ref{subsec:multiple}. 
We denote by $W_{\lambda} \subset W$ 
the stabilizer of $\lambda$ in $W$, and 
by $W^{\lambda}_{\min} \subset W$ 
the set of minimal (length) coset representatives 
modulo the subgroup $W_{\lambda} \subset W$. 

For a positive integer $N \in \BZ_{\ge 1}$, 
let us denote by 
\begin{equation*}
K_{N}: \mv(\lambda) \hookrightarrow 
 \mv(\lambda)^{\otimes N}=
 \underbrace{
   \mv(\lambda) \otimes 
   \mv(\lambda) \otimes \cdots \otimes 
   \mv(\lambda)}_{\text{$N$ times}}
\end{equation*}
the composite of the $N$-multiple map 
$S_{N}:\mv(\lambda) \hookrightarrow \mv(N\lambda)$ with 
the canonical embedding 
$G_{N}:\mv(N\lambda) \hookrightarrow \mv(\lambda)^{\otimes N}$ 
of crystals that sends the highest weight element 
$P_{N\lambda}$ of $\mv(N\lambda)$ to 
the highest weight element 
$P_{\lambda}^{\otimes N}=
 P_{\lambda} \otimes P_{\lambda} \otimes \cdots \otimes P_{\lambda}$ 
($N$ times) of $\mv(\lambda)^{\otimes N}$.
%
%
\begin{prop} \label{prop:N3}
Let $\lambda \in X_{\ast}(T) \subset \Fh_{\BR}$ be 
a dominant coweight, and 
let $x \in W^{\lambda}_{\min}$. 
If an MV polytope $P \in \mv(\lambda)$ lies 
in the Demazure crystal $\mv_{x}(\lambda)$, 
then there exist a positive integer $N \in \BZ_{\ge 1}$ 
and minimal coset representatives 
$x_{1},\,x_{2},\,\dots,\,x_{N} \in W^{\lambda}_{\min}$ 
such that 
%
%
\begin{equation} \label{eq:N3}
\begin{cases}
x \ge x_{k} \quad 
 \text{\rm for all $1 \le k \le N$;} \\[3mm]
K_{N}(P)=
 P_{x_{1} \cdot \lambda} \otimes 
 P_{x_{2} \cdot \lambda} \otimes \cdots \otimes 
 P_{x_{N} \cdot \lambda}.
\end{cases}
\end{equation}
\end{prop}

%
\begin{rem} \label{rem:MV-LS}
Keep the notation and assumption above. 
A positive integer $N \in \BZ_{\ge 1}$ and 
minimal coset representatives 
$x_{1},\,x_{2},\,\dots,\,x_{N} \in W^{\lambda}_{\min}$ 
satisfying the conditions \eqref{eq:N3}
are, in a sense, determined by 
the ``turning points'' and ``directions'' of 
the Lakshmibai-Seshadri path of 
shape $\lambda$ that corresponds to 
the MV polytope $P \in \mv(\lambda)$ 
under the (inexplicit) bijection 
via the crystal basis $\CB(\lambda)$.

We will explain the remark above more precisely. 
Let $\BB(\lambda)$ denote 
the set of Lakshmibai-Seshadri (LS for short) paths
of shape $\lambda$, which is endowed with 
a crystal structure for $U_{q}(\Fg^{\vee})$ 
by the root operators $e_{j}$, $f_{j}$, $j \in I$ 
(see \cite{L-inv} and \cite{L-ann} for details). 
We know from \cite[Theorem~4.1]{Kassim} 
(see also \cite[Th\'eor\`eme~8.2.3]{Kasb}) and 
\cite[Corollary~6.4.27]{Jo} that 
$\BB(\lambda)$ is isomorphic to 
the crystal basis $\CB(\lambda)$ 
as a crystal for $U_{q}(\Fg^{\vee})$.
Thus, we have $\BB(\lambda) \cong 
\CB(\lambda) \cong \mv(\lambda)$ 
as crystals for $U_{q}(\Fg^{\vee})$. 

Now, take an element 
$P \in \mv_{x}(\lambda) \subset \mv(\lambda)$, 
and assume that a positive integer $N \in \BZ_{\ge 1}$ 
and elements $x_{1},\,x_{2},\,\dots,\,x_{N} 
\in W^{\lambda}_{\min}$ satisfy 
conditions \eqref{eq:N3}. Consider 
a piecewise linear, continuous map 
$\pi:[0,1] \rightarrow \Fh_{\BR}$ by:
\begin{equation*}
\pi(t)=
 \sum_{l=1}^{k-1} \frac{1}{N}\,x_{l} \cdot \lambda +
 \left(t-\frac{k-1}{N}\right)x_{k} \cdot \lambda
 \qquad
 \text{for $t \in 
 \left[\frac{k-1}{N},\,\frac{k}{N}\right]$, \ 
 $1 \le k \le N$};
\end{equation*}
note that for each $1 \le k \le N$, 
the vector $x_{k} \cdot \lambda \in W \cdot \lambda \subset \Fh_{\BR}$ 
gives the direction of $\pi$ on the interval 
$[(k-1)/N,\,k/N]$, and the point $t=k/N$ is 
a turning point of $\pi$ if 
$x_{k-1} \cdot \lambda \ne x_{k} \cdot \lambda$. 
Then, we deduce from the proof of 
\cite[Theorem~4.1]{Kassim} 
(or \cite[Th\'eor\`eme~8.2.3]{Kasb}), together with
the commutative diagram \eqref{CD:multi} below, that 
this map $\pi:[0,1] \rightarrow \Fh_{\BR}$ is 
precisely the LS path of shape $\lambda$ 
that corresponds to the $P \in \mv(\lambda)$ under 
the isomorphism $\BB(\lambda) \cong 
\CB(\lambda) \cong \mv(\lambda)$ of crystals. 
The argument above implies, in particular, that 
for a fixed positive integer $N \in \BZ_{\ge 1}$, 
the elements $x_{1},\,x_{2},\,\dots,\,x_{N} \in W^{\lambda}_{\min}$ 
are determined uniquely by the MV polytope $P$ via 
the corresponding LS path. 
Also note that by the definition of LS paths, 
if a positive integer $N \in \BZ_{\ge 1}$ 
and elements $x_{1},\,x_{2},\,\dots,\,x_{N} 
\in W^{\lambda}_{\min}$ satisfy 
conditions \eqref{eq:N3}, then 
we necessarily have 
$x_{1} \ge x_{2} \ge \cdots \ge x_{N}$. 
\end{rem}

\begin{proof}[Proof of Proposition~\ref{prop:N3}]
Let $N \in \BZ_{\ge 1}$. 
We know from \cite[Theorem~3.1]{Kassim} and 
\cite[Corollarie~8.1.5]{Kasb} that there exists 
an injective map $S_{N}:\CB(\lambda) \hookrightarrow \CB(N\lambda)$ 
which sends the highest weight element $u_{\lambda} \in \CB(\lambda)$ to 
the highest weight element $u_{N\lambda} \in \CB(N\lambda)$, and 
which has the same properties as the $N$-multiple map 
$S_{N}:\mv(\lambda) \hookrightarrow \mv(N\lambda)$ 
given in Proposition~\ref{prop:N2}.
Let us denote by 
$K_{N}: \CB(\lambda) \hookrightarrow 
 \CB(\lambda)^{\otimes N}$
the composite of 
$S_{N}:\CB(\lambda) \hookrightarrow \CB(N\lambda)$ 
with the canonical embedding 
$G_{N}:\CB(N\lambda) \hookrightarrow \CB(\lambda)^{\otimes N}$ 
of crystals that sends the highest weight element 
$u_{N\lambda} \in \CB(N\lambda)$ to 
the highest weight element 
$u_{\lambda}^{\otimes N} \in \CB(\lambda)^{\otimes N}$ 
(see \cite[p.~181]{Kassim} and \cite[\S8.3]{Kasb}). 
Then, it is easily shown that the following diagram commutes:
%
%
\begin{equation} \label{CD:multi}
\begin{CD}
\CB(\lambda) @>{K_{N}}>> \CB(\lambda)^{\otimes N} \\
@V{\Psi_{\lambda}}VV @VV{\Psi_{\lambda}^{\otimes N}}V  \\
\mv(\lambda) @>{K_{N}}>> \mv(\lambda)^{\otimes N}.
\end{CD}
\end{equation}
Indeed, take an element $b \in \CB(\lambda)$, and write it as: 
$b=f_{j_{1}}f_{j_{2}} \cdots f_{j_{r}}u_{\lambda}$ 
for some $j_{1},\,j_{2},\,\dots,\,j_{r} \in I$; 
for simplicity of notation, we set 
$f_{\ast}:=f_{j_{1}}f_{j_{2}} \cdots f_{j_{r}}$ and 
$f_{\ast}^{N}:=f_{j_{1}}^{N}f_{j_{2}}^{N} \cdots f_{j_{r}}^{N}$. 
We have 
\begin{align*}
K_{N}(\Psi_{\lambda}(b))
 & = K_{N}(\Psi_{\lambda}(f_{\ast}u_{\lambda})) 
   = K_{N}(f_{\ast} P_{\lambda}) 
   = G_{N}(S_{N}(f_{\ast} P_{\lambda})) \\
 & = G_{N}(f_{\ast}^{N} P_{N\lambda})
     \quad \text{by Proposition~\ref{prop:N2}} \\
 & = f_{\ast}^{N}P_{\lambda}^{\otimes N}.
\end{align*}
Similarly, we have 
\begin{align*}
\Psi_{\lambda}^{\otimes N}(K_{N}(b))
 & = \Psi_{\lambda}^{\otimes N}(K_{N}(f_{\ast}u_{\lambda}))
   = \Psi_{\lambda}^{\otimes N}\bigl(G_{N}(S_{N}(f_{\ast} u_{\lambda}))\bigr) \\
 & = \Psi_{\lambda}^{\otimes N}\bigl(G_{N}(f_{\ast}^{N} u_{N\lambda})\bigr)
   = \Psi_{\lambda}^{\otimes N}(f_{\ast}^{N} u_{\lambda}^{\otimes N})
   = f_{\ast}^{N} P_{\lambda}^{\otimes N}.
\end{align*}
Thus, we obtain $K_{N}(\Psi_{\lambda}(b))=
\Psi_{\lambda}^{\otimes N}(K_{N}(b))$ 
for all $b \in \CB(\lambda)$. 

Now, let $P \in \mv(\lambda)$. 
Applying \cite[Proposition~8.3.2]{Kasb} to 
$\Psi_{\lambda}^{-1}(P) \in \CB(\lambda)$, we see that 
if $N \in \BZ_{\ge 1}$ contains ``sufficiently many'' divisors, 
then 
%
%
\begin{equation} \label{eq:c1}
K_{N}(\Psi_{\lambda}^{-1}(P))=
 u_{x_{1} \cdot \lambda} \otimes 
 u_{x_{2} \cdot \lambda} \otimes \cdots \otimes 
 u_{x_{N} \cdot \lambda}
\end{equation}
for some $x_{1},\,x_{2},\,\dots,\,x_{N} \in W^{\lambda}_{\min}$. 
Then, we have 
\begin{align*}
K_{N}(P) & = 
 \Psi_{\lambda}^{\otimes N}\bigl(K_{N}(\Psi_{\lambda}^{-1}(P))\bigr)
 \quad \text{by \eqref{CD:multi}} \\
& =
 \Psi_{\lambda}^{\otimes N}(u_{x_{1} \cdot \lambda} \otimes 
 u_{x_{2} \cdot \lambda} \otimes \cdots \otimes 
 u_{x_{N} \cdot \lambda}) \\
& = 
 P_{x_{1} \cdot \lambda} \otimes 
 P_{x_{2} \cdot \lambda} \otimes \cdots \otimes 
 P_{x_{N} \cdot \lambda}. 
\end{align*}
It remains to show that 
$x \ge x_{k}$ for every $1 \le k \le N$. 
In view of \cite[Proposition~4.4]{Kas} 
(see also \cite[\S9.1]{Kasb}), 
it suffices to show that 
$u_{x_{k} \cdot \lambda} \in \CB_{x}(\lambda)$ 
 for every $1 \le k \le N$. 
Let $x=s_{j_{1}}s_{j_{2}} \cdots s_{j_{r}}$ 
be a reduced expression of $x \in W$. 
We know from \cite[Proposition~9.1.3\,(2)]{Kasb} that 
%
%
\begin{equation} \label{eq:dem-f}
\CB_{x}(\lambda)=
 \bigl\{
   f_{j_{1}}^{c_{1}}
   f_{j_{2}}^{c_{2}} \cdots 
   f_{j_{r}}^{c_{r}}u_{\lambda} \mid 
 c_{1},\,c_{2},\,\dots,\,c_{r} \in \BZ_{\ge 0}
 \bigr\} \setminus \{0\}.
\end{equation}
Since $P \in \mv_{x}(\lambda)$ by our assumption, 
$\Psi_{\lambda}^{-1}(P)$ is contained in $\CB_{x}(\lambda)$, 
and hence $\Psi_{\lambda}^{-1}(P)$ can be written as:
\begin{equation*}
\Psi_{\lambda}^{-1}(P)=
   f_{j_{1}}^{c_{1}}
   f_{j_{2}}^{c_{2}} \cdots 
   f_{j_{r}}^{c_{r}}u_{\lambda}
\quad 
\text{for some 
$c_{1},\,c_{2},\,\dots,\,c_{r} \in \BZ_{\ge 0}$}.
\end{equation*}
Therefore, we have 
\begin{align*}
K_{N}(\Psi_{\lambda}^{-1}(P)) & = 
 K_{N}(f_{j_{1}}^{c_{1}}
   f_{j_{2}}^{c_{2}} \cdots 
   f_{j_{r}}^{c_{r}}u_{\lambda})=
 G_{N}(S_{N}(f_{j_{1}}^{c_{1}}
   f_{j_{2}}^{c_{2}} \cdots 
   f_{j_{r}}^{c_{r}}u_{\lambda})) \\
& = G_{N}(f_{j_{1}}^{Nc_{1}}
   f_{j_{2}}^{Nc_{2}} \cdots 
   f_{j_{r}}^{Nc_{r}}u_{N\lambda}) 
  =f_{j_{1}}^{Nc_{1}}
   f_{j_{2}}^{Nc_{2}} \cdots 
   f_{j_{r}}^{Nc_{r}}u_{\lambda}^{\otimes N}. 
\end{align*}
It follows from the tensor product rule for crystals that 
\begin{align*}
& f_{j_{1}}^{Nc_{1}}
  f_{j_{2}}^{Nc_{2}} \cdots 
  f_{j_{r}}^{Nc_{r}}u_{\lambda}^{\otimes N} = \\
& 
  \bigl(f_{j_{1}}^{b_{1,1}}
  f_{j_{2}}^{b_{1,2}} \cdots 
  f_{j_{r}}^{b_{1,r}}u_{\lambda}\bigr) \otimes
  \bigl(f_{j_{1}}^{b_{2,1}}
  f_{j_{2}}^{b_{2,2}} \cdots 
  f_{j_{r}}^{b_{2,r}}u_{\lambda}\bigr) \otimes \cdots \otimes
  \bigl(f_{j_{1}}^{b_{N,1}}
  f_{j_{2}}^{b_{N,2}} \cdots 
  f_{j_{r}}^{b_{N,r}}u_{\lambda}\bigr)
\end{align*}
for some $b_{k,t} \in \BZ_{\ge 0}$, 
$1 \le k \le N$, $1 \le t \le r$, 
with $\sum_{k=1}^{N}b_{k,t}=Nc_{t}$ 
for each $1 \le t \le r$.
Combining these equalities with \eqref{eq:c1}, 
we obtain
\begin{align*}
& u_{x_{1} \cdot \lambda} \otimes 
  u_{x_{2} \cdot \lambda} \otimes \cdots \otimes 
  u_{x_{N} \cdot \lambda}=K_{N}(\Psi_{\lambda}^{-1}(P))=\\
& 
  \bigl(f_{j_{1}}^{b_{1,1}}
  f_{j_{2}}^{b_{1,2}} \cdots 
  f_{j_{r}}^{b_{1,r}}u_{\lambda}\bigr) \otimes
  \bigl(f_{j_{1}}^{b_{2,1}}
  f_{j_{2}}^{b_{2,2}} \cdots 
  f_{j_{r}}^{b_{2,r}}u_{\lambda}\bigr) \otimes \cdots \otimes
  \bigl(f_{j_{1}}^{b_{N,1}}
  f_{j_{2}}^{b_{N,2}} \cdots 
  f_{j_{r}}^{b_{N,r}}u_{\lambda}\bigr),
\end{align*}
from which it follows that 
$u_{x_{k} \cdot \lambda}=
  f_{j_{1}}^{b_{k,1}}
  f_{j_{2}}^{b_{k,2}} \cdots 
  f_{j_{r}}^{b_{k,r}}u_{\lambda}$ 
for $1 \le k \le N$. 
This implies that 
$u_{x_{k} \cdot \lambda} \in \CB_{x}(\lambda)$ 
for each $1 \le k \le N$ since 
$f_{j_{1}}^{b_{k,1}}
 f_{j_{2}}^{b_{k,2}} \cdots 
 f_{j_{r}}^{b_{k,r}}u_{\lambda} \in \CB_{x}(\lambda)$ 
by \eqref{eq:dem-f}.
Thus, we have proved Proposition~\ref{prop:N3}. 
\end{proof}
%
%
\subsection{Main results.}
\label{subsec:main}

In this subsection, we prove the following theorem, 
by using a polytopal estimate 
(Theorem~\ref{thm:tensor} below) of tensor products of 
MV polytopes. We use the setting of \S\ref{subsec:app-multiple}.
%
%
\begin{thm} \label{thm:main}
Let $\lambda \in X_{*}(T) \subset \Fh_{\BR}$ be a dominant coweight, 
and let $x \in W^{\lambda}_{\min}$. 
If a positive integer $N \in \BZ_{\ge 1}$ and 
minimal coset representatives 
$x_{1},\,x_{2},\,\dots,\,x_{N} \in W^{\lambda}_{\min}$ 
satisfy the condition \eqref{eq:N3} 
in Proposition~\ref{prop:N3}, then 
\begin{equation*}
N \cdot P \subseteq
P_{x_1 \cdot \lambda} + 
P_{x_2 \cdot \lambda} + \cdots + 
P_{x_N \cdot \lambda},
\end{equation*}
where 
$P_{x_1 \cdot \lambda} + P_{x_2 \cdot \lambda} + \cdots + 
 P_{x_N \cdot \lambda}$ is the Minkowski sum of 
the extremal MV polytopes $P_{x_1 \cdot \lambda}$, 
$P_{x_2 \cdot \lambda}$, $\dots$, $P_{x_N \cdot \lambda}$.
\end{thm}

\begin{proof}
By our assumption, we have 
\begin{equation*}
\begin{cases}
x \ge x_{k} \quad 
 \text{\rm for all $1 \le k \le N$;} \\[3mm]
K_{N}(P)=G_{N}(N \cdot P)=
 P_{x_{1} \cdot \lambda} \otimes 
 P_{x_{2} \cdot \lambda} \otimes \cdots \otimes 
 P_{x_{N} \cdot \lambda}.
\end{cases}
\end{equation*}
Here we recall from \S\ref{subsec:app-multiple} 
that $G_{N}:\mv(N\lambda) \hookrightarrow \mv(\lambda)^{\otimes N}$ 
denotes the canonical embedding of crystals that 
sends $P_{N\lambda} \in \mv(N\lambda)$ to 
$P_{\lambda}^{\otimes N} \in \mv(\lambda)^{\otimes N}$.
Therefore, by using Theorem~\ref{thm:tensor} 
(or rather, Corollary~\ref{cor:tensor}) successively, 
we can show that
\begin{equation*}
N \cdot P \subset 
 P_{x_{1} \cdot \lambda}+
 P_{x_{2} \cdot \lambda}+\cdots+
 P_{x_{N} \cdot \lambda}.
\end{equation*}
This completes the proof of Theorem~\ref{thm:main}.
\end{proof}

This theorem, together with Proposition~\ref{prop:N3}, yields 
Theorem~\ref{ithm1} in the Introduction. 
As an immediate consequence, 
we can provide an affirmative answer to 
a question posed in \cite[\S4.6]{NS-dp}.
%
%
\begin{cor} \label{cor:main}
Let $\lambda \in X_{*}(T) \subset \Fh_{\BR}$ be a dominant coweight, 
and let $x \in W$. 
All the MV polytopes lying in the Demazure crystal 
$\mv_{x}(\lambda)$ are contained (as sets) in 
the extremal MV polytope $P_{x \cdot \lambda}=
 \Conv(W_{\le x} \cdot \lambda)$ of weight $x \cdot \lambda$. 
Namely, for all $P \in \mv_{x}(\lambda)$, there holds 
\begin{equation*}
P \subset P_{x \cdot \lambda}=
 \Conv(W_{\le x} \cdot \lambda).
\end{equation*}
\end{cor}

\begin{rem}
(1) The assertion of Theorem~\ref{thm:main} 
is not obvious, as explained in 
\cite[Remark~4.6.2 and Example~4.6.3]{NS-dp}.

(2) The converse statement fails to hold; 
see \cite[Remark~4.6.1]{NS-dp}.
\end{rem}

\begin{proof}[Proof of Corollary~\ref{cor:main}]
We know from Remark~\ref{rem:demcos} that 
if $x,\,y \in W$ satisfies 
$x \cdot \lambda=y \cdot \lambda$, 
then $\mv_{x}(\lambda)=\mv_{y}(\lambda)$ and 
$P_{x \cdot \lambda}=P_{y \cdot \lambda}$. 
Hence we may assume that $x \in W^{\lambda}_{\min}$. 
Let us take an arbitrary $P \in \mv_{x}(\lambda)$. 
By Proposition~\ref{prop:N3}, 
there exist $N \in \BZ_{\ge 1}$ and 
$x_{1},\,x_{2},\dots,\,x_{N} \in W^{\lambda}_{\min}$ 
satisfying the condition \eqref{eq:N3}. 
Also, for each $1 \le k \le N$, 
it follows from Proposition~\ref{prop:ext} and 
the inequality $x \ge x_{k}$ that 
%
%
\begin{equation} \label{eq:ext-inc02}
P_{x_{k} \cdot \lambda}=
 \Conv(W_{\le x_{k}} \cdot \lambda) \subset
 \Conv(W_{\le x} \cdot \lambda)=P_{x \cdot \lambda}. 
\end{equation}
Therefore, we have 
\begin{align*}
N \cdot P 
 & \subset  
   P_{x_{1} \cdot \lambda}+
   P_{x_{2} \cdot \lambda}+\cdots+
   P_{x_{N} \cdot \lambda} 
   \quad \text{by Theorem~\ref{thm:main}} \\
 & \subset 
   \underbrace{
    P_{x \cdot \lambda}+
    P_{x \cdot \lambda}+\cdots+
    P_{x \cdot \lambda}
   }_{\text{$N$ times}} 
   \quad \text{by \eqref{eq:ext-inc02}} \\[3mm]
 & =N \cdot P_{x \cdot \lambda}
   \quad \text{by Remark~\ref{rem:N1}}.
\end{align*}
Consequently, we obtain 
$N \cdot P \subset N \cdot P_{x \cdot \lambda}$, 
which implies that $P \subset P_{x \cdot \lambda}$. 
This proves Corollary~\ref{cor:main}. 
\end{proof}

%
\section{Polytopal estimate of tensor products of MV polytopes.}
\label{sec:tensor}

The aim of this section is to state and prove 
a polytopal estimate of tensor products of MV polytopes.
%
%
\subsection{Polytopal estimate.}
\label{subsec:tensor}

As in (the second paragraph of) \S\ref{subsec:notation}, 
we assume that $\Fg$ is a complex semisimple Lie algebra. 
Let $\lambda_{1},\,\lambda_{2} \in X_{\ast}(T) \subset \Fh_{\BR}$ 
be dominant coweights. 
Because $\mv(\lambda) \cong \CB(\lambda)$ as crystals 
for every dominant coweight $\lambda \in X_{\ast}(T) \subset \Fh_{\BR}$, 
the tensor product 
$\mv(\lambda_{2}) \otimes \mv(\lambda_{1})$ of 
the crystals $\mv(\lambda_{1})$ and $\mv(\lambda_{2})$ 
decomposes into a disjoint union of connected components 
as follows:
\begin{equation*}
\mv(\lambda_{2}) \otimes \mv(\lambda_{1}) \cong 
 \bigoplus_{
   \begin{subarray}{c}
   \lambda \in X_{\ast}(T) \\[0.5mm]
   \text{$\lambda$ : dominant}
   \end{subarray}
   } 
\mv(\lambda)^{\oplus m_{\lambda_{1},\lambda_{2}}^{\lambda}},
\end{equation*}
where $m_{\lambda_{1},\lambda_{2}}^{\lambda} \in \BZ_{\ge 0}$ denotes 
the multiplicity of $\mv(\lambda)$ in 
$\mv(\lambda_{2}) \otimes \mv(\lambda_{1})$. 
For each dominant coweight $\lambda \in X_{\ast}(T) \subset \Fh_{\BR}$ 
such that $m_{\lambda_{1},\lambda_{2}}^{\lambda} \ge 1$, 
we take (and fix) an arbitrary embedding 
$\iota_{\lambda}:\mv(\lambda) \hookrightarrow 
 \mv(\lambda_{2}) \otimes \mv(\lambda_{1})$ of crystals 
that maps $\mv(\lambda)$ onto a connected component of 
$\mv(\lambda_{2}) \otimes \mv(\lambda_{1})$, 
which is isomorphic to $\mv(\lambda)$ as a crystal. 
%
%
\begin{thm} \label{thm:tensor}
Keep the notation above.
Let $P \in \mv(\lambda)$, and 
write $\iota_{\lambda}(P) \in 
\mv(\lambda_{2}) \otimes \mv(\lambda_{1})$ as\,{\rm:} 
$\iota_{\lambda}(P)=P_{2} \otimes P_{1}$ for some 
$P_{1} \in \mv(\lambda_{1})$ and 
$P_{2} \in \mv(\lambda_{2})$. 
We assume that the MV polytope 
$P_{2} \in \mv(\lambda_{2})$ is an extremal MV polytope 
$P_{x \cdot \lambda_{2}}$ for some $x \in W$. 
Then, we have
%
%
\begin{equation} \label{eq:Minkowski}
P \subset P_{1} + P_{2},
\end{equation}
where $P_{1} + P_{2}$ is the Minkowski sum of 
the MV polytopes $P_{1} \in \mv(\lambda_{1})$ and 
$P_{2} \in \mv(\lambda_{2})$.
\end{thm}

\begin{rem}
It should be mentioned that in the theorem above, 
the $\iota_{\lambda}(P)$ may lie in an arbitrary 
connected component of 
$\mv(\lambda_{2}) \otimes \mv(\lambda_{1})$ 
that is isomorphic to $\mv(\lambda)$ as a crystal. 
\end{rem}

The proof of this theorem will be given 
in \S\ref{subsec:prf-tensor} below; 
it seems likely that this theorem still holds 
without the assumption of extremality 
on the MV polytope $P_{2} \in \mv(\lambda_{2})$.

For dominant coweights 
$\lambda_{1},\,\lambda_{2} \in X_{\ast}(T) \subset \Fh_{\BR}$, 
there exists a unique embedding
\begin{equation*}
\iota_{\lambda_{1},\lambda_{2}}:
 \mv(\lambda_{1}+\lambda_{2}) \hookrightarrow 
 \mv(\lambda_{1}) \otimes \mv(\lambda_{2})
\end{equation*}
of crystals, which maps $\mv(\lambda_{1}+\lambda_{2})$ 
onto the unique connected component of 
$\mv(\lambda_{1}) \otimes \mv(\lambda_{2})$ 
(called the Cartan component) 
that is isomorphic to $\mv(\lambda_{1}+\lambda_{2})$ 
as a crystal; note that 
$\iota_{\lambda_{1},\lambda_{2}}(P_{\lambda_{1}+\lambda_{2}})=
 P_{\lambda_{1}} \otimes P_{\lambda_{2}}$. 
Applying Theorem~\ref{thm:tensor} to the case 
$\lambda=\lambda_{1}+\lambda_{2}$, we obtain 
the following corollary; notice that the ordering of the 
tensor factors $\mv(\lambda_{1})$, $\mv(\lambda_{2})$ is reversed. 
%
%
\begin{cor} \label{cor:tensor}
Let $\lambda_{1},\,\lambda_{2} \in X_{\ast}(T) 
\subset \Fh_{\BR}$ be dominant coweights. 
Let $P \in \mv(\lambda_{1}+\lambda_{2})$, and 
write $\iota_{\lambda_{1},\lambda_{2}}(P) \in 
\mv(\lambda_{1}) \otimes \mv(\lambda_{2})$ as\,{\rm:} 
$\iota_{\lambda_{1},\lambda_{2}}(P)=P_{1} \otimes P_{2}$ for some 
$P_{1} \in \mv(\lambda_{1})$ and 
$P_{2} \in \mv(\lambda_{2})$. 
We assume that the MV polytope 
$P_{1} \in \mv(\lambda_{1})$ is an extremal MV polytope 
$P_{x \cdot \lambda_{1}}$ for some $x \in W$. 
Then, we have
%
%
\begin{equation} \label{eq:Minkowski2}
P \subset P_{1} + P_{2}. 
\end{equation}
\end{cor}

The following is a particular case in which 
the equality holds in \eqref{eq:Minkowski2}. 
%
%
\begin{prop} \label{prop:Min-Ext}
Let $\lambda_{1},\,\lambda_{2} \in X_{\ast}(T) \subset \Fh_{\BR}$ 
be dominant coweights, and let $x \in W$. Then, 
$\iota_{\lambda_{1},\lambda_{2}}(P_{x \cdot (\lambda_{1}+\lambda_{2})})=
P_{x \cdot \lambda_{1}} \otimes P_{x \cdot \lambda_{2}}$, and 
$P_{x \cdot (\lambda_{1}+\lambda_{2})}=
 P_{x \cdot \lambda_{1}}+P_{x \cdot \lambda_{2}}$. 
\end{prop}

\begin{proof}
First we show the equality 
$\iota_{\lambda_{1},\lambda_{2}}(P_{x \cdot (\lambda_{1}+\lambda_{2})})=
 P_{x \cdot \lambda_{1}} \otimes P_{x \cdot \lambda_{2}}$ 
(which may be well-known to experts) by induction on $\ell(x)$. 
If $x=e$, then we have 
$\iota_{\lambda_{1},\lambda_{2}}(P_{\lambda_{1}+\lambda_{2}})=
 P_{\lambda_{1}} \otimes P_{\lambda_{2}}$ as mentioned above. 
Assume now that $\ell(x) > 0$. We take $j \in I$ 
such that $\ell(s_{j}x) < \ell(x)$. Set
\begin{equation*}
k:=\pair{s_{j}x \cdot (\lambda_{1}+\lambda_{2})}{\alpha_{j}}, \quad
k_{1}:=\pair{s_{j}x \cdot \lambda_{1}}{\alpha_{j}}, \quad
k_{2}:=\pair{s_{j}x \cdot \lambda_{2}}{\alpha_{j}};
\end{equation*}
note that we have $k=k_{1}+k_{2}$, 
with $k_{1},\,k_{2},\,k \in \BZ_{\ge 0}$, 
since $\ell(s_{j}x) < \ell(x)$. We see from 
\cite[Lemme~8.3.1]{Kasb} that 
$f_{j}^{k}P_{s_{j}x \cdot (\lambda_{1}+\lambda_{2})}$
is equal to $P_{x \cdot (\lambda_{1}+\lambda_{2})}$. 
Hence we have 
\begin{align*}
\iota_{\lambda_{1},\lambda_{2}}
 (P_{x \cdot (\lambda_{1}+\lambda_{2})}) & 
 =\iota_{\lambda_{1},\lambda_{2}}
 (f_{j}^{k}P_{s_{j}x \cdot (\lambda_{1}+\lambda_{2})}) 
 =f_{j}^{k}\iota_{\lambda_{1},\lambda_{2}}
 (P_{s_{j}x \cdot (\lambda_{1}+\lambda_{2})}) \\
& = f_{j}^{k}(
  P_{s_{j}x \cdot \lambda_{1}} \otimes 
  P_{s_{j}x \cdot \lambda_{2}})
\quad \text{by the induction hypothesis}.
\end{align*}
Here, by the tensor product rule 
for crystals, 
\begin{equation*}
f_{j}^{k}(
  P_{s_{j}x \cdot \lambda_{1}} \otimes 
  P_{s_{j}x \cdot \lambda_{2}}
)
= (f_{j}^{l_1}P_{s_{j}x \cdot \lambda_{1}}) \otimes 
  (f_{j}^{l_2}P_{s_{j}x \cdot \lambda_{2}})
\end{equation*}
for some $l_{1},\,l_{2} \in \BZ_{\ge 0}$ with $k=l_{1}+l_{2}$.
It follows from \cite[Lemme~8.3.1]{Kasb} that 
$l_{1}=k_{1}$ and $l_{2}=k_{2}$. 
Therefore, we deduce that 
\begin{align*}
\iota_{\lambda_{1},\lambda_{2}}
 (P_{x \cdot (\lambda_{1}+\lambda_{2})})
& = 
f_{j}^{k}(
  P_{s_{j}x \cdot \lambda_{1}} \otimes 
  P_{s_{j}x \cdot \lambda_{2}})
= (f_{j}^{k_1}P_{s_{j}x \cdot \lambda_{1}}) \otimes 
  (f_{j}^{k_2}P_{s_{j}x \cdot \lambda_{2}}) \\
& = 
P_{x \cdot \lambda_{1}} \otimes P_{x \cdot \lambda_{2}}
\quad \text{by \cite[Lemme~8.3.1]{Kasb}}.
\end{align*}
This proves the first equality. 

Next we show the equality
$P_{x \cdot (\lambda_{1}+\lambda_{2})}=
 P_{x \cdot \lambda_{1}}+P_{x \cdot \lambda_{2}}$.
Let us denote by
\begin{equation*}
\mu_{\bullet}^{x \cdot (\lambda_{1}+\lambda_{2})}=
 (\mu_{w}^{x \cdot (\lambda_{1}+\lambda_{2})})_{w \in W}, \quad
\mu_{\bullet}^{x \cdot \lambda_{1}}=
 (\mu_{w}^{x \cdot \lambda_{1}})_{w \in W}, \quad \text{and} \quad
\mu_{\bullet}^{x \cdot \lambda_{2}}=
 (\mu_{w}^{x \cdot \lambda_{2}})_{w \in W}
\end{equation*}
the GGMS data of the extremal MV polytopes 
$P_{x \cdot (\lambda_{1}+\lambda_{2})} \in \mv(\lambda_{1}+\lambda_{2})$, 
$P_{x \cdot \lambda_{1}} \in \mv(\lambda_{1})$, and 
$P_{x \cdot \lambda_{2}} \in \mv(\lambda_{2})$, respectively.
We verify the equality 
%
%
\begin{equation} \label{eq:me1}
\mu_{w}^{x \cdot (\lambda_{1}+\lambda_{2})} = 
\mu_{w}^{x \cdot \lambda_{1}}+\mu_{w}^{x \cdot \lambda_{2}}
\quad \text{for every $w \in W$}.
\end{equation}
Let $w \in W$, and 
take $\bi \in R(w_{0})$ such that 
$w=\wi{l}$ for some $0 \le l \le m$. 
Then it follows from 
Theorem~\ref{thm:GGMS-Ext} that 
\begin{align*}
& \mu_{w}^{x \cdot (\lambda_{1}+\lambda_{2})}=
  \mu_{\wi{l}}^{x \cdot (\lambda_{1}+\lambda_{2})}=
  \yi{l} \cdot (\lambda_{1}+\lambda_{2}), \quad \text{and} \\
& \mu_{w}^{x \cdot \lambda_{1}}=
  \mu_{\wi{l}}^{x \cdot \lambda_{1}}=
  \yi{l} \cdot \lambda_{1}, \qquad
  \mu_{w}^{x \cdot \lambda_{2}}=
  \mu_{\wi{l}}^{x \cdot \lambda_{2}}=
  \yi{l} \cdot \lambda_{2},
\end{align*}
where $\yi{l}$ is defined as \eqref{eq:yi}; 
recall from Remark~\ref{rem:zi} that 
$\yi{l} \in W$ does not depend 
on the dominant coweights $\lambda_{1}+\lambda_{2}$, 
$\lambda_{1}$, and $\lambda_{2}$. 
Therefore, we deduce that 
\begin{equation*}
\mu_{w}^{x \cdot (\lambda_{1}+\lambda_{2})} = 
\yi{l} \cdot (\lambda_{1}+\lambda_{2})=
\yi{l} \cdot \lambda_{1}+\yi{l} \cdot \lambda_{2}=
\mu_{w}^{x \cdot \lambda_{1}}+\mu_{w}^{x \cdot \lambda_{2}},
\end{equation*}
as desired. Hence it follows that 
%
%
\begin{equation} \label{eq:GGMS-Ext}
\mu_{\bullet}^{x \cdot (\lambda_{1}+\lambda_{2})}=
 (\mu_{w}^{x \cdot (\lambda_{1}+\lambda_{2})})_{w \in W}=
(\mu_{w}^{x \cdot \lambda_{1}}+\mu_{w}^{x \cdot \lambda_{2}})_{w \in W}=
\mu_{\bullet}^{x \cdot \lambda_{1}}+
\mu_{\bullet}^{x \cdot \lambda_{2}}.
\end{equation}
Consequently, we have 
\begin{align*}
P_{x \cdot \lambda_{1}}+
P_{x \cdot \lambda_{2}} & =
P(\mu_{\bullet}^{x \cdot \lambda_{1}})+
P(\mu_{\bullet}^{x \cdot \lambda_{2}}) = 
P(\mu_{\bullet}^{x \cdot \lambda_{1}}+
  \mu_{\bullet}^{x \cdot \lambda_{2}}) \qquad
\text{by Proposition~\ref{prop:Minkowski}} \\
& =
P(\mu_{\bullet}^{x \cdot (\lambda_{1}+\lambda_{2})}) \qquad
\text{by \eqref{eq:GGMS-Ext}} \\
& = P_{x \cdot (\lambda_{1}+\lambda_{2})}.
\end{align*}
This proves the second equality, 
thereby completes the proof of the proposition. 
\end{proof}

%
\subsection{Reformation of Braverman-Gaitsgory's result on tensor products.}
\label{subsec:BG}
In this subsection, we revisit results of Braverman-Gaitsgory on 
tensor products of highest weight crystals, and provide 
a reformation of it, which is needed in our proof of 
Theorem~\ref{thm:tensor} given in \S\ref{subsec:prf-tensor}. 

We now recall the construction of 
certain twisted product varieties. 
Let $G$ be a complex, connected, reductive algebraic group 
as in (the beginning of) \S\ref{subsec:geom}. 
The twisted product variety $\Gr \twp \Gr$ is defined to be 
the quotient space
\begin{equation*}
G(\CK) \times^{G(\CO)} \Gr = \bigl(G(\CK) \times \Gr\bigr)/\sim, 
\end{equation*}
where $\sim$ is an equivalence relation on 
$G(\CK) \times \Gr$ given by: 
$(a,\,bG (\CO)) \sim (ah^{-1},\,hb G(\CO))$ for 
$a,\,b \in G(\CK)$, and $h \in G (\CO)$; 
for $(a,\,bG(\CO)) \in G(\CK) \times \Gr$, we denote by 
$[(a,\,bG(\CO))]$ the equivalence class of $(a,\,bG(\CO))$. 
This variety can be thought of as a fibration over $\Gr$ 
(the first factor) with its typical fiber $\Gr$ (the second factor). 
Let $\pi_{1} : \Gr \twp \Gr \twoheadrightarrow \Gr$, 
$[(a,\,bG(\CO))] \mapsto a G(\CO)$, 
be the projection onto the first factor, and 
$\mul:\Gr \twp \Gr \rightarrow \Gr$, 
$[(a,\,bG(\CO))] \mapsto ab G(\CO)$, the multiplication map. 
If $\CX \subset \Gr$ is an algebraic subvariety and 
$\CY \subset \Gr$ is an algebraic subvariety that 
is stable under the left $G(\CO)$-action on $\Gr$, 
then we can form an algebraic subvariety
\begin{equation*}
\CX \twp \CY := \ti{\CX} \times^{G(\CO)} \CY \subset \Gr \twp \Gr,
\end{equation*}
where $\ti{\CX}:=\pi_{1}^{-1}(\CX)$ is the pullback of 
$\CX \subset \Gr=G(\CK)/G(\CO)$ to $G(\CK)$.
%
%
\begin{prop}[{\cite{Lu-Ast}, \cite[Lemma~4.4]{MV2}}] \label{prop:mv-44}
The multiplication map 
$\mul:\Gr \twp \Gr \rightarrow \Gr$, when restricted to 
$\ol{\Gr^{\lambda_{1}}} \twp \ol{\Gr^{\lambda_{2}}}$ 
for $\lambda_{1},\,\lambda_{2} \in X_{*}(T)_{+}$, is 
projective, birational, and semi-small 
with respect to the stratification by 
$G(\CO)$-orbits. In particular, 
for $\lambda_{1},\,\lambda_{2} \in X_{*}(T)_{+}$ and 
$\lambda \in X_{*}(T)_{+}$, 
\begin{equation*}
\mul \left(\ol{\Gr^{\lambda_{1}}} \twp \ol{\Gr^{\lambda_{2}}} \right) = 
\ol{\Gr^{\lambda_{1} + \lambda_{2}}},
\end{equation*}
\begin{equation*}
\mul^{-1} (\Gr^{\lambda}) \cap 
\left(
 \ol{\Gr^{\lambda_{1}}} \twp 
 \ol{\Gr^{\lambda_{2}}}\right) \neq \emptyset
\quad \text{\rm if and only if} \quad
\lambda_{1} + \lambda_{2} \ge \lambda;
\end{equation*}
if $\lambda_{1} + \lambda_{2} \ge \lambda$, then 
\begin{equation*}
\dim \left( 
 \mul^{-1} (\Gr^{\lambda}) \cap 
 \left(\ol{\Gr^{\lambda_{1}}} \twp \ol{\Gr^{\lambda_{2}}}\right)
\right) 
\le \dim \Gr^{\lambda} + 
\pair{\lambda_{1}+\lambda_{2}-\lambda}{\rho}.
\end{equation*}
\end{prop}

Let us introduce another kind of twisted products. 
For each $\nu_{1},\,\nu_{2} \in X_{*}(T)$ and $w \in W$, 
we define $S^{w}_{\nu_{1},\,\nu_{2}}$ to be the quotient space
\begin{equation*}
{}^{w}U(\CK) t^{\nu_{1}} \times^{{}^{w}U(\CO)} {}^{w}U(\CK)[t^{\nu_{2}}]
=
\bigl(
 {}^{w}U(\CK) t^{\nu_{1}} \times {}^{w}U(\CK)[t^{\nu_{2}}]
\bigr)/\sim,
\end{equation*}
where $\sim$ is an equivalence relation on 
${}^{w}U(\CK) t^{\nu_{1}} \times {}^{w}U(\CK)[t^{\nu_{2}}]$ 
given by: 
$(a,\,bG (\CO)) \sim (au^{-1},\,ub G(\CO))$ for 
$a \in {}^{w}U(\CK) t^{\nu_{1}} \subset G(\CK)$, 
$b \in {}^{w}U(\CK)[t^{\nu_{2}}] \subset \Gr$, and 
$u \in {}^{w}U(\CO) \subset G(\CO)$. 
Since ${}^{w}U(\CO)=G(\CO) \cap {}^{w}U(\CK)$ and 
$t^{\nu} ({}^{w}U(\CK)) = ({}^{w}U(\CK)) t^{\nu}$ 
for $\nu \in X_{*}(T)$, we have a canonical embedding
%
%
\begin{equation} \label{eq:emb-s}
S^{w}_{\nu_{1},\,\nu_{2}} \hookrightarrow 
G(\CK) \times^{G(\CO)} \Gr=\Gr \twp \Gr.
\end{equation}
%
%
\begin{lem} \label{lem:invS}
For each $\nu \in X_{*}(T)$ and $w \in W$, we have
\begin{equation*}
\mul^{-1} (S^{w}_{\nu}) = 
 \bigsqcup_{
   \begin{subarray}{c}
   \nu_{1},\,\nu_{2} \in X_{*}(T) \\[1.5mm]
   \nu_{1} + \nu_{2} = \nu
   \end{subarray}
 } S^{w}_{\nu_{1},\,\nu_{2}}
\end{equation*}
under the canonical embedding \eqref{eq:emb-s}.
\end{lem}

\begin{proof}
It follows from the Iwasawa decomposition
$\Gr=\bigsqcup_{\nu_{1} \in X_{*}(T)} S_{\nu_{1}}^{w}$ that 
\begin{equation*}
\Gr \twp \Gr=
\bigsqcup_{\nu_{1} \in X_{*}(T)} 
 \pi_{1}^{-1}(S_{\nu_{1}}^{w}),
\end{equation*}
where $\pi_{1}^{-1}(S_{\nu_{1}}^{w})=
 {}^{w}U(\CK)t^{\nu_{1}}G(\CO) \times^{G(\CO)}\Gr$. 
Therefore it suffices to show that
\begin{equation*}
\pi_{1}^{-1}(S_{\nu_{1}}^{w}) \cap 
\mul^{-1} (S^{w}_{\nu}) =
S^{w}_{\nu_{1},\,\nu-\nu_{1}}
\quad
\text{for each $\nu_{1} \in X_{*}(T)$}.
\end{equation*}
Now, for each $\nu_{1} \in X_{*}(T)$, 
let us take $[y] \in 
\pi_{1}^{-1}(S_{\nu_{1}}^{w}) \cap 
\mul^{-1} (S^{w}_{\nu})$, 
where $y \in 
 {}^{w}U(\CK)t^{\nu_{1}}G(\CO) \times \Gr$, and 
write it as: $y=(u_{1}t^{\nu_{1}}g_{1},\,g_{2}G(\CO))$ 
for $u_{1} \in {}^{w}U(\CK)$, $g_{1} \in G(\CO)$, and 
$g_{2} \in G(\CK)$. 
Since $\mul([y])=u_{1}t^{\nu_{1}}g_{1}g_{2}G(\CO) \in S^{w}_{\nu}$, 
we have 
\begin{equation*}
g_{1}g_{2}G(\CO) \in 
 (u_{1}t^{\nu_{1}})^{-1}S^{w}_{\nu}=
 {}^{w}U(\CK)t^{\nu-\nu_{1}}G(\CO).
\end{equation*}
Consequently, using the equivalence relation $\sim$ 
on $G(\CK) \times \Gr$, we see that
\begin{equation*}
[y]
 =\bigl[ (u_{1}t^{\nu_{1}}g_{1},\,g_{2}G(\CO)) \bigr]
 =\bigl[ (u_{1}t^{\nu_{1}},\,g_{1}g_{2}G(\CO)) \bigr]
 =\bigl[ (u_{1}t^{\nu_{1}},\,u_{2}t^{\nu-\nu_{1}}G(\CO)) \bigr]
\end{equation*}
for some $u_{2} \in {}^{w}U(\CK)$. 
This implies that 
\begin{equation*}
[y] \in 
  {}^{w}U(\CK)t^{\nu_{1}} 
  \times^{{}^{w}U(\CO)} 
  {}^{w}U(\CK)[t^{\nu-\nu_{1}}]=
S^{w}_{\nu_{1},\,\nu-\nu_{1}},
\end{equation*}
and hence 
$\pi_{1}^{-1}(S_{\nu_{1}}^{w}) \cap 
 \mul^{-1} (S^{w}_{\nu}) \subset 
S^{w}_{\nu_{1},\,\nu-\nu_{1}}$.
The opposite inclusion is obvious. 
Thus, we obtain 
$\pi_{1}^{-1}(S_{\nu_{1}}^{w}) \cap 
 \mul^{-1} (S^{w}_{\nu})=
 S^{w}_{\nu_{1},\,\nu-\nu_{1}}$.
This proves the lemma. 
\end{proof}

For $\nu_{1},\,\nu_{2} \in X_{*}(T)$, we set
$S_{\nu_{1},\,\nu_{2}} := S_{\nu_{1},\,\nu_{2}}^{e}$. 
If we take (and fix) an element $t \in T(\BR)$ such that
\begin{equation*}
\lim_{k \to \infty} \Ad(t^{k}) u = e
\quad \text{for all $u \in U$}, 
\end{equation*}
then we have (by \cite[Eq.(3.5)]{MV2})
\begin{equation*}
S_{\nu}=
 \left\{
 [y] \in \Gr \ \biggm| \ 
 \lim_{k \to \infty} t^{k}[y]=[t^{\nu}]
 \right\}
\quad \text{for $\nu \in X_{*}(T)$}.
\end{equation*}
From this, by using Lemma~\ref{lem:invS}, 
we have
%
%
\begin{equation} \label{eq:Snn}
S_{\nu_{1},\,\nu_{2}} = 
 \left\{ [y] \in \Gr \twp \Gr \ \biggm| \ 
 \lim_{k \to \infty} t^{k}[y]=
 (t^{\nu_{1}},\,[t^{\nu_{2}}])
 \right\}
\quad \text{for $\nu_{1},\,\nu_{2} \in X_{*}(T)$};
\end{equation}
in particular, these strata of $\Gr \twp \Gr$ 
are simply-connected. 

For each $\nu \in X_{*}(T)$ and $w \in W$, 
we set
\begin{equation*}
\BS^{w}_{\nu} := {}^{w}U(\CK)t^{\nu}({}^{w}U(\CO)) / {}^{w}U(\CO), 
\end{equation*}
which is canonically isomorphic to 
${}^{w}U(\CK) t^{\nu} G (\CO) / G (\CO) = S^{w}_{\nu}$ 
since ${}^{w}U(\CK) \cap G(\CO)={}^{w}U(\CO)$; 
note that ${}^{w}U(\CK)t^{\nu}({}^{w}U(\CO))={}^{w}U(\CK)t^{\nu}$
since $t^{\nu}({}^{w}U(\CK))=({}^{w}U(\CK))t^{\nu}$. 
Also, for a subset $\CX \subset \Gr$, 
we define the intersection $\CX \cap \BS^{w}_{\nu}$ 
to be the image of $\CX \cap S^{w}_{\nu} \subset S_{\nu}^{w}$ 
under the identification $\BS^{w}_{\nu} = S^{w}_{\nu}$ above. 

In the sequel, for an algebraic variety $\CX$, 
we denote by $\Irr(\CX)$ the set of irreducible components of $\CX$.
Let $\lambda \in X_{*}(T)_{+}$ and $\nu \in X_{*}(T)$ be 
such that $\Gr^{\lambda} \cap S_{\nu} \ne \emptyset$, 
i.e., $\nu \in \Omega(\lambda)$; note that 
$\Gr^{\lambda} \cap S_{\nu} \ne \emptyset$ if and only if 
$\Gr^{\lambda} \cap S_{w^{-1} \cdot \nu} \ne \emptyset$, i.e., 
$w^{-1} \cdot \nu \in \Omega(\lambda)$ for each $w \in W$. 
Let us take an arbitrary $w \in W$. 
Because $\Gr^{\lambda} \cap S_{\nu}^{w}=
\dot{w}(\Gr^{\lambda} \cap S_{w^{-1} \cdot \nu})$ 
for $w \in W$, we have a bijection
\begin{equation*}
\Irr\left(\ol{\Gr^{\lambda} \cap S_{w^{-1} \cdot \nu}}\right)
\rightarrow
\Irr\left(\ol{\Gr^{\lambda} \cap S_{\nu}^{w}}\right), \quad
\bb \mapsto \dot{w}\bb. 
\end{equation*}
Thus, each element of 
$\Irr\left(\ol{\Gr^{\lambda} \cap S_{\nu}^{w}}\right)$ 
can be written in the form $\dot{w}\bb$ for a unique 
MV cycle $\bb \in \CZ(\lambda)_{w^{-1} \cdot \nu}=
\Irr\left(\ol{\Gr^{\lambda} \cap S_{w^{-1} \cdot \nu}}\right)$.
The variety $\bb^{w}$ defined by \eqref{eq:bbw} is a special case of 
such elements, in which $\bb \in \CZ(\lambda)$ is an extremal 
MV cycle with GGMS datum $\mu_{\bullet}$ and $\nu=\mu_{w}$. 
%
%
\begin{lem} \label{lem:stable}
With the notation as above, 
let us take an arbitrary element 
\begin{equation*}
\dot{w}\bb \in 
\Irr\left(\ol{\Gr^{\lambda} \cap S_{\nu}^{w}}\right), \quad
\text{\rm where} \quad 
\bb \in \CZ(\lambda)_{w^{-1} \cdot \nu}.
\end{equation*} 
Then, the intersection $\dot{w}\bb \cap S_{\gamma}^{w}$ 
(and hence $\dot{w}\bb \cap \BS_{\gamma}^{w}$) is 
stable under the action of ${}^{w}U(\CO)$ 
for all $\gamma \in X_{*}(T)$. 
\end{lem}

\begin{proof}
By definition, each MV cycle is an irreducible component of 
the (Zariski-) closure of the intersection of a $G(\CO)$-orbit 
and a $U(\CK)$-orbit. Since $U(\CO)=G(\CO) \cap U(\CK)$ is 
connected, such an irreducible component is stable under 
the action of $U(\CO)$. Therefore, the variety $\dot{w}\bb$ 
(and hence its intersection with an arbitrary ${}^{w}U(\CK)$-orbit) 
is stable under the action of ${}^{w}U(\CO)$. 
This proves the lemma. 
\end{proof}

Let $\lambda_{1},\,\lambda_{2} \in X_{*}(T)_{+}$ and 
$\nu_{1},\,\nu_{2} \in X_{*}(T)$ be such that 
$\Gr^{\lambda_{i}} \cap S_{\nu_{i}} \ne \emptyset$ 
for $i=1,\,2$, and let $w \in W$. 
Let us take $\bb_{1} \in \CZ(\lambda_{1})_{\nu_{1}}$, 
$\dot{w}\bb_{2} \in 
\Irr\left(\ol{\Gr^{\lambda_{2}} \cap S_{\nu_{2}}^{w}}\right)$, 
and $\gamma_{1},\,\gamma_{2} \in X_{*}(T)$. 
Then, by virtue of the lemma above, 
we can form the twisted product 
\begin{equation*}
(\bb_{1} \cap \BS^{w}_{\gamma_{1}})^{\sim}
 \times^{{}^{w}U(\CO)} 
(\dot{w}\bb_{2} \cap \BS^{w}_{\gamma_{2}}) 
\subset 
{}^{w}U(\CK)t^{\gamma_{1}} 
 \times^{{}^{w}U(\CO)} S^{w}_{\gamma_{2}}=
S^{w}_{\gamma_{1},\,\gamma_{2}},
\end{equation*}
where $(\bb_{1} \cap \BS^{w}_{\gamma_{1}})^{\sim}$
denotes the pullback of 
\begin{equation*}
\bb_{1} \cap \BS^{w}_{\gamma_{1}} \subset \BS^{w}_{\gamma_{1}}=
{}^{w}U(\CK)t^{\gamma_{1}}({}^{w}U(\CO))/{}^{w}U(\CO)
\end{equation*}
to ${}^{w}U(\CK)t^{\gamma_{1}}({}^{w}U(\CO))=
{}^{w}U(\CK)t^{\gamma_{1}} \subset G(\CK)$. 
By $\bb_{1} \str{w}{\gamma_{1}}{\gamma_{2}} \dot{w}\bb_{2}$, 
we denote the image of this algebraic subvariety of 
$\Gr \twp \Gr$ under the map 
$\mul:\Gr \twp \Gr \rightarrow \Gr$; note that 
$\bb_{1} \str{w}{\gamma_{1}}{\gamma_{2}} \dot{w}\bb_{2} 
 \subset \mul(S^{w}_{\gamma_{1},\,\gamma_{2}})
 =S^{w}_{\gamma_{1}+\gamma_{2}}$. 

The following is a reformulation of 
Braverman-Gaitsgory's result on tensor products of 
highest weight crystals in \cite{BrGa} (see also \cite{BFG});
we will give a brief account of the relationship 
in the Appendix. Here we should warn the reader that 
the convention on the tensor product rule for crystals 
in \cite{BrGa} is opposite to ours, i.e., to that of Kashiwara 
(see, for example, \cite{Kasoc} and \cite{Kasb}). 
%
%
\begin{thm} \label{thm:BG}
Let $\lambda_{1},\,\lambda_{2} \in X_{*}(T)_{+}$. 
There exists a bijection
\begin{equation*}
\Phi_{\lambda_{1},\,\lambda_{2}} : 
 \mv (\lambda_{1}) \times \mv (\lambda_{2})
\rightarrow
 \bigsqcup_{\nu \in X_{*}(T)} 
   \Irr \left(
      \ol{\mul^{-1} (S_{\nu}) \cap 
      \bigl(\Gr^{\lambda_{1}} \twp \Gr^{\lambda_{2}}}\bigr) 
   \right)
\end{equation*}
given as follows\,{\rm:}
for $P_{1} \in \mv (\lambda_{1})$ and $P_{2} \in \mv(\lambda_{2})$, 
%
%
\begin{equation} \label{eq:Phi-ll}
\Phi_{\lambda_{1},\,\lambda_{2}}
 (P_{1},\,P_{2})=
\ol{
 \bigl(\Phi_{\lambda_{1}}(P_{1}) \cap \BS_{\nu_{1}}\bigr)^{\sim}
 \times^{U(\CO)}
 \bigl(\Phi_{\lambda_{2}}(P_{2}) \cap \BS_{\nu_{2}}\bigr)
}, 
\end{equation}
where $\nu_{1}:=\wt (P_{1})$, $\nu_{2}:=\wt(P_{2})$, and 
$\bigl(\Phi_{\lambda_{1}}(P_{1}) \cap \BS_{\nu_{1}}\bigr)^{\sim}$ denotes 
the pullback of $\Phi_{\lambda_{1}}(P_{1}) \cap \BS_{\nu_{1}} \subset S_{\nu_{1}}$ 
to $U(\CK)t^{\nu_{1}} \subset G(\CK)$. 
Moreover, the bijection $\Phi_{\lambda_{1},\,\lambda_{2}}$ 
has the following properties. 

{\rm (i)} 
For each $P_{1} \in \mv(\lambda_{1})$ and 
$P_{2} \in \mv(\lambda_{2})$, 
the image of 
$\Phi_{\lambda_{1},\,\lambda_{2}}(P_{1},\,P_{2})$ 
under the map $\mul$ is equal to $\Phi_{\lambda}(P)$ 
for a unique $\lambda \in X_{*}(T)_{+}$ and 
a unique $P \in \mv(\lambda)$ such that 
$\iota_{\lambda}(P)=P_{2} \otimes P_{1}$, 
where $\iota_{\lambda}:\mv(\lambda) \hookrightarrow 
 \mv(\lambda_{2}) \otimes \mv(\lambda_{1})$ is 
an embedding of crystals\,{\rm;}

{\rm (ii)}
$\pi_{1} (\Phi_{\lambda_{1},\,\lambda_{2}} (P_{1},\,P_{2})) = 
 \Phi_{\lambda_{1}} (P_{1})$ 
for each $P_{1} \in \mv (\lambda_{1})$ and 
$P_{2} \in \mv (\lambda_{2})$\,{\rm;}

{\rm (iii)} 
$[(t^{\nu_{1}},\,\Phi_{\lambda_{2}} (P_{2}))] \subset 
 \Phi_{\lambda_{1},\,\lambda_{2}} (P_{1},\,P_{2})$ 
for each $P_{1} \in \mv (\lambda_{1})$ with $\nu_{1}=\wt (P_{1})$ 
and $P_{2} \in \mv (\lambda_{2})$. 
\end{thm}

%
\subsection{Proof of the polytopal estimate.}
\label{subsec:prf-tensor}

This subsection is devoted to the proof of Theorem~\ref{thm:tensor}. 
Let $G$ be a complex, connected, semisimple algebraic group 
with Lie algebra $\Fg$. We keep the setting of \S\ref{subsec:tensor}. 
Let $\mu^{(1)}_{\bullet}=(\mu^{(1)}_{w})_{w \in W}$ and 
$\mu^{(2)}_{\bullet}=(\mu^{(2)}_{w})_{w \in W}$ be the GGMS data of 
$P_{1} \in \mv(\lambda_{1})$ and $P_{2} \in \mv(\lambda_{2})$, 
respectively. Also, let $\mu_{\bullet}=(\mu_{w})_{w \in W}$ be 
the GGMS datum of $P \in \mv(\lambda)$; 
note that
\begin{equation*}
\mu_{e}=\wt(P)=
\wt(P_{1})+\wt(P_{2})=
\mu^{(1)}_{e}+\mu^{(1)}_{e}.
\end{equation*}
Recall from Proposition~\ref{prop:Minkowski} that 
the Minkowski sum $P_{1}+P_{2}$ is a pseudo-Weyl polytope 
$P(\mu^{(1)}_{\bullet}+\mu^{(2)}_{\bullet})$ with GGMS datum 
$\mu^{(1)}_{\bullet}+\mu^{(1)}_{\bullet}
 =(\mu^{(1)}_{w}+\mu^{(2)}_{w})_{w \in W}$. 
Therefore, it follows from equation \eqref{eq:poly} 
together with Remark~\ref{rem:moment} that 
\begin{align*}
P_{1}+P_{2} & = 
\bigcap_{w \in W}
 \bigl\{ 
 v \in \Fh_{\BR} \mid 
 w^{-1} \cdot v- w^{-1} \cdot \bigl(\mu^{(1)}_{w}+\mu^{(2)}_{w}\bigr) 
 \in \textstyle{\sum_{j \in I}\BR_{\ge 0}h_{j}}
 \bigr\} \\[3mm]
& = 
\bigcap_{w \in W}
 \Conv \biggl\{
  \gamma \in X_{*}(T) \subset \Fh_{\BR} \ \biggm| \ 
  [t^{\gamma}] \in \ol{S^{w}_{\mu^{(1)}_{w}+\mu^{(2)}_{w}}}
\biggr\}.
\end{align*}
Also, recall from Theorem~\ref{thm:Kam1} that 
$\Phi_{\lambda}(P) \in \CZ(\lambda)$ and 
\begin{equation*}
P=\Conv \bigl\{
  \gamma \in X_{*}(T) \subset \Fh_{\BR} \mid 
  [t^{\gamma}] \in \Phi_{\lambda}(P)
\bigr\}.
\end{equation*}
Hence, in order to prove that $P \subset P_{1}+P_{2}$, 
it suffices to show that
%
%
\begin{equation} \label{eq:goal}
\Phi_{\lambda} (P) \subset 
 \ol{S^{w}_{\mu_{w}^{(1)} + \mu_{w}^{(2)}}}
\quad \text{for all $w \in W$}.
\end{equation}

We set $\bb^{(1)} := \Phi_{\lambda_{1}} (P_{1}) \in \CZ(\lambda_{1})$ 
and $\bb^{(2)} := \Phi_{\lambda_{2}} (P_{2}) \in \CZ(\lambda_{2})$.
Because $P_{2}=P(\mu^{(2)}_{\bullet})$ is the extremal MV polytope 
of weight $x \cdot \lambda$ for some $x \in W$ by our assumption, 
we know from Theorem~\ref{thm:GGMS-Ext} that 
$\mu^{(2)}_{w} \in W \cdot \lambda_{2}$ for all $w \in W$. 
Hence the algebraic variety $\bb^{(2),w}:=
\ol{\Gr^{\lambda_{2}} \cap S^{w}_{\mu^{(2)}_{w}}}$ is 
irreducible, and is the $\dot{w}$-translate of 
the extremal MV cycle $\bb_{w^{-1} \cdot \mu^{(2)}_{w}}$ 
of weight $w^{-1} \cdot \mu^{(2)}_{w}$ 
(see Remark~\ref{rem:extcyc}); note that 
$\bb^{(2),e}=\bb^{(2)}$. 

Now suppose, contrary to our assertion \eqref{eq:goal}, 
that 
$\Phi_{\lambda} (P) \not\subset 
 \ol{S^{w}_{\mu_{w}^{(1)} + \mu_{w}^{(2)}}}$
for some $w \in W$; we take and fix such a $w \in W$. 
Then, by equation \eqref{eq:Snuw} (see also Remark~\ref{rem:Snuw}), 
there exists some $\nu \in X_{*}(T)$ such that
%
%
\begin{equation} \label{eq:prf1}
\Phi_{\lambda}(P) \cap S_{\nu}^{w} \ne \emptyset
\quad \text{and} \quad
w^{-1} \cdot \nu \not\ge 
w^{-1} \cdot (\mu^{(1)}_{w} + \mu^{(2)}_{w}). 
\end{equation}

\begin{claim*}
For the (fixed) $w \in W$ above, 
we have the following inclusion of varieties when 
they are regarded as subvarieties of $\Gr$\,{\rm:}
\begin{equation*}
\bb^{(1)} \stra{e} \bb^{(2)}
\subset 
\bb^{(1)} \stra{w} \bb^{(2),w}.
\end{equation*}
\end{claim*}

\noindent
{\it Proof of Claim.} 
Let $w = s_{i_{1}} s_{i_{2}} \cdots s_{i_{\ell}}$ be 
a reduced expression, and set 
$w_{k} = s_{i_{1}} s_{i_{2}} \cdots s_{i_{k}}$ 
for $0 \le k \le \ell$. 
For simplicity of notation, we set for $0 \le k \le \ell$
\begin{equation*}
\bb^{(2)}_{k}:=\bb^{(2),w_{k}}, \qquad
\bb^{(1)} \star_{k} \bb^{(2)}:=
\bb^{(1)} \stra{w_{k}} \bb^{(2), w_{k}}, 
\quad \text{and}
\end{equation*}
\begin{equation*}
\mu^{(1)}_{k}:=\mu^{(1)}_{w_{k}}, \quad
\mu^{(2)}_{k}:=\mu^{(2)}_{w_{k}}; 
\end{equation*}
note that
\begin{equation*}
\bb^{(1)} \star_{0} \bb^{(2)}=
\bb^{(1)} \stra{e} \bb^{(2),e}=
\bb^{(1)} \stra{e} \bb^{(2)}. 
\end{equation*}
If we can show the inclusion
%
%
\begin{equation} \label{eq:incl}
\bb^{(1)} \star_{k} \bb^{(2)} \subset 
\ol{ \bb^{(1)} \star_{k+1} \bb^{(2)} }
\quad \text{for each $0 \le k \le \ell-1$},
\end{equation}
then we will obtain the following sequence of 
inclusions: 
\begin{align*}
& \bb^{(1)} \stra{e} \bb^{(2)}=
  \bb^{(1)} \star_{0} \bb^{(2)} \subset 
  \ol{ \bb^{(1)} \star_{1} \bb^{(2)} } \subset 
  \ol{ \bb^{(1)} \star_{2} \bb^{(2)} } \subset \cdots \\
& \hspace*{20mm} \cdots \subset 
\ol{ \bb^{(1)} \star_{\ell} \bb^{(2)} } =
\ol{ \bb^{(1)} \stra{w_{\ell}} \bb^{(2),w_{\ell}} } =
\ol{ \bb^{(1)} \stra{w} \bb^{(2),w} },
\end{align*}
as desired. In order to show the inclusion \eqref{eq:incl}, 
take an element 
$[(y,\,gG(\CO))] \in \bb^{(1)} \star_{k} \bb^{(2)}$, 
where 
\begin{equation*}
y \in 
 (\bb^{(1)} \cap \BS^{w_{k}}_{\mu^{(1)}_{k}})^{\sim}
 \subset {}^{w_{k}}U(\CK)t^{\mu^{(1)}_{k}}, \qquad
gG(\CO) \in 
  \bb^{(2)}_{k} \cap S^{w_{k}}_{\mu^{(2)}_{k}} \cong 
  \bb^{(2)}_{k} \cap \BS^{w_{k}}_{\mu^{(2)}_{k}}, 
\end{equation*}
and write the element $y \in {}^{w_{k}}U(\CK)t^{\mu^{(1)}_{k}}$ 
as: $y =u_{k}t^{\mu^{(1)}_{k}}$ for $u_{k} \in {}^{w_{k}}U(\CK)$. 
Since $\bb^{(1)}=\ol{\bigcap_{z \in W} S^{z}_{\mu^{(1)}_{z}}}$ 
by Theorem~\ref{thm:Kam1}, we may (and do) assume that 
$yG(\CO) \in S^{w_{k}}_{\mu^{(1)}_{k}} \cap 
 S^{w_{k+1}}_{\mu^{(1)}_{w_{k+1}}}$ to show 
the inclusion \eqref{eq:incl}. Therefore, we can take 
$u_{k+1} \in {}^{w_{k+1}}U(\CK)$ and $g_{k+1} \in G(\CO)$
such that $y=u_{k}t^{\mu^{(1)}_{k}}=
u_{k+1}t^{\mu^{(1)}_{k+1}}g_{k+1}$; note that 
\begin{equation*}
g_{k+1} \in 
 T(\CK)({}^{w_{k+1}}U(\CK))({}^{w_{k}}U(\CK))T(\CK).
\end{equation*}
Here, since $w_{k+1}=w_{k}s_{i_{k+1}}$ by definition, 
it follows that
\begin{equation*}
T(\CK)({}^{w_{k+1}}U(\CK))({}^{w_{k}}U(\CK))T(\CK)
 = \dot{w}_{k}
 \bigl(
  \dot{s}_{i_{k+1}} U(\CK)\dot{s}_{i_{k+1}}U(\CK)
 \bigr) \dot{w}_{k}^{-1}
 \subset {}^{w_{k}}P_{i_{k+1}}(\CK),
\end{equation*}
and hence that $g_{k+1} \in {}^{w_{k}}P_{i_{k+1}}(\CK)$. 
Moreover, since $g_{k+1} \in G(\CO)$, we get 
$g_{k+1} \in {}^{w_{k}}P_{i_{k+1}}(\CK) \cap G(\CO)=
{}^{w_{k}}P_{i_{k+1}}(\CO)$. Therefore, we obtain 
\begin{equation*}
g_{k+1}\bb^{(2)}_{k} \subset 
{}^{w_{k}}P_{i_{k+1}}(\CO)\bb^{(2)}_{k}=
{}^{w_{k}}L_{i_{k+1}}(\CO)
{}^{w_{k}}U_{i_{k+1}}(\CO)\bb^{(2)}_{k}. 
\end{equation*}
Since the extremal MV cycle $\dot{w}_{k}^{-1}\bb^{(2)}_{k}$ is stable under 
$U_{i_{k+1}}(\CO) \subset U(\CO)=G(\CO) \cap U(\CK)$ 
(see the proof of Lemma~\ref{lem:stable}), we have
${}^{w_{k}}U_{i_{k+1}}(\CO)\bb^{(2)}_{k} \subset \bb^{(2)}_{k}$, 
and hence
\begin{equation*}
g_{k+1}\bb^{(2)}_{k} \subset 
{}^{w_{k}}L_{i_{k+1}}(\CO)\bb^{(2)}_{k}. 
\end{equation*}
Also, we see that 
\begin{align*}
{}^{w_{k}}L_{i_{k+1}}(\CO)
\bb^{(2)}_{k} 
& ={}^{w_{k+1}}L_{i_{k+1}}(\CO)\bb^{(2)}_{k}
  \quad \text{since $w_{k+1}=w_{k}s_{i_{k+1}}$ and 
  $\dot{s}_{i_{k+1}} \in L_{i_{k+1}}$} \\[3mm]
& \subset \bb^{(2)}_{k+1}
  \quad \text{by Lemma~\ref{lem:tran} 
  since $w_{k} < w_{k}s_{i_{k+1}}=w_{k+1}$}.
\end{align*}
Combining these, we obtain 
$g_{k+1}\bb^{(2)}_{k} \subset \bb^{(2)}_{k+1}$, 
and hence $g_{k+1}gG(\CO) \in g_{k+1}\bb^{(2)}_{k} 
\subset \bb^{(2)}_{k+1}$. 
Furthermore, we have 
\begin{align*}
[(y,\,gG(\CO))] & =
[(u_{k}t^{\mu^{(1)}_{k}},\,gG(\CO))]=
[(u_{k+1}t^{\mu^{(1)}_{k+1}}g_{k+1},\,gG(\CO))] \\ 
& = [(u_{k+1}t^{\mu^{(1)}_{k+1}},\,g_{k+1}gG(\CO))]
\end{align*}
by the equivalence relation 
$\sim$ on $G(\CK) \times \Gr$, and 
\begin{equation*}
yG(\CO)=u_{k+1}t^{\mu^{(1)}_{k+1}}G(\CO) 
 \in \bb^{(1)} \cap S^{w_{k+1}}_{\mu^{(1)}_{k+1}}
\end{equation*}
by the choice above of $y$. 
Consequently, we conclude that 
\begin{align*}
[(y,\,gG(\CO))] & \in 
\ol{ \mul \Bigl(
     \bigl(\bb^{(1)} \cap \BS^{w_{k+1}}_{ \mu^{(1)}_{k+1} }\bigr)^{\sim}
     \times^{{}^{w_{k+1}}U(\CO)} 
     \bb^{(2)}_{k+1}
     \Bigr)
} \\[3mm]
& = 
\ol{ \mul \Bigl(
     \bigl(\bb^{(1)} \cap \BS^{w_{k+1}}_{ \mu^{(1)}_{k+1} }\bigr)^{\sim}
     \times^{{}^{w_{k+1}}U(\CO)} 
     \bigl(\bb^{(2)}_{k+1} \cap \BS^{w_{k+1}}_{ \mu^{(2)}_{k+1} }\bigr)
     \Bigr) 
} \\[3mm]
& = \ol{ \bb^{(1)} \star_{k+1} \bb^{(2)} },
\end{align*}
since $\ol{ \bb^{(2)}_{k+1} \cap S^{w_{k+1}}_{ \mu^{(2)}_{k+1} } }=
\bb^{(2)}_{k+1}$ (see Remark~\ref{rem:extcyc}).
This proves the inclusion \eqref{eq:incl}, 
and hence the claim. \bqed

\vsp

Finally, we complete the proof of Theorem~\ref{thm:tensor}. 
By the claim above, we obtain
%
%
\begin{equation} \label{eq:prf2}
S^{w}_{\nu} \cap 
\bigl(\ol{\bb^{(1)} \stra{e} \bb^{(2)}}\bigr)
\subset 
S^{w}_{\nu} \cap 
\bigl(\ol{\bb^{(1)} \stra{w} \bb^{(2),w}}\bigr). 
\end{equation}
Also, we have 
%
%
\begin{equation} \label{eq:prf4}
S^{w}_{\nu} \cap 
\bigl(\ol{ \bb^{(1)} \stra{w} \bb^{(2),w} }\bigr)
\subset 	
S^{w}_{\nu} \cap 
\ol{S^{w}_{\mu^{(1)}_{w}+\mu^{(2)}_{w}}}
\end{equation}
by the inclusion $\bb^{(1)} \stra{w} \bb^{(2),w} \subset 
S^{w}_{\mu^{(1)}_{w}+\mu^{(2)}_{w}}$, which is 
an immediate consequence of the definition. 
Since $w^{-1} \cdot \nu \not\ge w^{-1} \cdot 
(\mu^{(1)}_{w}+\mu^{(2)}_{w})$ by \eqref{eq:prf1}, 
it follows from equation \eqref{eq:Snuw} that
$S^{w}_{\nu} \cap 
\ol{S^{w}_{\mu^{(1)}_{w}+\mu^{(2)}_{w}}}=\emptyset$,
and hence by \eqref{eq:prf2} together with \eqref{eq:prf4} that 
%
%
\begin{equation} \label{eq:prf3}
S^{w}_{\nu} \cap 
\bigl(\ol{\bb^{(1)} \stra{e} \bb^{(2)}}\bigr)=\emptyset.
\end{equation}

Now, we know from Theorem~\ref{thm:BG} that
\begin{equation*}
S_{\mu^{(1)}_{e}+\mu^{(2)}_{e}} \cap 
\mul(\Phi_{\lambda_{1},\,\lambda_{2}}(P_{1},\,P_{2})) 
\subset \Phi_{\lambda}(P)
\end{equation*}
is an open dense subset, where 
$\wt(P)=\mu^{(1)}_{e}+\mu^{(2)}_{e}$; 
for an MV polytope $P \in \mv(\lambda)$, choosing 
an embedding $\iota_{\lambda}:\mv(\lambda) \hookrightarrow 
\mv(\lambda_{2}) \otimes \mv(\lambda_{1})$ of crystals 
so that $\iota_{\lambda}(P)=P_{2} \otimes P_{1}$ 
for some $P_{1} \in \mv(\lambda_{1})$ and 
$P_{2} \in \mv(\lambda_{2})$ corresponds, 
via Theorems~\ref{thm:Kam1} and \ref{thm:BG}, to choosing 
an irreducible component 
$\CX \in \Irr(\Gr^{\lambda_{1},\,\lambda_{2}} \cap S_{\nu_{1},\,\nu_{2}})$ 
such that $\Phi_{\lambda}(P) \in \CZ(\lambda)$ is the (Zariski-) closure 
of the image of $\CX$ under the map $\star$ 
in the commutative diagram of Theorem~\ref{thm:BGc}
in the Appendix, where $\nu_{1}=\wt(P_{1})=\mu^{(1)}_{e}$ and 
$\nu_{2}=\wt(P_{2})=\mu^{(2)}_{e}$.
Also, we see from the explicit construction of 
$\Phi_{\lambda_{1},\,\lambda_{2}}(P_{1},\,P_{2})$ 
given in Theorem~\ref{thm:BG} that 
\begin{equation*}
\bb^{(1)} \stra{e} \bb^{(2)} 
\subset
S_{\mu^{(1)}_{e}+\mu^{(2)}_{e}} \cap 
\mul(\Phi_{\lambda_{1},\,\lambda_{2}}(P_{1},\,P_{2}))
\end{equation*}
is an open dense subset. Therefore, 
$\bb^{(1)} \stra{e} \bb^{(2)} \subset \Phi_{\lambda}(P)$ 
is an open dense subset. 
Combining this fact with \eqref{eq:prf3}, we conclude that 
$\Phi_{\lambda}(P) \cap S^{w}_{\nu}=\emptyset$, 
which contradicts \eqref{eq:prf1}. Thus, we have completed 
the proof of Theorem~\ref{thm:tensor}. 

%
\appendix
\section{Appendix: On Braverman-Gaitsgory's bijection.}
\label{subsec:BGB}

\setcounter{subsection}{1}
\setcounter{equation}{0}

The aim of this appendix is to explain 
why Theorem~\ref{thm:BG} is a reformulation of results on 
tensor products of highest weight crystals in \cite{BrGa}. 
Keep the setting of \S\ref{subsec:BG}. 
In addition, we generally follow 
the notation of \cite[Chapter~8]{ChGi}.

For $\lambda \in X_{*}(T)_{+}$, 
we denote by $\IC_{\lambda}$ 
the intersection cohomology complex of $\ol{\Gr^{\lambda}}$ 
(and also its extension by zero to the whole $\Gr$). 
Similarly, for $\lambda_{1},\,\lambda_{2} \in X_{*}(T)_{+}$, 
we denote by $\IC_{\lambda_{1},\,\lambda_{2}}$ 
the intersection cohomology complex of 
$\ol{\Gr^{\lambda_{1}}} \twp \ol{\Gr^{\lambda_{2}}}$ 
(and also its extension by zero to the whole $\Gr \twp \Gr$).
%
%
\begin{thm}[\cite{Lu-Ast}, \cite{Gin}, \cite{MV2}, \cite{BeDr}] \label{thm:LD}
{\rm (1)} 
Let $\lambda \in X_{\ast}(T)_{+}$. Then, 
%
%
\begin{equation} \label{eq:fib}
 H^{\bullet} (\IC_{\lambda}) \cong L (\lambda) \qquad
 \text{\rm as $G^{\vee}$-modules}.
\end{equation}

{\rm (2)} 
Let $\lambda_{1},\,\lambda_{2} \in X_{\ast}(T)_{+}$.
The direct image $\bR^{\bullet} \mul_{*} 
\IC_{\lambda_{1},\,\lambda_{2}}$ is isomorphic 
to a direct sum of simple perverse sheaves as follows\,{\rm:}
%
%
\begin{equation} \label{eq:conv2}
\bR^{\bullet} \mul_{*} \IC_{\lambda_{1},\,\lambda_{2}} \cong 
 \bigoplus_{\lambda \in X_{*}(T)_{+}} 
   \BC^{m^{\lambda}_{\lambda_{1},\,\lambda_{2}}} \boxtimes 
   \IC_{\lambda},
\end{equation}
where $\BC^{m^{\lambda}_{\lambda_{1},\,\lambda_{2}}}$ denotes 
the vector space of dimension 
$m^{\lambda}_{\lambda_{1},\,\lambda_{2}} \in \BZ_{\ge 0}$ 
over $\BC$ for $\lambda \in X_{*}(T)_{+}$.

{\rm (3)} Let $\CP$ be the full subcategory of 
the category of perverse sheaves on $\Gr$ whose objects are 
direct sums of $\IC_{\lambda}$'s. Then, the assignment
\begin{equation*}
(\IC_{\lambda_{1}},\,\IC_{\lambda_{2}}) \mapsto 
 \IC_{\lambda_{1},\,\lambda_{2}} \mapsto 
 \bR^{\bullet} \mul_{*} \IC_{\lambda_{1},\,\lambda_{2}}
\end{equation*}
defines on $\CP$ the structure of a tensor category 
with a fiber functor $\IC_{\lambda} \mapsto H^{\bullet} (\IC_{\lambda})$.

{\rm (4)} We have an equivalence 
$\CP \cong \Rep(G^{\vee})$ of tensor categories with fiber functors, 
where $\Rep(G^{\vee})$ denotes the category of finite-dimensional 
rational representations of $G^{\vee}$.
\end{thm}

Let us fix arbitrarily 
$\lambda_{1},\,\lambda_{2} \in X_{*}(T)_{+}$, and 
$\nu_{1},\,\nu_{2} \in X_{*}(T)$. 
In the sequel, we assume that 
$\Gr^{\lambda_{i}} \cap S_{\nu_{i}} \ne \emptyset$ for $i=1,\,2$, 
i.e., $\nu_{1},\,\nu_{2} \in X_{*}(T)$ are weights of 
the $G^{\vee}$-modules $L(\lambda_{1})$ and $L(\lambda_{2})$, 
respectively; namely, we assume that $\nu_{i} \in \Omega(\lambda_{i})$ 
for $i=1,\,2$ in the notation of \S\ref{subsec:geom}. 
For simplicity of notation, we set
\begin{equation*}
S_{\nu_{1},\,\nu_{2}} := S_{\nu_{1},\,\nu_{2}}^{e}, \quad
\Gr^{\lambda_{1},\,\lambda_{2}}:=
\Gr^{\lambda_{1}} \twp \Gr^{\lambda_{2}}.
\end{equation*}
The following Lemma is essentially due to 
Ngo and Polo \cite[Corollary~9.2]{NP}; 
the assumption of this corollary can be dropped 
by using \cite[Lemme~9.1]{NP} along with 
\cite[Theorem~3.2\,(i)]{MV2}. 
%
%
\begin{lem} \label{lem:dimint}
Keep the setting above. 
The intersection
$\Gr^{\lambda_{1},\,\lambda_{2}} \cap S_{\nu_{1},\,\nu_{2}}$
is a union of irreducible components of dimension 
$\pair{\lambda_{1}+\lambda_{2}-(\nu_{1}+\nu_{2})}{\rho}$.
\end{lem}

By virtue of \eqref{eq:Snn} and Lemma~\ref{lem:dimint}, 
we can imitate all the constructions in 
\cite[Theorem~3.5 and Proposition~3.10]{MV2} to obtain:
%
%
\begin{thm}[{cf. \cite[\S3]{MV2}; see also \cite{BrGa}}] \label{thm:MVBG}
Keep the setting above. 

{\rm (1)}
The cohomology group 
$H^{k}_{c} (S_{\nu_{1},\,\nu_{2}},\,\IC_{\lambda_{1},\,\lambda_{2}})$ 
vanishes except for $k=-2\pair{\nu_{1}+\nu_{2}}{\rho}$. 

{\rm (2)} 
There is an isomorphism of vector spaces
\begin{equation*}
H^{-2\pair{\nu_{1}+\nu_{2}}{\rho}}_{c} 
(S_{\nu_{1},\,\nu_{2}},\,\IC_{\lambda_{1},\,\lambda_{2}}) 
\cong 
\BC \Irr \bigl(
 \Gr^{\lambda_{1},\,\lambda_{2}} \cap S_{\nu_{1},\,\nu_{2}}
 \bigr).
\end{equation*}
\end{thm}

The following result is implicit 
in \cite[\S8]{A} and \cite{BrGa}.
%
%
\begin{thm} \label{thm:BGc}
Keep the setting above. 
For each $\nu \in X_{*}(T)$, 
there exists the following commutative diagram\,{\rm:}

\newcommand{\diagv}{%
${\displaystyle\bigoplus_{
  \nu_{1}+\nu_{2}=\nu }
  H^{-2\pair{\nu}{\rho}}_{c} 
  (S_{\nu_{1},\,\nu_{2}},\,\IC_{\lambda_{1},\,\lambda_{2}})}$
}

\newcommand{\diagw}{%
${\displaystyle\bigoplus_{
  \nu_{1}+\nu_{2}=\nu}
 \BC \Irr 
 \bigl(%
   \Gr^{\lambda_{1},\,\lambda_{2}} \cap S_{\nu_{1},\,\nu_2}
 \bigr)}$
}

\newcommand{\diagx}{%
$H^{-2\pair{\nu}{\rho}}_{c}
(S_{\nu},\,\bR^{\bullet}\mul_{!}\IC_{\lambda_{1},\,\lambda_{2}})$}

\newcommand{\diagy}{%
$\disp{\bigoplus_{\lambda \in X_{*}(T)_{+}}}
 H^{-2\pair{\nu}{\rho}}_{c} (S_{\nu},\,\IC_{\lambda})$.}
\newcommand{\diagz}{%
$\disp{\bigoplus_{\lambda \in X_{*}(T)_{+}}}
 \BC\Irr (\Gr^{\lambda} \cap S_{\nu})$}

\vspace*{5mm}

\hspace*{15mm}
\unitlength 0.1in
\begin{picture}( 42.7500, 27.9500)( 11.3000,-33.1500)
\put(14.0000,-9.0000){\makebox(0,0){\diagv}}%
\put(14.0000,-22.0000){\makebox(0,0){\diagx}}%
\put(34.0000,-34.0000){\makebox(0,0){\diagy}}%
\put(54.0000,-22.0000){\makebox(0,0){\diagz}}%
\put(54.0000,-9.0000){\makebox(0,0){\diagw}}%
%
\special{pn 8}%
\special{pa 1400 1200}%
\special{pa 1400 2000}%
\special{fp}%
\special{sh 1}%
\special{pa 1400 2000}%
\special{pa 1420 1934}%
\special{pa 1400 1948}%
\special{pa 1380 1934}%
\special{pa 1400 2000}%
\special{fp}%
%
\special{pn 8}%
\special{pa 5400 1200}%
\special{pa 5400 2000}%
\special{fp}%
\special{sh 1}%
\special{pa 5400 2000}%
\special{pa 5420 1934}%
\special{pa 5400 1948}%
\special{pa 5380 1934}%
\special{pa 5400 2000}%
\special{fp}%
%
\special{pn 8}%
\special{pa 1400 2400}%
\special{pa 2800 3100}%
\special{fp}%
\special{sh 1}%
\special{pa 2800 3100}%
\special{pa 2750 3052}%
\special{pa 2752 3076}%
\special{pa 2732 3088}%
\special{pa 2800 3100}%
\special{fp}%
%
\special{pn 8}%
\special{pa 4000 3100}%
\special{pa 5400 2400}%
\special{fp}%
\special{sh 1}%
\special{pa 5400 2400}%
\special{pa 5332 2412}%
\special{pa 5352 2424}%
\special{pa 5350 2448}%
\special{pa 5400 2400}%
\special{fp}%
\put(16.0000,-16.0000){\makebox(0,0){$\cong$}}%
\put(34.0000,-6.0500){\makebox(0,0){$\cong$}}%
\put(56.0000,-16.0000){\makebox(0,0){$\star$}}%
\put(20.0000,-28.0000){\makebox(0,0)[rt]{$\tr$}}%
\put(48.0000,-28.0000){\makebox(0,0)[lt]{$\cong$}}%
%
\special{pn 8}%
\special{ar 3400 2000 600 600  2.1587989 6.2831853}%
\special{ar 3400 2000 600 600  0.0000000 0.9827937}%
%
\special{pn 8}%
\special{pa 3746 2492}%
\special{pa 3734 2500}%
\special{fp}%
\special{sh 1}%
\special{pa 3734 2500}%
\special{pa 3800 2480}%
\special{pa 3778 2470}%
\special{pa 3778 2446}%
\special{pa 3734 2500}%
\special{fp}%
%
\special{pn 8}%
\special{pa 2600 806}%
\special{pa 4200 806}%
\special{fp}%
\special{sh 1}%
\special{pa 4200 806}%
\special{pa 4134 786}%
\special{pa 4148 806}%
\special{pa 4134 826}%
\special{pa 4200 806}%
\special{fp}%
\end{picture}%

\vspace*{7mm}

\noindent
Here, $\tr$ is the map obtained by summing up 
the isotypical component of 
$\IC_{\lambda}$ for each $\lambda \in X_{*}(T)_{+}$ inside 
$\mul_{!}\IC_{\lambda_{1},\,\lambda_{2}}$ 
via Theorem~\ref{thm:LD}\,(2), and 
$\star$ is the map induced by $\mul$.
\end{thm}

\begin{proof}[(Sketch of) Proof]
We write simply $\ud{\BC}$ for 
the constant sheaf $\BC_{\CX}$ of an algebraic variety $\CX$. 
Consider the restriction of the multiplication map 
$\mul: \Gr \twp \Gr \rightarrow \Gr$ (see Proposition~\ref{prop:mv-44}) 
to $\Gr^{\lambda_{1},\,\lambda_{2}}$, and write it as: 
\begin{equation*}
\mul':\Gr^{\lambda_{1},\,\lambda_{2}} \rightarrow \ol{\Gr^{\lambda_{1}+\lambda_{2}}}.
\end{equation*}
Then, there exists the Leray spectral sequence
%
%
\begin{equation} \label{eq:leray}
E_{2}^{p,q} := 
 H^{q}_{c} 
 (S_{\nu} \cap \ol{\Gr^{\lambda_{1}+\lambda_{2}}},\,\bR^{p}\mul_{!}'\ud{\BC}) 
 \hspace{7pt} \Rightarrow 
 \bigoplus_{\nu_{1}+\nu_{2}=\nu}
 H^{q+p}_{c}(
  S_{\nu_{1},\,\nu_{2}} \cap \Gr^{\lambda_{1},\,\lambda_{2}},\,\ud{\BC}).
\end{equation}
Here, from Theorem~\ref{thm:MVBG}\,(2) 
together with Theorem~\ref{thm:LD}\,(3), 
we see that each direct summand of the right-hand side of \eqref{eq:leray}, 
with $q+p=2\pair{\lambda_{1}+\lambda_{2}-\nu}{\rho}$, is isomorphic to 
the tensor product of the $\nu_{1}$-weight space of $L(\lambda_{1})$ 
and the $\nu_{2}$-weight space of $L(\lambda_{2})$. 
Also, note that the map $\mul'$ above is a $G(\CO)$-equivalent fibration. 
Hence, by virtue of the simply-connectedness of $\Gr^{\lambda}$
for $\lambda \in X_{\ast}(T)_{+}$, we conclude that the proper direct 
image $R^{\bullet}\mul_{!}'\ud{\BC}$, when restricted to $\Gr^{\lambda}$
with $\lambda \le \lambda_{1}+\lambda_{2}$, 
decomposes into a direct sum of constant sheaves (with degree shifts). 
Moreover, we know from \cite[\S3.4]{BrGa} that the number of irreducible components 
of (generic) fiber of an element of $\CZ(\lambda)_{\nu}$ with its dimension 
$\pair{\lambda_{1}+\lambda_{2}-\lambda}{\rho}$ along $\mul'$ is 
$m^{\lambda}_{\lambda_{1},\,\lambda_{2}}$ for $\lambda \in X_{\ast}(T)_{+}$
(see Theorem~\ref{thm:LD}\,(2)); 
note that this is the largest possible dimension of 
an irreducible component of (generic) fiber of an element 
of $\CZ(\lambda)_{\nu}$ by Lemma~\ref{lem:dimint}.
Combining all these, we deduce that
\begin{equation*}
\sum_{p+q=2\pair{\lambda_{1}+\lambda_{2}-\nu}{\rho}} \dim E_{2}^{p,q}
=
\sum_{\nu_{1}+\nu_{2}=\nu}
\dim H^{2\pair{\lambda_{1}+\lambda_{2}-\nu}{\rho}}_{c}(
  S_{\nu_{1},\,\nu_{2}} \cap \Gr^{\lambda_{1},\,\lambda_{2}},\,\ud{\BC}).
\end{equation*}
As a result, the spectral sequence $(E_{r}^{p,q})$ stabilizes at $E_{2}$-terms
when $q+p=2\pair{\lambda_{1}+\lambda_{2}-\nu}{\rho}$ (top degree cohomology 
groups). In addition, the stalk 
$(\bR^{p} \mul_{!}' \ud{\BC})_{x}$ at $x \in \Gr^{\lambda}$ 
for $\lambda \in X_{*}(T)_{+}$ vanishes whenever
$p > 2 \pair{\lambda_{1}+\lambda_{2} - \lambda}{\rho}$ 
since the inequality
\begin{equation*}
\pair{\lambda_{1}+\lambda_{2} - \lambda}{\rho} \ge 
\dim \left( 
 \mul^{-1} (x) \cap 
 \left(\ol{\Gr^{\lambda_{1}}} \twp \ol{\Gr^{\lambda_{2}}}\right)
\right)
\end{equation*}
holds by Proposition~\ref{prop:mv-44}. It follows that
\begin{equation*}
H^{q}_{c} 
(S_{\nu} \cap \ol{\Gr^{\lambda_{1}+\lambda_{2}}},\,
 \bR^{p} \mul_{!}' \ud{\BC}) \cong
\bigoplus_{
 \begin{subarray}{c}
 \lambda \in X_{*}(T)_{+} \\[0.5mm]
 2 \pair{\lambda_{1}+\lambda_{2} - \lambda}{\rho} = p
 \end{subarray}} 
H^{q}_{c} 
(S_{\nu} \cap \Gr^{\lambda},\,
 \bR^{p} \mul_{!}' \ud{\BC})
\end{equation*}
for each $p,\,q \in \BZ$ with 
$p = 2 \pair{\lambda_{1}+\lambda_{2} - \nu}{\rho} - q$. 
Consequently, again by comparing the dimensions, 
we get an isomorphism
%
%
\begin{equation} \label{eq:cohom}
\bigoplus_{\lambda \in X_{*}(T)_{+}} 
H^{2 \pair{\lambda - \nu}{\rho}}_{c} 
(S_{\nu} \cap \Gr^{\lambda},\,
 \bR^{n_{\lambda}} \mul_{!}' \ud{\BC})
\stackrel{\sim}{\rightarrow} 
\bigoplus_{\nu_{1}+\nu_{2}=\nu}
H^{2 \pair{\lambda_{1}+\lambda_{2} - \nu}{\rho}}_{c} 
(S_{\nu_{1},\,\nu_{2}} \cap 
 \Gr^{\lambda_{1},\,\lambda_{2}},\,\ud{\BC}),
\end{equation}
where $n_{\lambda} := 2 \pair{\lambda_{1}+\lambda_{2} - \lambda}{\rho}$. 

Now we note that 
the stalk $(\bR^{n_{\lambda}} \mul_{!}' \ud{\BC})_{x}$ 
at $x \in \Gr^{\lambda}$ for $\lambda \in X_{*}(T)_{+}$ 
admits a basis corresponding to 
top-dimensional irreducible components of 
$\mul^{-1} (x) \cap \Gr^{\lambda_{1},\,\lambda_{2}}$. 
If we normalize each class of such components to $1 \in \BC$, 
then we get
%
%
\begin{equation} \label{eq:cohom2}
H^{2 \pair{\lambda - \nu}{\rho}}_{c} 
 (S_{\nu} \cap \Gr^{\lambda},\, \bR^{n_{\lambda}} \mul_{!}' \ud{\BC}) 
\stackrel{\tr}{\rightarrow} 
H^{2 \pair{\lambda - \nu}{\rho}}_{c} 
 (S_{\nu} \cap \Gr^{\lambda},\,\ud{\BC})
\end{equation}
for $\lambda \in X_{*}(T)_{+}$. 

By putting together the maps $\tr$ in \eqref{eq:cohom2}, 
Theorem~\ref{thm:MVBG}, \cite[Theorem~3.5]{MV2}, and 
the isomorphism \eqref{eq:cohom}, we obtain 
a commutative diagram

\vsp

\newcommand{\diaga}{%
${\displaystyle\bigoplus_{
  \nu_{1}+\nu_{2}=\nu }
 H^{2\pair{\lambda_{1}+\lambda_{2}-\nu}{\rho}}_{c} 
 (\Gr^{\lambda_{1},\,\lambda_{2}} \cap S_{\nu_{1},\,\nu_{2}},\,\ud{\BC})}$}
\newcommand{\diagb}{%
${\displaystyle\bigoplus_{
  \nu_{1}+\nu_{2}=\nu }
 \BC \Irr (\Gr^{\lambda_{1},\,\lambda_{2}} \cap S_{\nu_{1},\,\nu_{2}})}$}
\newcommand{\diagc}{%
$\disp{\bigoplus_{\lambda \in X_{*}(T)_{+}}} 
 H^{2 \pair{\lambda-\nu}{\rho}}_{c} 
 (S_{\nu} \cap \Gr^{\lambda},\,\bR^{n_{\lambda}} \mul_{!}' \ud{\BC})$}
\newcommand{\diagd}{%
$\disp{\bigoplus_{\lambda \in X_{*}(T)_{+}}} 
 \BC\Irr (\Gr^{\lambda} \cap S_{\nu})$}
\newcommand{\diage}{%
$\disp{\bigoplus_{\lambda \in X_{*}(T)_{+}}}
 H^{2 \pair{\lambda - \nu}{\rho}}_{c} 
 (S_{\nu} \cap \Gr^{\lambda},\,\ud{\BC})$,}

\vspace*{7mm}

\hspace*{30mm}
\unitlength 0.1in
\begin{picture}( 34.7500, 21.1000)(  7.3000,-27.0500)
\put(10.0000,-7.0000){\makebox(0,0){\diaga}}%
\put(12.0000,-17.9000){\makebox(0,0){\diagc}}%
\put(42.0000,-17.9000){\makebox(0,0){\diagd}}%
\put(44.0000,-7.0000){\makebox(0,0){\diagb}}%
\put(42.0000,-27.9000){\makebox(0,0){\diage}}%
%
\special{pn 8}%
\special{pa 4200 1000}%
\special{pa 4200 1500}%
\special{fp}%
\special{sh 1}%
\special{pa 4200 1500}%
\special{pa 4220 1434}%
\special{pa 4200 1448}%
\special{pa 4180 1434}%
\special{pa 4200 1500}%
\special{fp}%
%
\special{pn 8}%
\special{pa 1200 2100}%
\special{pa 3200 2500}%
\special{fp}%
\special{sh 1}%
\special{pa 3200 2500}%
\special{pa 3140 2468}%
\special{pa 3148 2490}%
\special{pa 3132 2508}%
\special{pa 3200 2500}%
\special{fp}%
\put(14.0000,-12.0000){\makebox(0,0){$\cong$}}%
\put(44.0000,-12.0000){\makebox(0,0){$\star$}}%
\put(44.0000,-22.0000){\makebox(0,0){$\cl$}}%
\put(29.0000,-8.0000){\makebox(0,0){$\cl$}}%
\put(22.0000,-24.0000){\makebox(0,0)[rt]{$\tr$}}%
%
\special{pn 8}%
\special{pa 1200 1500}%
\special{pa 1200 1000}%
\special{fp}%
\special{sh 1}%
\special{pa 1200 1000}%
\special{pa 1180 1068}%
\special{pa 1200 1054}%
\special{pa 1220 1068}%
\special{pa 1200 1000}%
\special{fp}%
%
\special{pn 8}%
\special{pa 4200 2500}%
\special{pa 4200 2000}%
\special{fp}%
\special{sh 1}%
\special{pa 4200 2000}%
\special{pa 4180 2068}%
\special{pa 4200 2054}%
\special{pa 4220 2068}%
\special{pa 4200 2000}%
\special{fp}%
%
\special{pn 8}%
\special{ar 2900 1600 300 300  2.1587989 6.2831853}%
\special{ar 2900 1600 300 300  0.0000000 0.9827937}%
%
\special{pn 8}%
\special{pa 3080 1842}%
\special{pa 3066 1850}%
\special{fp}%
\special{sh 1}%
\special{pa 3066 1850}%
\special{pa 3132 1828}%
\special{pa 3110 1820}%
\special{pa 3110 1796}%
\special{pa 3066 1850}%
\special{fp}%
%
\special{pn 8}%
\special{pa 2600 600}%
\special{pa 3200 600}%
\special{fp}%
\special{sh 1}%
\special{pa 3200 600}%
\special{pa 3134 580}%
\special{pa 3148 600}%
\special{pa 3134 620}%
\special{pa 3200 600}%
\special{fp}%
\end{picture}%

\vspace*{7mm}

\noindent
where the maps $\cl$ are isomorphisms obtained 
by taking the corresponding cycles. Here we note that 
by comparison of dimensions, 
each of $\ud{\BC}$ on $\Gr^{\lambda}$ 
(resp., $\Gr^{\lambda_{1},\,\lambda_{2}}$) can be replaced by 
$\IC_{\lambda} [ - \dim \Gr^{\lambda} ]$ 
(resp., $\IC_{\lambda_{1},\,\lambda_{2}}
[-\dim \Gr^{\lambda_{1},\,\lambda_{2}}]$).
Therefore, inverting the isomorphism \eqref{eq:cohom}
in the commutative diagram above yields 
the desired commutative diagram.
\end{proof}

\begin{rem}
In the theorem above, the proper direct image 
$\bR^{\bullet}\mul_{!}\IC_{\lambda_{1},\,\lambda_{2}}$ and 
the direct image $\bR^{\bullet} \mul_{*} \IC_{\lambda_{1},\,\lambda_{2}}$ 
are isomorphic since $\mul$ is a projective map 
when restricted to $\ol{\Gr^{\lambda_{1}}} \twp \ol{\Gr^{\lambda_{2}}}$.
\end{rem}

\begin{proof}[(Sketch of) Proof of Theorem~\ref{thm:BG}]
Note that $\Phi_{\lambda_{2}}(P_{2}) \cap S_{\nu_{2}}$ is 
$U(\CO)$-stable by Lemma~\ref{lem:stable}. It follows that 
\begin{equation*}
\dim \bigl(
 (\Phi_{\lambda_{1}}(P_{1}) \cap S_{\nu_{1}})^{\sim} 
 \times^{U(\CO)} 
 (\Phi_{\lambda_{2}}(P_{2}) \cap S_{\nu_{2}})
\bigr)
= \dim \Phi_{\lambda_{1}} (P_{1}) + \dim \Phi_{\lambda_{2}} (P_{2})
\end{equation*}
from the isomorphism (induced by the map $\mul$)
\begin{equation*}
(\Phi_{\lambda_{1}}(P_{1}) \cap S_{\nu_{1}}) \times 
(\Phi_{\lambda_{2}}(P_{2}) \cap S_{\nu_{2}})
\stackrel{\sim}{\rightarrow} 
(\Phi_{\lambda_{1}}(P_{1}) \cap S_{\nu_{1}})^{\sim}
\times^{U(\CO)} 
(\Phi_{\lambda_{2}}(P_{2}) \cap S_{\nu_{2}}). 
\end{equation*}
Therefore, using  Lemma~\ref{lem:dimint} and the statement 
preceding Definition~\ref{dfn:MVcycle}, we deduce that 
\begin{equation*}
\ol{ 
 (\Phi_{\lambda_{1}}(P_{1}) \cap S_{\nu_{1}})^{\sim} 
 \times^{U(\CO)} 
 (\Phi_{\lambda_{2}}(P_{2}) \cap S_{\nu_{2}})
} 
\in 
\Irr\left(
  \ol{\mul^{-1} (S_{\nu_{1}+\nu_{2}}) \cap 
  \Gr^{\lambda_{1},\,\lambda_{2}}} \right)
\end{equation*}
by comparison of dimensions.
Thus, by defining as in \eqref{eq:Phi-ll}, we obtain 
the desired map $\Phi_{\lambda_{1},\,\lambda_{2}}$ 
with properties (ii) and (iii). Moreover, 
properties (ii) and (iii) uniquely determine this map.

In the case that $G$ has semisimple rank $1$, we see that 
the map $\Phi_{\lambda_{1},\,\lambda_{2}}$ is 
indeed the canonical bijection that respects 
the BFG crystal structure by comparing 
Theorem~\ref{thm:BGc} with \cite[\S5.4]{BrGa}; 
note the (unusual) convention on the tensor product rule for 
crystals in \cite{BrGa}.
Also, it is straightforward to see that 
the commutative diagram of Theorem~\ref{thm:BGc} is compatible with 
the restriction to $P_{j}(\CK)$-orbits for each $j \in I$. 
Since a crystal structure is determined uniquely 
by weights and the behavior of Kashiwara operators for all simple roots, 
in view of \cite[\S5.2]{BrGa}, 
we conclude that the map $\Phi_{\lambda_{1},\,\lambda_{2}}$ is 
the canonical bijection that respects 
the BFG crystal structure for a general reductive $G$. 
Hence, by taking into account the commutative diagram of 
Theorem~\ref{thm:BGc}, property (i) can also be verified, as desired.
\end{proof}


{\small
\setlength{\baselineskip}{13pt}
\renewcommand{\refname}{References}

}

\end{document}